\documentclass[11pt]{article}
\usepackage{a4}
\usepackage{amsthm}
\usepackage{amsfonts}
\usepackage{amssymb}
\usepackage{amsmath}
\usepackage{cite}
\usepackage{epsfig}

\usepackage[margin= 2 cm, bottom=17mm,footskip=10mm,top=17mm]{geometry}
\newtheorem{theorem}{Theorem}[section]

\newtheorem{lemma}[theorem]{Lemma}
\newtheorem{proposition}[theorem]{Proposition}

\DeclareMathOperator{\ex}{ex}
\newcommand\NN{{\mathbb N}}

\newcommand\PP{{\cal P}}
\newcommand\optlink{OPTLINK}
\newcommand\optineq{OPTINEQ}

\DeclareTextCompositeCommand{\v}{OT1}{l}{l\nobreak\hspace{-.1em}'}
\DeclareTextCompositeCommand{\v}{OT1}{t}{t\nobreak\hspace{-.1em}'\nobreak\hspace{-.15em}}

\begin{document}
\title{Uniform Tur\'an density of cycles\thanks{The work of the third and fourth authors has received funding from the European Research Council (ERC) under the European Union's Horizon 2020 research and innovation programme (grant agreement No 648509). This publication reflects only its authors' view; the European Research Council Executive Agency is not responsible for any use that may be made of the information it contains. The second, third and fourth authors were also supported by the MUNI Award in Science and Humanities of the Grant Agency of Masaryk University. The work of the fifth author is supported in part by the SNSF grant 200021\_196965.}}
\author{Matija Buci\'c\thanks{School of Mathematics, Institute for Advanced Study and Department of Mathematics, Princeton University, USA. Email:  {\tt matija.bucic@ias.edu}.}\and
        \newcounter{kth}
	\setcounter{kth}{2}
        Jacob W. Cooper\thanks{Faculty of Informatics, Masaryk University, 
        Botanick\'a 68A, 602 00 Brno,
        Czech Republic. E-mail: {\tt \{xcooper,dkral,mohr\}@fi.muni.cz}.}\and
        \newcounter{lth}
	\setcounter{lth}{3}
        Daniel Kr{\'a}\v{l}$^\fnsymbol{lth}$\and
	Samuel Mohr$^\fnsymbol{lth}$\and
  David Munh\'a Correia\thanks{Department of Mathematics, ETH Z\"urich, Switzerland. E-mail: {\tt david.munhacanascorreia@math.ethz.ch}.}}
\date{} 
\maketitle

\begin{abstract}
In the early 1980s, Erd\H{o}s and S\'os initiated the study of the classical Tur\'an problem with a uniformity condition:
the uniform Tur\'an density of a hypergraph $H$ is the infimum over all $d$ for which
any sufficiently large hypergraph with the property that all its linear-size subhyperghraphs have density at least $d$ contains $H$.
In particular, they raise the questions of determining the uniform Tur\'an densities of $K_4^{(3)-}$ and $K_4^{(3)}$.
The former question was solved only recently in [Israel J. Math. 211 (2016), 349--366] and [J. Eur. Math. Soc. 20 (2018), 1139--1159],
while the latter still remains open for almost 40 years.
In addition to $K_4^{(3)-}$, the only $3$-uniform hypergraphs whose uniform Tur\'an density is known
are those with zero uniform Tur\'an density classified by Reiher, R\"odl and Schacht~[J. London Math. Soc. 97 (2018), 77--97] and
a specific family with uniform Tur\'an density equal to $1/27$.

We develop new tools for embedding hypergraphs in host hypergraphs with positive uniform density and
apply them to completely determine the uniform Tur\'an density of a fundamental family of $3$-uniform hypergraphs,
namely tight cycles $C_\ell^{(3)}$.
The uniform Tur\'an density of $C_\ell^{(3)}$, $\ell\ge 5$, is equal to $4/27$
if $\ell$ is not divisible by three, and is equal to zero otherwise.
The case $\ell=5$ resolves a problem suggested by Reiher.
\end{abstract}

\section{Introduction}
\label{sec:intro}

One of the most central questions in extremal graph theory
is to determine the maximum number of edges in a graph avoiding a given graph $F$.
Formally, the \emph{Tur\'an number $\ex(n,F)$} of a (hyper)graph $F$
is the maximum number of edges in an $n$-vertex (hyper)graph that does not contain $F$ as a sub(hyper)graph.
The limiting behavior of Tur\'an numbers is captured by the concept of \emph{Tur\'an density} of a $k$-uniform hypergraph $F$,
which is defined as the limit of $\ex(n,F)/\binom{n}{k}$ (the existence of the limit is a classical result~\cite{KatNS64} from 1964),
i.e., the Tur\'an density is the maximum possible edge density of a large $F$-free hypergraph.
An equivalent, also frequently used, definition of Tur\'an density of $F$ is the infimum over all $d$ such that
any sufficiently large hypergraph with density at least $d$ must contain $F$.

Tur\'an numbers and Tur\'an densities are well-understood for graphs, i.e., $2$-uniform hypergraphs.
Indeed, the classical results of Mantel~\cite{Man07} and Tur\'an~\cite{Tur41}
give the Tur\'an number of complete graphs exactly, and
Erd\H os and Stone~\cite{ErdS46} determined
the Tur\'an density of any $r$-chromatic graph to be equal to $\frac{r-2}{r-1}$, see also~\cite{ErdS66}.
However, the analogous questions for hypergraphs are among the most challenging problems in extremal graph theory:
Erd\H os~\cite{Erd81} offered \$1\,000 for determining
the Tur\'an density of all complete $k$-uniform hypergraphs for $k\ge 3$ and
\$500 for determining the Tur\'an density of any single complete $k$-uniform hypergraph (with at least $k+1$ vertices).
However, even the smallest and likely the simplest case of
determining the Tur\'an density of the complete $3$-uniform hypergraph $K_4^{(3)}$ with four vertices has shrugged off attempts,
see in particular~\cite{FraF84,ChuL99}, on its resolution since its formulation by Tur\'an 80 years ago~\cite{Tur41} and
despite recent partial results~\cite{BabT11,Raz10} obtained using the flag algebra method of Razborov~\cite{Raz07}.
For further details, we refer the reader to the recent survey by Keevash~\cite{Kee11} or the survey by Sidorenko~\cite{Sid95}.

The Tur\'an density of the complete $3$-uniform hypergraph $K_4^{(3)}$ is conjectured to be equal to $5/9$ and
the following construction witnesses this as a lower bound:
let $V_1$, $V_2$ and $V_3$ be a balanced partition of $n$ vertices and
consider the $3$-uniform hypergraph that contains all edges $e$ such that
either $|e\cap V_i|=1$ for each $i=1,2,3$, or
$|e\cap V_i|=2$ and $|e\cap V_{i+1}|=1$ for some $i=1,2,3$ (the subscripts are modulo $3$).
Note that the edges in the construction are distributed in a highly uneven way,
in particular, each of the sets $V_i$ spans no edges.
This led Erd\H os and S\'os~\cite{ErdS82,Erd90} to propose studying Tur\'an problems
with an additional requirement that the edges of the host hypergraph are distributed uniformly.
This requirement is captured by the notion of the \emph{$\varepsilon$-linear} density of an $n$-vertex hypergraph $H$,
which is defined as the minimum edge density of an induced subhypergraph of $H$ with at least $\varepsilon n$ vertices.
Then, the \emph{uniform Tur\'an density} of a hypergraph $F$ is defined as the supremum over all $d$ such that
for all $\varepsilon>0$, there exists an arbitrarily large $F$-free hypergraph $H$ with the $\varepsilon$-linear density $d$.
So, unlike in the case of Tur\'an problems,
uniform Tur\'an problems require extremal hypergraphs to possess a minimum edge density on all linear-sized vertex subsets.
Similarly to the non-uniform case,
the uniform Tur\'an density of $F$ can also be defined as
the infimum over all $d$ for which there exists $\varepsilon>0$ such that
every sufficiently large hypergraph $H$ with the property that
every $\varepsilon$-fraction of its vertices induce a subhypergraph with edge density at least $d$ contains $F$.

Along with introducing the concept of uniform Tur\'an density, Erd\H os and S\'os also asked to determine the uniform Tur\'an density of the complete $3$-uniform hypergraph $K_4^{(3)}$ and
the $3$-uniform hypergraph $K_4^{(3)-}$, which is the hypergraph $K_4^{(3)}$ with an edge removed.
As in the case of Tur\'an density,
determining the uniform Tur\'an density of a given $3$-uniform hypergraph is very challenging.
The uniform Tur\'an density of $K_4^{(3)-}$ has only recently been shown to be equal to $1/4$
by Glebov, Volec and the third author~\cite{GleKV16} and by Reiher, R\"odl and Schacht~\cite{ReiRS18a}.
Despite this, determining the uniform Tur\'an density of $K_4^{(3)}$ remains a challenging open problem 
though it is believed that a 35-year-old construction of R\"odl~\cite{Rod86} showing that
the uniform Tur\'an density of $K_4^{(3)}$ is at least $1/2$ is optimal~\cite{Rei20}.
We refer the reader to the survey by Reiher~\cite{Rei20} for a more comprehensive treatment and further results,
which also include~\cite{ReiRS16,ReiRS18b,ReiRS18c} on a stronger notion of the uniform density of a hypergraph.

Reiher, R\"odl and Schacht~\cite{ReiRS18}
characterized $3$-uniform hypergraphs with uniform Tur\'an density equal to zero,
however, there are very few additional $3$-uniform hypergraphs whose uniform Tur\'an density has been determined:
in addition to $K_4^{(3)-}$,
there is only a specific family of $3$-uniform hypergraphs with uniform Tur\'an density equal to $1/27$~\cite{GarKL21}.
A fundamental family of (hyper)graphs of particular importance in extremal graph theory, in particular in relation to Tur\'an-type questions,
is that of cycles, see e.g.~\cite{HaxLPRRS09,RRS11,Ver16}.
There are several different important notions of cycles when considering hypergraphs, perhaps the most studied of these is that of tight cycles.
A tight $3$-uniform cycle of length $\ell\ge 5$, denoted by $C_{\ell}^{(3)}$,
is a hypergraph with $\ell$ vertices that
can be ordered in such a way that
the edges are precisely the triples of cyclically consecutive vertices.
We remark that for $\ell=4$,
one obtains the complete $3$-uniform hypergraph $K_4^{(3)}$, and
so the Tur\'an problems for tight cycles present a natural sparse generalization of problems involving $K_4^{(3)}$
both in the classical and the uniform settings.

In his survey, Reiher~\cite[Example 2.5]{Rei20} considered the uniform Tur\'an density of the tight $3$-uniform cycle $C_5^{(3)}$.
Specifically, he presented a construction showing that
the uniform Tur\'an density of the tight $3$-uniform cycle on $5$ vertices $C_5^{(3)}$
is at least $4/27\approx 0.1481$ and noted the absence of any interesting upper bound.
Balogh, Clemen and Lidick\'y~\cite{BalCL21} have recently established an upper bound of $0.402$ using the flag algebra method.
In this paper, we determine the uniform Tur\'an density of $C_5^{(3)}$ and more generally of all tight $3$-uniform cycles.

\begin{theorem}
\label{thm:cycles-main}
Let $\ell\ge 5$ be a positive integer.
The uniform Tur\'an density of the tight $3$-uniform cycle $C_\ell^{(3)}$
\begin{itemize}
\item is equal to $4/27$ if $\ell$ is not divisible by three, and
\item is equal to $0$ otherwise.
\end{itemize}
\end{theorem}

The core of Theorem~\ref{thm:cycles-main} is the first case, i.e., the case of cycles with length non-divisible by three,
while the second case follows
from the characterization of $3$-uniform hypergraphs with zero uniform Tur\'an density by Reiher et al.~\cite{ReiRS18} (this case also follows since tight cyles of length divisible by three are tripartite so even their usual Tur\'an density is zero).
In order to determine the uniform Tur\'an density of tight cycles,
we had to develop a number of tools for embedding hypergraphs in uniformly dense host hypergraphs,
which we present, in particular, in Section~\ref{sec:specific}, and
which extend tools contained (sometimes implicitly) in~\cite{GarKL21,ReiRS18}.
This machinery allows one to reduce the problem of determining the uniform Tur\'an density of a hypergraph to an optimization problem,
which can then be solved using techniques from mathematical optimization.
We expect that this approach could be useful in determining the uniform Tur\'an density of a number of other hypergraphs,
in particular, of those with uniform Tur\'an density (much) larger than $1/27$.

We find it interesting to compare
the Tur\'an problems for $3$-uniform tight cycles in the classical and uniform settings.
It is easy to see that the Tur\'an density of $C_{\ell}^{(3)}$ is $0$
if and only if the length $\ell$ of the cycle is divisible by three.
On the other hand, determining the Tur\'an density of the tight $3$-uniform cycle $C_5^{(3)}$ of length five is open and
conjectured to be equal to the lower bound of $2\sqrt{3}-3$ due to Mubayi and R\"odl \cite{MubR02}.
The best upper bound comes from the flag algebra method of Razborov~\cite{Raz10} and
is only about 1\% larger than the lower bound.
Tur\'an type problems for tight cycles have a long history,
dating back to an old question of S\'os (see~\cite{MubPS}), and
appear in relation to seemingly unrelated topics~\cite{Con10}.
We refer the reader particularly to the report~\cite{MubPS}
from the 2011 American Institute of Mathematics (AIM) workshop ``Hypergraph Tur\'an Problem'',
which contains many such problems, and to~\cite{HuaM19,TomS21,Jan21,Let21} for some recent specific results.

The paper is organized as follows.
In Section~\ref{sec:prelim}, we fix notation used throughout the paper,
discuss the case of tight $3$-uniform cycle $C_\ell^{(3)}$ of length divisible by three and
present a lower bound on the uniform Tur\'an density of tight $3$-uniform cycle $C_\ell^{(3)}$ of length not divisible by three.
In Section~\ref{sec:partitioned},
we review a framework presented by Reiher in the survey~\cite{Rei20} that
encapsulates hypergraph regularity arguments, and
present some general tools to work with this framework in Section~\ref{sec:intersection},
which we apply in Section~\ref{sec:specific} to develop methods for embeddings hypergraphs in dense host hypergraphs.
In Section~\ref{sec:optim}, we present an optimization problem whose solution is essentially equivalent
to determining the uniform Tur\'an density of a tight $3$-uniform cycle $C_\ell^{(3)}$ of length $\ell$ not divisible by three.
We then combine the results of Sections~\ref{sec:specific} and~\ref{sec:optim}
to determine the uniform Tur\'an density of tight $3$-uniform cycles in Section~\ref{sec:main},
where we present our main results.

\section{Notation and preliminary results}
\label{sec:prelim}

In this section, we fix the notation used throughout the paper and
present preliminary results on the uniform Tur\'an density of cycles.
In particular,
we will verify that
the uniform Tur\'an density of the tight $3$-uniform cycle of length divisible by three is equal to zero, and
the uniform Tur\'an density of the tight $3$-uniform cycle of other lengths is at least $4/27$.

We write $[n]$ for the set of the first $n$ positive integers, i.e., $[n]:=\{1,\ldots,n\}$.
The Ramsey number $R_k^m(n)$ is the smallest $N$ such that
every $k$-edge-colored $m$-uniform complete hypergraph with $N$ vertices
contains a monochromatic copy of the $m$-uniform complete hypergraph with $n$ vertices.

To verify that the uniform Tur\'an density of the tight $3$-uniform cycle of length divisible by three is equal to zero,
we recall the characterization of $3$-uniform hypergraphs with zero uniform Tur\'an density
obtained by Reiher, R\"odl and Schacht~\cite{ReiRS18}.

\begin{theorem}[Reiher, R\"odl and Schacht~\cite{ReiRS18}]
\label{thm:RRS}
Let $H$ be a $3$-uniform hypergraph. The following two statements are equivalent.
\begin{itemize}
\item The uniform Tur\'an density of $H$ is equal to zero.
\item The vertices of $H$ can be ordered as $v_1,\ldots,v_n$ and
      there exists a complete graph $G$ on the vertex set $\{1,\ldots,n\}$ with edges colored red, green and blue such that
      if $\{v_i,v_j,v_k\}$, $i<j<k$, is an edge of $H$,
      then the edge $ij$ is red, the edge $ik$ is green and the edge $jk$ is blue.
\end{itemize}
\end{theorem}

Using Theorem~\ref{thm:RRS}, we obtain the following.

\begin{proposition}
\label{prop:cycle3}
The uniform Tur\'an density of a tight $3$-uniform cycle $C_\ell^{(3)}$ of length $\ell$ divisible by three
is equal to zero.
\end{proposition}

\begin{proof}
Fix $\ell=3m$ to be a positive integer divisible by three, and
let us denote by $w_1,\ldots,w_{3m}$ the vertices of a cycle $C_{3m}^{(3)}$, ordered as they appear on the cycle.
We will show that the cycle $C_{3m}^{(3)}$ has the second property given in Theorem~\ref{thm:RRS}.
To do so, we set the vertices $v_1,\ldots,v_{3m}$ as in Theorem~\ref{thm:RRS} to be
\[w_1,w_4,\ldots,w_{3m-2},w_2,w_5,\ldots,w_{3m-1},w_3,w_6,\ldots,w_{3m},\]
and consider the complete graph $G$ on the vertex set $\{1,\ldots,3m\}$ with edges colored as follows:
an edge formed by $i<j$ is colored
red if $1\le i\le m$ and $m+1\le j\le 2m$,
green if $1\le i\le m$ and $2m+1\le j\le 3m$, and
blue if $m+1\le i\le 2m$ and $2m+1\le j\le 3m$;
other edges of $G$ are colored arbitrarily.
Finally, observe that if $\{v_i,v_j,v_k\}$ is an edge of the cycle $C_{3m}^{(3)}$ with $i<j<k$,
then 
it holds that $1\le i\le m$, $m+1\le j\le 2m$ and $2m+1\le k\le 3m$, and
so the vertices $v_1,\ldots,v_{3m}$ and the edge-colored complete graph $G$ witness that
$C_{3m}^{(3)}$ has the second property given in Theorem~\ref{thm:RRS}.
\end{proof}

The lower bounds for uniform Tur\'an densities are based on the probabilistic framework presented in~\cite[Section 2]{Rei20};
the framework is inspired and unifies earlier probabilistic constructions,
in particular the one from~\cite{Rod86}.
We summarize this framework in the next proposition.

\begin{proposition}
\label{prop:framework}
Let $F$ be a $3$-uniform hypergraph.
Suppose that there exists $k\in\NN$ and a set $\PP\subseteq [k]^3$ with the following properties:
for every $n\in\NN$ and every $\varphi:\binom{[n]}{2}\to [k]$,
the hypergraph $H$ with vertex set $[n]$ such that
$a<b<c$, $a,b,c\in [n]$, form an edge
iff $\left(\varphi(\{a,b\}),\varphi(\{b,c\}),\varphi(\{a,c\})\right)\in\PP$
is $F$-free.
Then, the uniform Tur\'an density of $F$ is at least $|\PP|/k^3$.
\end{proposition}

This allows us to prove that the uniform Tur\'an density of the tight $3$-uniform cycle of length not divisible by three
is at least $4/27$; the proof extends the argument for $\ell=5$ given in~\cite[Example 2.5]{Rei20}.

\begin{proposition}
\label{prop:lower}
The uniform Tur\'an density of a tight $3$-uniform cycle $C_\ell^{(3)}$ of length not divisible by three is at least $4/27$.
\end{proposition}

\begin{proof}
Fix $\ell\ge 5$ that is not divisible by three and
let $v_1,\ldots,v_\ell$ be the vertices of the tight $3$-uniform cycle $C_\ell^{(3)}$ listed in the cyclic order,
i.e., $v_i$, $v_{i+1}$ and $v_{i+2}$ form an edge for every $i\in [\ell]$ (all subscripts are modulo $\ell$ throughout the proof).

We will apply Proposition~\ref{prop:framework}.
Set $k=3$ and $\PP=\{(1,3,1),(1,3,2),(2,3,1),(2,3,2)\}$.
We now show that for every $n\in\NN$ and every $\varphi:\binom{[n]}{2}\to [k]$,
the hypergraph $H$ as defined in Proposition~\ref{prop:framework} is $C_\ell^{(3)}$-free.
Clearly, it is enough to show this for $n=\ell$.

Suppose that there exists such a hypergraph $H$ that contains $C_\ell^{(3)}$, and
let $\varphi:\binom{[\ell]}{2}\to\{1,2,3\}$ be the function from the definition of $H$ and
$f:\{v_1,\ldots,v_\ell\}\to [\ell]$ the function mapping the vertices of $C_\ell^{(3)}$ to those of $H$.
To simplify our notation, we will write $\varphi(v_iv_j)$ for $\varphi(\{f(v_i),f(v_j)\})$.
We will show that for every $i\in [\ell]$,
exactly one of the values $\varphi(v_iv_{i+1})$,
$\varphi(v_{i+1}v_{i+2})$ and
$\varphi(v_{i+2}v_{i+3})$ is equal to $3$.

If none of the three values is equal to $3$,
then $\varphi(v_iv_{i+2})=3$ and $\varphi(v_{i+1}v_{i+3})=3$; hence $f(v_{i+1})$ must be smaller than $f(v_i)$ and $f(v_{i+2})$ (otherwise, $f(v_i)$, $f(v_{i+1})$ and $f(v_{i+2})$ would not form an edge in $H$) and, similarly,
$f(v_{i+2})$ must be smaller than $f(v_{i+1})$ and $f(v_{i+3})$.
However, this is impossible as the values of $f(v_{i+1})$ and $f(v_{i+2})$ cannot both be smaller than the other.

If two or more of the three values are equal to $3$,
it must hold that $\varphi(v_iv_{i+1})=3$,
$\varphi(v_{i+1}v_{i+2})\not=3$ and
$\varphi(v_{i+2}v_{i+3})=3$.
Hence, $f(v_{i+2})$ is smaller than $f(v_i)$ and $f(v_{i+1})$ (otherwise, $f(v_i)$, $f(v_{i+1})$ and $f(v_{i+2})$ would not form an edge in $H$) and, similarly,
$f(v_{i+1})$ is smaller than $f(v_{i+2})$ and $f(v_{i+3})$.
Again, this is impossible.

We have shown that exactly one of the values $\varphi(v_iv_{i+1})$,
$\varphi(v_{i+1}v_{i+2})$ and
$\varphi(v_{i+2}v_{i+3})$ is equal to $3$ for every $i\in [\ell]$,
which is impossible as $\ell$ is not divisible by three.
It follows that the uniform Tur\'an density of $C_\ell^{(3)}$ is at least $|\PP|/k^3=4/27$.
\end{proof}

\section{Partitioned hypergraphs}
\label{sec:partitioned}

In this section, we introduce the notion of a partitioned hypergraph,
which encapsulates hypergraph regularity arguments related to the uniform Tur\'an density of $3$-uniform hypergraphs.
The notion is based on reduced hypergraphs presented in the survey by Reiher~\cite{Rei20}.

An \emph{$n$-partitioned hypergraph $H$} is a $3$-uniform hypergraph such that
its vertex set is partitioned to sets $V_{ij}$, $1\le i<j\le n$, and
every edge $e$ of $H$ satisfies that there exist indices $1\le i<j<k\le n$ such that
one vertex of $e$ belongs to $V_{ij}$, one to $V_{ik}$ and one to $V_{jk}$.
The set of all edges of $H$ that
have exactly one vertex in each of $V_{ij}$, $V_{ik}$ and $V_{jk}$ is called an \emph{$(i,j,k)$-triad}.
We refer to vertices that belong to $V_{ij}$ as \emph{left vertices} of the $(i,j,k)$-triad,
those that belong to $V_{jk}$ as \emph{right vertices} and
those that belong to $V_{ik}$ as \emph{top vertices}.
To make our notation easier to follow,
we will often use $A$, possibly with subscripts or superscripts, to denote subsets of left vertices of a triad, and
$\alpha$, again possibly with subscripts or superscripts, to denote left vertices of a triad;
similarly, we will use $B$ and $\beta$ to denote subsets of right vertices and right vertices of a triad, and
$C$ and $\gamma$ to denote subsets of top vertices and top vertices.

The \emph{density} of an $(i,j,k)$-triad is the number of edges forming the triad divided by $|V_{ij}|\cdot|V_{ik}|\cdot|V_{jk}|$.
Finally, the \emph{density} of an $n$-partitioned hypergraph $H$ is the minimum density of a triad of $H$.
We will use the following convention to simplify our notation:
if $H$ is an $n$-partitioned hypergraph, we write $V_{ij}$, $1\le i<j\le n$, for its vertex parts, and
if $H'$ is an $n'$-partitioned hypergraph, we write $V'_{ij}$, $1\le i<j\le n'$, for its vertex parts,
i.e., we use the same mathematical accents for a hypergraph as we do for its vertex parts without specifying the relation explicitly.

An $N$-partitioned hypergraph $H$ \emph{embeds} an $n$-vertex hypergraph $H_0$
if it is possible to choose distinct $1\le a_1,\ldots,a_n\le N$ corresponding to the vertices of $H_0$ and
vertices $v_{ij}\in V_{a_ia_j}$ for $1\le i<j\le n$ such that
if the $i$-th, $j$-th and $k$-th vertex of $H_0$ form an edge,
then $\{v_{ij},v_{ik},v_{jk}\}$ is an edge in the $(a_i,a_j,a_k)$-triad of $H$.

Reiher~\cite{Rei20} provided a general statement that
reduces proving an upper bound on the uniform Tur\'an density of a hypergraph $H$
to embedding $H$ in partitioned hypergraphs of the same density.

\begin{theorem}[{Reiher~\cite[Theorem 3.3]{Rei20}}]
\label{thm:reiher}
Let $H$ be a $3$-uniform hypergraph and $d\in [0,1]$.
Suppose that for every $\delta>0$ there exists $N$ such that
every $N$-partitioned hypergraph with density at least $d+\delta$ embeds $H$.
Then, the uniform Tur\'an density of $H$ is at most $d$.
\end{theorem}

In order to find such an embedding it will often be useful to pass to a, usually much smaller, subhypergraph
but in return gain additional structural information.
We will consider the notions of subhypergraphs as defined in this paragraph.
Let $H$ be an $n$-partitioned hypergraph.
We say that $H'$ is an \emph{induced subhypergraph} of $H$
if there exists $I\subseteq [n]$ such that
$H'$ is an $|I|$-partitioned hypergraph,
its vertex parts are the parts $V_{ij}$ of $H$ such that $i,j\in I$ and
$H'$ consists of all edges of $H$ which only contain vertices in the vertex parts forming $H'$.
We refer to the set $I$ as the \emph{index set} of $H'$.
In the setting as above,
we refer to $H'$ as to the subhypergraph of $H$ \emph{induced} by $I$.
A \emph{subhypergraph} of $H$ is any spanning subhypergraph of an induced subhypergraph of $H$,
i.e., a subhypergraph of $H$ is obtained from an induced subhypergraph by removing some edges.

The additional structural information we will obtain is going to, among other things,
consist of restricting degrees of vertices of triads.
We work with various notions of the degree of a vertex that we now define.
Fix $1\le i<j<k\le n$, which determines a triad of $H$.
The \emph{degree} of a vertex $v\in V_{ij}\cup V_{ik}\cup V_{jk}$ in the $(i,j,k)$-triad
is the number of edges of the triad containing $v$
divided by the product of the sizes of the two parts of the triad that do not contain $v$.
We adopt the following notation for the degree of a vertex $v$ which encapsulates the information about which triad the degree is referring to as well as the role $v$ plays in it (whether it is a left, a right or a top vertex): we write
$d_{ij\to k}(v)$ for the degree of $v\in V_{ij}$ in the $(i,j,k)$-triad, so the number of edges in the $(i,j,k)$-triad containing $v$ divided by $|V_{ik}|\cdot|V_{jk}|$,
$d_{ik\to j}(v)$ for the degree of $v\in V_{ik}$ in the $(i,j,k)$-triad, so the number of edges in the $(i,j,k)$-triad containing $v$ divided by $|V_{ij}|\cdot|V_{jk}|$, and
$d_{jk\to i}(v)$ for the degree of $v\in V_{jk}$ in the $(i,j,k)$-triad, so the number of edges of the $(i,j,k)$-triad containing $v$ divided by $|V_{ij}|\cdot|V_{ik}|$.
The \emph{codegree} of vertices $v\in V_{ij}$ and $v'\in V_{ik}$
is the number of edges of the $(i,j,k)$-triad containing $v$ and $v'$ divided by $|V_{jk}|$;
we write $d_{ij|ik}(v,v')$ for the codegree of vertices $v\in V_{ij}$ and $v'\in V_{ik}$.
In the analogous way, we define the codegree for the other two pairs of the parts of the $(i,j,k)$-triad and
write $d_{ij|jk}(v,v')$ and $d_{ik|jk}(v,v')$ for these quantities.
If $v$ is a top vertex, i.e., $v\in V_{ik}$,
a (left) vertex $v'\in V_{ij}$ is a \emph{left neighbor} of $v$ if the $(i,j,k)$-triad has an edge containing both $v$ and $v'$ and
a (right) vertex $v'\in V_{jk}$ is a \emph{right neighbor} of $v$ if the $(i,j,k)$-triad has an edge containing both $v$ and $v'$.
The \emph{left degree} of a top vertex $v\in V_{ik}$
is the number of left neighbors of $v$ divided by $|V_{ij}|$ and
the \emph{right degree} of a top vertex $v\in V_{ik}$
is the number of right neighbors of $v$ divided by $|V_{jk}|$;
the left and the right degree of $v$ are denoted by $d_{ik\to ij}(v)$ and $d_{ik\to jk}(v)$, respectively.
In the analogous way,
we define top and right neighbors of a left vertex,
top and left neighbors of a right vertex,
the top degree and the right degree of a left vertex, and
the top degree and the left degree of a right vertex;
the last four quantities are denoted $d_{ij\to ik}(v)$, $d_{ij\to jk}(v)$, $d_{jk\to ik}(v)$ and $d_{jk\to ij}(v)$, respectively.

\section{Intersection lemmas}
\label{sec:intersection}

In this section, we present tools to deal with the following scenario,
which appears repeatedly in our arguments.
Given an $N$-partitioned hypergraph $H$ there are for each triad vertices with a ``good'' property; for instance such a property might be 
left vertices of high degree.
In $H$ we would like to choose a representative vertex that has the good property with respect to \emph{every} triad it belongs to.
Since any given part of $H$ belongs to many different triads,
even if the number of vertices in the part with the good property with respect to every single triad involving the part
is linear in the size of the part,
there may be no vertex that is good for all triads involving the part in the hypergraph $H$ itself.
However, leveraging the power of Ramsey theory,
it is possible to find such a vertex when we pass to an induced subhypergraph of $H$.

The goal of this section is to establish several lemmas that can be applied in the scenario described above and
in a more general setting,
when the ``good'' property does not depend only on a single triad but on two or more triads.

\subsection{General intersection lemma}
\label{subsec:tool}

We start with presenting a general lemma (Lemma~\ref{lm:intersect}),
which directly yields most of the results in this section.
Its proof is based on the following lemma,
which in turn is based on the following result of Erd\H os~\cite{Erd64}:
every $3$-uniform $n$-vertex hypergraph with $\Omega(n^{3-1/\ell^2})$ edges
contains the complete $3$-uniform tripartite hypergraph with $\ell$ vertices in each of its parts.

\begin{lemma}
\label{lm:tripartite}
For every $\delta>0$ and $n\in\NN$, there exists $N\in\NN$ such that
for all sets $I$, $J$ and $K$ each with at least $N$ elements, every set $X$ and
all subsets $X_{ijk}\subseteq X$, $i\in I$, $j\in J$ and $k\in K$, such that $|X_{ijk}|\ge\delta |X|$,
there exist $x\in X$ and subsets $I'\subseteq I$, $J'\subseteq J$ and $K'\subseteq K$ each with at least $n$ elements such that
$x\in X_{ijk}$ for all $i\in I'$, $j\in J'$ and $k\in K'$.
\end{lemma}

\begin{proof}
Let $N$ be such that every $3$-uniform tripartite hypergraph with $N$ vertices in each of its parts and
at least $\delta N^3$ edges contains the complete $3$-uniform tripartite hypergraph with $n$ vertices in each of its parts;
the existence of such $N$ follows from the result of Erd\H os mentioned before the statement of the lemma.
Without loss of generality, we may suppose that each of the sets $I$, $J$ and $K$ have precisely $N$ elements.
Let $x$ be an element of $X$ that is contained in at least $\delta N^3$ sets $X_{ijk}$, $i\in I$, $j\in J$ and $k\in K$;
such an element $x$ exists by considering the average number of sets $X_{ijk}$ containing individual elements of $X$.
Next consider the $3$-uniform tripartite hypergraph $G$ with vertex set $I\times J\times K$ such that
$(i,j,k)\in I\times J\times K$ is an edge of $G$ if $x\in X_{ijk}$.
Since the hypergraph $G$ has at least $\delta N^3$ edges, there exist $I'\subseteq I$, $J'\subseteq J$ and $K'\subseteq K$,
each with $n$ elements, such that
every $(i,j,k)\in I'\times J'\times K'$ is an edge of $G$.
Hence, we have obtained that $x\in X_{ijk}$ for all $i\in I'$, $j\in J'$ and $k\in K'$.
\end{proof}

We derive the following from Lemma~\ref{lm:tripartite}.

\begin{lemma}
\label{lm:intersect}
For every $\delta>0$ and $n\in\NN$, there exists $N\in\NN$ such that
the following holds for every $N$-partitioned hypergraph $H$ and
every choice of subsets $W_{jij'kj''}\subset V_{ik}$, $i,k,j,j',j''\in [N]$, $j<i<j'<k<j''$, such that
$|W_{jij'kj''}|\ge\delta |V_{ik}|$.
There exists an induced $n$-partitioned subhypergraph $H'$ with index set $I\subseteq [N]$ and
there exist $w_{ik}$, $i<k$, $i,k\in I$, such that $w_{ik}\in W_{jij'kj''}$ for all $j,j',j''\in I$ satisfying $j<i<j'<k<j''$.
\end{lemma}

\begin{proof}
Let $\delta>0$ and $n\in\NN$ be given, and
let $n'$ be the value of $N$ from Lemma~\ref{lm:tripartite} for $\delta$ and $n$;
we can assume without loss of generality that $n'\ge n^2$.
We prove that the statement of the lemma holds with $N=R_2^{3n+2}(3n'+2)$.

Let an $N$-partitioned hypergraph $H$ and subsets $W_{jij'kj''}$ as in the statement of the lemma be given.
We construct an auxiliary $2$-edge-colored $(3n+2)$-uniform complete hypergraph $G$ with vertex set $[N]$ as follows:
a $(3n+2)$-tuple $i_1<\ldots<i_{3n+2}$ is colored blue
if all the $n^3$ sets $W_{ji_{n+1}j'i_{2n+2}j''}$ have a common vertex
where $j\in\{i_1,\ldots,i_n\}$, $j'\in\{i_{n+2},\ldots,i_{2n+1}\}$ and $j''\in\{i_{2n+3},\ldots,i_{3n+2}\}$;
otherwise, the $(3n+2)$-tuple is colored red.
By Ramsey's Theorem there exist $3n'+2$ indices $i_1<\ldots<i_{3n'+2}$ such that
all $(3n+2)$-tuples formed by these indices have the same color.

Set $I=\{i_1,\ldots,i_{n'}\}$, $J=\{i_{n'+2},\ldots,i_{2n'+1}\}$ and $K=\{i_{2n'+3},\ldots,i_{3n'+2}\}$ and
$X_{jj'j''}=W_{ji_{n'+1}j'i_{2n'+2}j''}$ for $j\in I$, $j'\in J$ and $j''\in K$.
By Lemma~\ref{lm:tripartite},
there exist $I'\subseteq I$, $J'\subseteq J$ and $K'\subseteq K$, each containing $n$ elements, such that
the $n^3$ sets $X_{jj'j''}$, $j\in I'$, $j'\in J$ and $j''\in K$, have a non-empty intersection.
This implies that the $(3n+2)$-tuple $\{i_{n'+1},i_{2n'+2}\}\cup I'\cup J'\cup K'$ is colored blue.
Hence, the common color of all our $(3n+2)$-tuples must be blue.

We show that
the conclusion of the lemma holds with the set $I$ consisting of $i_{k(n+1)}$ for $k\in [n]$ (note that $n(n+1)\le 3n'+2$).
Consider $k<k'$, $k,k'\in [n]$ and
a $(3n+2)$-tuple $F\subseteq\{i_1,\ldots,i_{3n'+2}\}$ such that
$F$ contains all indices $i_{n+1},\ldots,i_{n(n+1)}$,
$F$ contains exactly $n$ indices among $i_1,\ldots,i_{k(n+1)-1}$,
$F$ contains exactly $n$ indices among $i_{k(n+1)+1},\ldots,i_{k'(n+1)-1}$, and
$F$ contains exactly $n$ indices among $i_{k'(n+1)+1},\ldots,i_{3n'+2}$ (note that $3n'+2\ge n^2+2n$).
Since the color of the $(3n+2)$-tuple $F$ is blue,
there exists $w\in V_{i_{k(n+1)}i_{k'(n+1)}}$ that
is contained in all sets $W_{ji_{k(n+1)}j'i_{k'(n+1)}j''}$ with $j<i_{k(n+1)}<j'<i_{k'(n+1)}<j''$ and $j,j',j''\in F$,
in particular, for all such $j,j',j''\in I$ as $I\subseteq F$.
The conclusion of the lemma now follows.
\end{proof}

\subsection{Specific intersection lemmas}
\label{subsec:interspecific}

We now derive several corollaries of Lemma~\ref{lm:intersect} that match scenarios used in our arguments.
We remark that the lemmas in this subsection are implicitly proven in~\cite{ReiRS18} using a direct iterative approach;
alternative proofs based on Ramsey's Theorem can be found in~\cite{GarKL21}.
The first two lemmas aim to identify common vertices based on candidate sets depending on a single triad.

\begin{lemma}
\label{lm:extend0}
For every $\delta>0$ and $n\in\NN$,
there exists $N\in\NN$ such that
the following holds for every $N$-partitioned hypergraph $H$ and
every choice of subsets $C_{ijk}\subseteq V_{ik}$, $i<j<k$, $i,j,k\in [N]$, such that $|C_{ijk}|\ge\delta |V_{ik}|$.
There exists an induced $n$-partitioned subhypergraph $H'$ with index set $I\subseteq [N]$ and
there exist $\gamma_{ik}$, $i<k$, $i,k\in I$, such that $\gamma_{ik}\in C_{ijk}$ for all $j\in I$ such that $i<j<k$.
\end{lemma}

\begin{proof}
Let $\delta>0$ and $n\in\NN$ be given and apply Lemma~\ref{lm:intersect} with $n+2$ to get $N$.
Given an $N$-partitioned hypergraph $H$ and subsets $C_{ijk}$,
set $W_{j'ijkj''}$ to be $C_{ijk}$ for any $j'<i$ and $j''>k$.
By Lemma~\ref{lm:intersect}, there exists an $(n+2)$-element index set $I_0\subseteq [N]$ and
$\gamma_{ik}$, $i<k$, $i,k\in I_0$, such that $\gamma_{ik}\in W_{j'ijkj''}$ for all $j'<i<j<k<j''$ such that $j,j',j''\in I_0$.
The statement of the lemma holds with the set $I$ obtained from $I_0$ by removing its smallest and largest elements.
\end{proof}

\begin{lemma}
\label{lm:extend0a}
For every $\delta>0$ and $n\in\NN$,
there exists $N\in\NN$ such that
the following holds for every $N$-partitioned hypergraph $H$ and
every choice of subsets $A_{ijk}\subseteq V_{ij}$, $i<j<k$, $i,j,k\in [N]$, such that $|A_{ijk}|\ge\delta |V_{ij}|$.
There exists an induced $n$-partitioned subhypergraph $H'$ with index set $I\subseteq [N]$ and
there exist $\alpha_{ij}$, $i<j$, $i,j\in I$, such that $\alpha_{ij}\in A_{ijk}$ for all $k\in I$ such that $k>j$.
\end{lemma}

\begin{proof}
Let $\delta>0$ and $n\in\NN$ be given and apply Lemma~\ref{lm:intersect} with $2n$ to get $N$.
Given an $N$-partitioned hypergraph $H$ and subsets $A_{ijk}$,
set $W_{j'ij''jk}$ to be $A_{ijk}$ for $j'<i<j''<j$.
By Lemma~\ref{lm:intersect}, there exist a $2n$-element index set $I_0\subseteq [N]$ and
$\alpha_{ij}$, $i<j$, $i,j\in I_0$, such that $\alpha_{ij}\in W_{j'ij''jk}$ for all $j',j'',k\in I_0$ satisfying $j'<i<j''<j<k$.
The statement of the lemma holds with the set $I$ containing the second, fourth, sixth, etc. elements of $I_0$.
\end{proof}

The following lemma can be obtained by Lemma~\ref{lm:extend0a} by reversing the ordering of the index set.

\begin{lemma}
\label{lm:extend0b}
For every $\delta>0$ and $n\in\NN$,
there exists $N\in\NN$ such that
the following holds for every $N$-partitioned hypergraph $H$ and
every choice of subsets $B_{ijk}\subseteq V_{jk}$, $i<j<k$, $i,j,k\in [N]$, such that $|B_{ijk}|\ge\delta |V_{jk}|$.
There exists an induced $n$-partitioned subhypergraph $H'$ with index set $I\subseteq [N]$ and
there exist $\beta_{jk}$, $j<k$, $j,k\in I$, such that $\beta_{jk}\in B_{ijk}$ for all $i\in I$ such that $i<j$.
\end{lemma}

The next lemma follows from Lemmas~\ref{lm:extend0a} and~\ref{lm:extend0b}.

\begin{lemma}
\label{lm:extend0ab}
For every $\delta>0$ and $n\in\NN$,
there exists $N\in\NN$ such that
the following holds for every $N$-partitioned hypergraph $H$ and
every choice of vertices $\gamma_{ik}\in V_{ik}$, $i<k$, $i,k\in [N]$ such that
the degree of $\gamma_{ik}$ is at least $\delta$ in every $(i,j,k)$-triad, $i<j<k$, $i,j,k\in [N]$.
There exists an induced $n$-partitioned subhypergraph $H'$ with index set $I\subseteq [N]$ and
vertices $\alpha_{ij}$ and $\beta_{ij}$, $i<j$, $i,j\in I$, such that
$\{\alpha_{ij},\beta_{jk},\gamma_{ik}\}$ is an edge in the $(i,j,k)$-triad for all $i<j<k$, $i,j,k\in I$.
\end{lemma}

\begin{proof}
Apply Lemma~\ref{lm:extend0a} with $\delta/2$ and $n$ to get $n'$ and
then apply Lemma~\ref{lm:extend0b} with $\delta/2$ and $n'$ to get $N$.
Consider an $N$-partitioned hypergraph $H$ and vertices $\gamma_{ik}\in V_{ik}$ such that
the degree of $\gamma_{ik}$ is at least $\delta$ in every $(i,j,k)$-triad, $i<j<k$, $i,j,k\in [N]$.

Define $A_{ijk}$ for $i<j<k$, $i,j,k\in [N]$, to be the set of vertices $w\in V_{ij}$ such that
$d_{ij|ik}(w,\gamma_{ik})\ge\delta/2$.
Since the degree of $\gamma_{ik}$ is at least $\delta$ in every $(i,j,k)$-triad, $i<j<k$, $i,j,k\in [N]$,
the set $A_{ijk}$ contains at least $\delta |V_{ij}|/2$ vertices for $i<j<k$, $i,j,k\in [N]$ (otherwise,
the degree of $\gamma_{ik}$ in the $(i,j,k)$-triad would be smaller than $\delta$).
We apply Lemma~\ref{lm:extend0a} to $H$ and the sets $A_{ijk}$
to get an $n'$-partitioned subhypergraph of $H$ induced by $I'\subseteq [N]$ and
vertices $\alpha_{ij}$, $i<j$, $i,j\in I'$, such that
$\alpha_{ij}\in A_{ijk}$ for all $i<j<k$, $i,j,k\in I'$.
In particular, it holds that $d_{ij|ik}(\alpha_{ij},\gamma_{ik})\ge\delta/2$ for all $i<j<k$, $i,j,k\in I'$.

We next define $B_{ijk}$ for $i<j<k$, $i,j,k\in I'$, to be the set of vertices $w\in V_{jk}$ such that
$\{\alpha_{ij},w,\gamma_{ik}\}$ is an edge (in the $(i,j,k)$-triad);
note that $|B_{ijk}|\ge\delta|V_{jk}|/2$ for all $i<j<k$, $i,j,k\in I'$.
We apply Lemma~\ref{lm:extend0b} to the $n'$-partitioned subhypergraph of $H$ induced by $I'$ and the sets $B_{ijk}$
to get an $n$-partitioned subhypergraph $H'$ of $H$ induced by $I\subseteq I'$ and vertices $\beta_{ij}$, $i<j$, $i,j\in I'$, such that
$\{\alpha_{ij},\beta_{jk},\gamma_{ik}\}$ is an edge in the $(i,j,k)$-triad for all $i<j<k$, $i,j,k\in I$.
\end{proof}

The next two lemmas identify common vertices based on candidate sets depending on a pair of intersecting triads.

\begin{lemma}
\label{lm:extendKST2}
For every $\delta>0$ and $n\in\NN$,
there exists $N\in\NN$ such that
the following holds for every $N$-partitioned hypergraph $H$ and
every choice of subsets $X_{ijk\ell}\subseteq V_{jk}$, $i<j<k<\ell$, $i,j,k,\ell\in [N]$, such that $|X_{ijk\ell}|\ge\delta |V_{jk}|$.
There exist an induced $n$-partitioned subhypergraph $H'$ with index set $I\subseteq [N]$ and
there exist $\omega_{jk}$, $j<k$, $j,k\in I$, such that $\omega_{jk}\in X_{ijk\ell}$ for all $i,\ell\in I$ such that $i<j<k<\ell$.
\end{lemma}

\begin{proof}
Let $\delta>0$ and $n\in\NN$ be given and apply Lemma~\ref{lm:intersect} with $2n-1$ to get $N$.
Given an $N$-partitioned hypergraph $H$ and
subsets $X_{ijk\ell}\subseteq V_{jk}$, set $W_{ijj'k\ell}$ to be $X_{ijk\ell}$ for $j<j'<k$.
By Lemma~\ref{lm:intersect}, there exist a $(2n-1)$-element index set $I_0\subseteq [N]$ and
$\omega_{jk}$, $j<k$, $j,k\in I_0$, such that $\omega_{jk}\in W_{ijj'k\ell}$ for all $i<j<j'<k<\ell$ that $i,j',\ell\in I_0$.
The statement of the lemma holds with the set $I$ containing the first, third, fifth, etc. element of $I_0$.
\end{proof}

\begin{lemma}
\label{lm:extendKST1}
For every $\delta>0$ and $n\in\NN$,
there exists $N\in\NN$ such that
the following holds for every $N$-partitioned hypergraph $H$ and
every choice of subsets $C_{ijk\ell}\subseteq V_{ik}$, $i<j<k<\ell$, $i,j,k,\ell\in [N]$, such that $|C_{ijk\ell}|\ge\delta |V_{ik}|$.
There exists an induced $n$-partitioned subhypergraph $H'$ with index set $I\subseteq [N]$ and
there exist $\gamma_{ik}$, $i<k$, $i,k\in I$, such that $\gamma_{ik}\in C_{ijk\ell}$ for all $j,\ell\in I$ such that $i<j<k<\ell$.
\end{lemma}

\begin{proof}
Let $\delta>0$ and $n\in\NN$ be given and apply Lemma~\ref{lm:intersect} with $n+1$ to get $N$.
Given an $N$-partitioned hypergraph $H$ and
subsets $C_{ijk\ell}\subseteq V_{ik}$, set $W_{j'ijk\ell}$ to be $C_{ijk\ell}$ for all $j'<i$.
By Lemma~\ref{lm:intersect}, there exist a $(n+1)$-element index set $I_0\subseteq [N]$ and
$\gamma_{ik}$, $i<k$, $i,k\in I_0$ such that $\gamma_{ik}\in W_{j'ijk\ell}$ for all $j'<i<j<k<\ell$ that $j',j,\ell\in I_0$.
The statement of the lemma holds with the set $I$ obtained from $I_0$ by removing its first element.
\end{proof}

The following lemma is obtained by Lemma~\ref{lm:extendKST1rev} by reversing the ordering of the index set.

\begin{lemma}
\label{lm:extendKST1rev}
For every $\delta>0$ and $n\in\NN$,
there exists $N\in\NN$ such that
the following holds for every $N$-partitioned hypergraph $H$ and
every choice of subsets $C_{ijk\ell}\subseteq V_{j\ell}$, $i<j<k<\ell$, $i,j,k,\ell\in [N]$, such that $|C_{ijk\ell}|\ge\delta |V_{j\ell}|$.
There exists an induced $n$-partitioned subhypergraph $H'$ with index set $I\subseteq [N]$ and
there exist $\gamma_{j\ell}$, $j<\ell$, $j,\ell\in I$, such that $\gamma_{j\ell}\in C_{ijk\ell}$ for all $i,k\in I$ such that $i<j<k<\ell$.
\end{lemma}

\section{Structural results}
\label{sec:specific}

In this section, we present lemmas that
guarantee the existence of various structures in partitioned hypergraphs with positive density and
in particular with density larger than $4/27$.

\subsection{Top vertices with large degree}
\label{subsec:maxdegree}

The first lemma is designed to select universal top vertices with degree close to the maximum degree of a top vertex.

\begin{lemma}
\label{lm:gamma}
For every $\varepsilon>0$ and $n\in\NN$,
there exists $N\in\NN$ such that
every $N$-partitioned hypergraph $H$ with density $d$
has an $n$-partitioned subhypergraph $H_0$ of $H$ with density at least $d-\varepsilon$ with the following property:
there exist vertices $\gamma_{ik}\in V_{ik}$ such that for every $j$, $i<j<k$,
the vertex $\gamma_{ik}$ is a top vertex of the $(i,j,k)$-triad, and
the degree of each top vertex of the $(i,j,k)$-triad is at most that of $\gamma_{ik}$.
\end{lemma}

\begin{proof}
We set $N$ to be the value from Lemma~\ref{lm:extend0} applied for $n$ and $\delta=\varepsilon$.
Let $H$ be an $N$-partitioned hypergraph with density $d$.
For $1\le i<j<k\le N$, define $C_{ijk}$ to be the set containing
the $\left\lceil\varepsilon V_{ik}\right\rceil$ vertices of $V_{ik}$ with largest degree in the $(i,j,k)$-triad.
We now apply Lemma~\ref{lm:extend0} with the sets $C_{ijk}$
to get an induced $n$-partitioned subhypergraph $H'$ of $H$ with index set $I$ and
vertices $\gamma_{ik}\in V_{ik}$, $i,k\in I$, such that $\gamma_{ik}\in C_{ijk}$ for every $i<j<k$, $j\in I$.
Next remove from each $(i,j,k)$-triad of $H'$ all edges containing a vertex of $C_{ijk}$ different from $\gamma_{ik}$, and
let $H_0$ be the resulting $n$-partitioned subhypergraph of $H'$ and so of $H$.
Observe that at most $\varepsilon |V_{ij}|\cdot |V_{jk}|\cdot |V_{ik}|$ edges can be removed from the $(i,j,k)$-triad
as $|C_{ijk}\setminus\{\gamma_{ik}\}|\le\varepsilon |V_{ik}|$,
which implies that the density of $H_0$ is at least $d-\varepsilon$.
It follows that the $n$-partitioned hypergraph $H_0$ and the vertices $\gamma_{ik}$ have the properties given in the lemma.
\end{proof}

The next lemma is designed to identify left and right neighbors of selected vertices that
have the maximum degree in corresponding triads.

\begin{lemma}
\label{lm:gamma12}
For every $\varepsilon>0$ and $n\in\NN$,
there exist $N\in\NN$ such that
the following holds for every $N$-partitioned hypergraph $H$ with density $d$ and
every choice of vertices $\gamma_{ik}\in V_{ik}$ with degree at least $\varepsilon$ in each $(i,j,k)$-triad with $i<j<k$.

There exists an $n$-partitioned subhypergraph $H_0$ with index set $I\subseteq [N]$ that has density at least $d-\varepsilon$ and
there exist $\alpha^*_{ij},\beta^*_{ij},\gamma^1_{ij},\gamma^2_{ij}\in V_{ij}$, $i<j$, $i,j\in I$, such that
the following holds for all $i<j<k$, $i,j,k\in I$:
\begin{itemize}
\item the degree of $\gamma_{ik}$ in the $(i,j,k)$-triad in $H_0$ is smaller by at most $\varepsilon$ compared to $H$,
\item the vertices $\gamma^1_{ij}$, $\beta^*_{jk}$ and $\gamma_{ik}$ form an edge in the $(i,j,k)$-triad in $H_0$,
\item the vertices $\alpha^*_{ij}$, $\gamma^2_{jk}$ and $\gamma_{ik}$ form an edge in the $(i,j,k)$-triad in $H_0$,
\item for every $\ell\in I$, $i<\ell<j$,
      among left neighbors of $\gamma_{ik}$ in the $(i,j,k)$-triad, $\gamma^1_{ij}$ has the largest degree in the $(i,\ell,j)$-triad of $H_0$,
\item for every $\ell\in I$, $j<\ell<k$,
      among right neighbors of $\gamma_{ik}$ in the $(i,j,k)$-triad, $\gamma^2_{jk}$ has the largest degree in the $(j,\ell,k)$-triad of $H_0$.
\end{itemize}
\end{lemma}

\begin{proof}
We determine $N$ as follows:
first apply Lemma~\ref{lm:extend0b} with $n$ and $\varepsilon/6$ to get $N_1$,
then Lemma~\ref{lm:extend0a} with $N_1$ and $\varepsilon/6$ to get $N_2$, and
finally Lemma~\ref{lm:extendKST1rev} with $N_2$ and $\varepsilon/(3N_2)$ to get $N_3$.
We then set $N_4=R_3^4(N_3)$ and apply Lemma~\ref{lm:extendKST1} with $N_4$ and $\varepsilon/(3N_4)$ to get $N_5$, and
finally set $N=R_2^4(N_5)$.

Let $H$ be an $N$-partitioned hypergraph and
let $\gamma_{ik}\in V_{ik}$ be vertices with properties described in the statement of the lemma.
From every $(i,j,k)$-triad of $H$, $1\le i<j<k\le N$,
remove all edges containing the vertex $\gamma_{ik}$ and a vertex $w\in V_{ij}$ such that $d_{ij|ik}(w,\gamma_{ik})<\varepsilon/6$ and
all edges containing the vertex $\gamma_{ik}$ and a vertex $w\in V_{jk}$ such that $d_{jk|ik}(w,\gamma_{ik})<\varepsilon/6$ (we keep on removing edges as long as there exists a vertex $w$ in $V_{ij}$ or $V_{jk}$ with given codegree properties).
Let $H'$ be the resulting $N$-partitioned subhypergraph.
Since at most $2\varepsilon |V_{ij}|\cdot |V_{jk}|/6$ edges are removed from an $(i,j,k)$-triad,
it follows that the degree of $\gamma_{ik}$ in an $(i,j,k)$-triad in $H'$ is at least $\varepsilon-2\varepsilon/6=2\varepsilon/3$;
in particular, the density of $H'$ is at least $d-\varepsilon/3$.
Since the degree of $\gamma_{ik}$ in an $(i,j,k)$-triad is at least $2\varepsilon/3$,
the vertex $\gamma_{ik}$ has at least $2\varepsilon |V_{ij}|/3$ left neighbors in the $(i,j,k)$-triad and
at least $2\varepsilon |V_{jk}|/3$ right neighbors.
Note that if $w$ is a left neighbor of $\gamma_{ik}$ in the $(i,j,k)$-triad of $H'$, then $d_{ij|ik}(w,\gamma_{ik})\ge\varepsilon/6$, and
if $w$ is a right neighbor of $\gamma_{ik}$ in the $(i,j,k)$-triad, then $d_{jk|ik}(w,\gamma_{ik})\ge\varepsilon/6$.

For $1\le i<j<k<\ell\le N$,
define $C_{ijk\ell}$ to be the set of the $\left\lceil\varepsilon |V_{ik}|/(3N_4)\right\rceil$ left neighbors of $\gamma_{i\ell}$ (in the $(i,k,\ell)$-triad) with largest degree in the $(i,j,k)$-triad.
We next define an auxiliary $2$-edge-colored $4$-uniform complete hypergraph $G$ with vertex set $[N]$ such that
an edge formed by $i<j<k<\ell$ is colored red if $\gamma_{ik}\in C_{ijk\ell}$, and it is colored blue otherwise.
By Ramsey's Theorem, there exists an $N_5$-element set $I_5\subseteq [N]$ such that
all quadruples formed by the elements of $I_5$ have the same color.

If the common color of the quadruples of $I_5$ is red,
we choose $I_4$ to be any $N_4$-element subset of $I_5$,
we set $H_4$ to be the $N_4$-partitioned subhypergraph of $H'$ induced by $I_4$ and
we set $\gamma^1_{ik}=\gamma_{ik}$ for $i,k\in I_4$;
note that $\gamma_{ik}\in C_{ijk\ell}$ for all $i<j<k<\ell$, $i,j,k,\ell\in I_4$.
If the common color of the quadruples of $I_5$ is blue,
we apply Lemma~\ref{lm:extendKST1} with the $N_5$-partitioned subhypergraph of $H'$ induced by $I_5$ and $C_{ijk\ell}$
to get an $N_4$-partitioned induced subhypergraph $H_4$ of $H'$ with index set $I_4\subseteq I_5$ and $\gamma^1_{ik}$ such that
$\gamma^1_{ik}\in C_{ijk\ell}$ for all $i<j<k<\ell$, $i,j,k,\ell\in I_4$.
Note that $\gamma_{ik}\not\in C_{ijk\ell}$ for all $i<j<k<\ell$, $i,j,k,\ell\in I_4$ if the common color is blue.

For each triad $(i,j,k)$, with $i,j,k \in I_4$ we will now remove some edges in order to ensure the fourth condition of the lemma is satisfied for our choice of $\gamma^1_{ij}$.  We remove all edges from the $(i,j,k)$-triad that contain a vertex $w\in C_{ijk\ell}$ with degree larger than that of $\gamma^1_{ik}$ for some $\ell>k$, $\ell\in I_4$.
As there are at most $N_4$ choices of $\ell$, the density of the $(i,j,k)$-triad decreases by at most $\varepsilon/3$. 
Observe also that we do not remove any edges of the $(i,j,k)$-triad containing $\gamma_{ik}$ 
since either $\gamma^1_{ik}=\gamma_{ik}$ or $\gamma_{ik}\notin C_{ijk\ell}$ for any $\ell>k$, $\ell\in I_4$.
Let $H'_4$ be the $N_4$-partitioned hypergraph that
is obtained from $H_4$ by performing this removal for each triad indexed by $I_4$.
Observe that since we preserve all edges having $\gamma_{ik}$ as a top vertex, with $i,k \in I_4$,
it holds for any any $i<j<k<\ell$ that
the vertex $\gamma^1_{ik}$ has maximum degree in the $(i,j,k)$-triad of $H'_4$
among all left neighbors of $\gamma_{i\ell}$ (in the $(i,k,\ell)$-triad).
Also observe that the density of $H_4'$ is at least $d-2\varepsilon/3$.

We now repeat the process to find vertices $\gamma^2_{ik}$.
We start by defining sets $C_{ijk\ell}\subseteq V_{j\ell}$
to contain the $\left\lceil\varepsilon |V_{j\ell}|/(3N_2)\right\rceil$ right neighbors of $\gamma_{i\ell}$ (in the $(i,j,\ell)$-triad) with largest degree in the $(j,k,\ell)$-triad, and
an auxiliary $3$-edge-colored $4$-uniform complete hypergraph $G'$ with $N_4$ vertices corresponding to the parts of $H'_4$ such that
an edge formed by $i<j<k<\ell$ is colored red if $\gamma_{j\ell}\in C_{ijk\ell}$,
orange if $\gamma^1_{j\ell}\in C_{ijk\ell}$, and blue otherwise;
in case that $C_{ijk\ell}$ contains both $\gamma_{j\ell}$ and $\gamma^1_{j\ell}$,
we color the edge red if the degree of $\gamma_{j\ell}$ in the $(j,k,\ell)$-triad is at least the degree of $\gamma^1_{j\ell}$, and
we color the edge orange if the degree of $\gamma^1_{j\ell}$ is larger.
By Ramsey's Theorem, there exists an $N_3$-element set $I_3\subseteq I_4$ such that
all quadruples formed by the elements of $I_3$ have the same color.

If the common color of the quadruples of $I_3$ is red or orange,
we set $I_2$ to be any $N_2$-element subset of $I_3$ and
$H'_2$ to be the $N_2$-partitioned subhypergraph of $H'_4$ induced by $I_2$.
For all $i,k\in I_2$,
we set $\gamma^2_{ik}$ to $\gamma_{ik}$ if the color is red and
to $\gamma^1_{ik}$ if the color is orange.
If the common color of the quadruples of $I_3$ is blue,
we apply Lemma~\ref{lm:extendKST1rev} with the $N_3$-partitioned subhypergraph of $H'_4$ induced by $I_3$ and $C_{ijk\ell}$
to get an $N_2$-partitioned induced subhypergraph $H'_2$ of $H'_4$ with index set $I_2\subseteq I_3$ and $\gamma^2_{j\ell}$ such that
$\gamma^2_{j\ell}\in C_{ijk\ell}$ for all $i<j<k<\ell$, $i,j,k,\ell\in I$.

We now define the $N_2$-partitioned hypergraph $H_2$.
The hypergraph $H_2$ is obtained from $H'_2$ by removing from each $(j,k,\ell)$-triad all edges that
contain a vertex $w\in C_{ijk\ell}$ with degree larger than that of $\gamma^2_{j\ell}$ for some $i<j$, $i\in I_2$.
Since there are at most $N_2$ choices of an index $i$, the density of the $(j,k,\ell)$-triad decreases by at most $\varepsilon/3$ and is in particular at least $d-\varepsilon$.
Moreover, this operation does not affect edges containing $\gamma_{j\ell}$ or $\gamma^1_{j\ell}$
by the construction of the edge-coloring of $G'$ and the choice of $\gamma^2_{j\ell}$.
Observe that all edges containing $\gamma_{j\ell}$ or $\gamma^1_{j\ell}$ as top vertices are preserved by this procedure as
if the degree of $\gamma^2_{j\ell}$ in the $(j,k,\ell)$-triad is smaller than that of $\gamma_{j\ell}$,
then $\gamma_{j\ell}\not\in C_{ijk\ell}$ for any $i<j$, $i\in I_2$, and
similarly, if the degree of $\gamma^2_{j\ell}$ in the $(j,k,\ell)$-triad is smaller than that of $\gamma^1_{j\ell}$,
then $\gamma^1_{j\ell}\not\in C_{ijk\ell}$ for any $i<j$, $i\in I_2$. This means we preserve the fourth condition and establish the fifth one.

Next observe that for every $i<j<k$, $i,j,k\in I_2$,
the edges of the $(i,j,k)$-triad containing the vertex $\gamma_{ij}$ are the same in $H_2$ and in $H'$;
in particular,
it holds that $d_{ij|ik}(\gamma^1_{ij},\gamma_{ik})\ge\varepsilon/6$ and $d_{jk|ik}(\gamma^2_{jk},\gamma_{ik})\ge\varepsilon/6$.
For $i<j<k$, $i,j,k\in I_2$, define $A_{ijk}$ to be the set of vertices in $V_{ij}$ that
form an edge together with $\gamma^2_{jk}$ and $\gamma_{ik}$ in the $(i,j,k)$-triad, and
$B_{ijk}$ to be the set of vertices in $V_{jk}$ that
form an edge together with $\gamma^1_{ij}$ and $\gamma_{ik}$.
By Lemma~\ref{lm:extend0a} applied with the $N_2$-partitioned hypergraph $H_2$ and the sets $A_{ijk}$,
there exist an $N_1$-partitioned subhypergraph $H_1$ of $H_2$ induced by $I_1\subseteq I_2$ and
vertices $\alpha^*_{ij}$, $i<j$, $i,j\in I_1$, such that
$\{\alpha^*_{ij},\gamma^2_{jk},\gamma_{ik}\}$ is an edge for all $i<j<k$, $i,j,k\in I_1$.
By Lemma~\ref{lm:extend0b} applied with the $N_1$-partitioned hypergraph $H_1$ and the sets $B_{ijk}$,
there exist an $n$-partitioned subhypergraph $H_0$ of $H_1$ induced by $I\subseteq I_1$ and
vertices $\beta^*_{jk}$, $j<k$, $j,k\in I$, such that
$\{\gamma^1_{ij},\beta^*_{jk},\gamma_{ik}\}$ is an edge for all $i<j<k$, $i,j,k\in I$.

We conclude that $H_0$ is an $n$-partitioned subhypergraph of $H$ with density at least $d-\varepsilon$,
the degree of $\gamma_{ik}$ in an $(i,j,k)$-triad has dropped by at most $\varepsilon/3$, and
$H_0$ together with the vertices $\alpha^*_{ij}$, $\beta^*_{ij}$, $\gamma^1_{ij}$ and $\gamma^2_{ij}$, $i<j$, $i,j\in I$,
satisfy all the conditions of the lemma.
\end{proof}

\subsection{Neighbors of selected vertices}
\label{subsec:AB}

In this subsection we establish several lemmas concerning partitioned hypergraphs such that each of the lemmas have two outcomes:
either we find a small collection of particularly well-structured edges, which will ultimately allow us to embed a tight cycle, or
we pass to a subhypergraph by removing a small number of edges in such a way that
the resulting subhypergraph does not contain certain subconfigurations.
In the following subsection,
we deduce that if neither of the lemmas provides the structure for embedding a tight cycle,
the final subhypergraph cannot have density larger than $4/27-\varepsilon$;
this would contradict that the density of the original partitioned hypergraph was at least $4/27$.

We begin with an auxiliary lemma that is used to prove the lemmas given later in the subsection.

\begin{lemma}
\label{lm:Acover}
For every $\varepsilon>0$ and $n\in\NN$,
there exist $N\in\NN$ and $m\le 2/\varepsilon$ such that
the following holds for every $N$-partitioned hypergraph $H$ and
every choice of vertices $\gamma_{ik}\in V_{ik}$, $i,k\in [N]$.
There exists an $n$-partitioned subhypergraph $H_0$ with index set $I\subseteq [N]$ and
there exist (not necessarily distinct) vertices $\alpha_{ij}^1,\ldots,\alpha_{ij}^m\in V_{ij}$, $i<j$, $i,j\in I$, such that
the following holds in every $(i,j,k)$-triad with $i<j<k$, $i,j,k\in I$:
\begin{itemize}
\item the degree of $\gamma_{ik}$ in $H_0$
      is smaller by at most $\varepsilon$ compared to its degree in $H$,
\item the degree of any vertex of $V_{ik}$ different from $\gamma_{ik}$ in $H_0$ is the same as in $H$, and
\item for every right neighbor $\beta$ of $\gamma_{ik}$, there exists $\ell\in [m]$ such that
      $\{\alpha_{ij}^\ell,\beta,\gamma_{ik}\}$ is an edge.
\end{itemize}
\end{lemma}

\begin{proof}
Set $m=\lfloor 2/\varepsilon\rfloor$, and define iteratively $n_0,\ldots,n_m$ as follows.
We set $n_0=n$ and for $\ell\in [m]$,
define $N_\ell$ to be the value of $N$ obtained by applying Lemma~\ref{lm:extend0a} with $n_{\ell-1}$ and $\varepsilon/2$ and
set $n_\ell=R_2^3(N_\ell)$.
We will show that the statement of the lemma holds with $N=n_m$.

Let $H$ be an $N$-partitioned hypergraph and
let $\gamma_{ik}\in V_{ik}$ be vertices as described in the statement of the lemma.
For $\ell=m,\ldots,0$,
we will iteratively define $n_{\ell}$-partitioned subhypergraphs $H_{\ell}$ with index set $I_\ell$ and
vertices $\alpha_{ij}^{\ell+1},\ldots,\alpha_{ij}^m\in V_{ij}$, $i<j$, $i,j\in I_{\ell}$, such that
for every $i<j<k$, $i,j,k,\in I_{\ell}$,
there are at least $\varepsilon(m-\ell)|V_{jk}|/2$ vertices of $V_{jk}$ that form an edge with with the vertex $\gamma_{ik}$ and
one of the vertices $\alpha_{ij}^{\ell+1},\ldots,\alpha_{ij}^m$ in the $(i,j,k)$-triad of $H_{\ell}$.

At the beginning of the process, i.e., for $\ell=m$, we set $H_m=H$ and $I_m=[N]$.
Suppose that we have already constructed $H_{\ell}$, $I_\ell$ and $\alpha_{ij}^{\ell+1},\ldots,\alpha_{ij}^m$.
For all $i<j<k$, $i,j,k\in I_{\ell}$,
let $B_{ijk}\subseteq V_{jk}$ consist of those vertices of $V_{jk}$ that
do not form an edge together with $\gamma_{ik}$ and any of $\alpha_{ij}^{\ell+1},\ldots,\alpha_{ij}^m$, and
let $A_{ijk} \subseteq V_{ij}$ consist of all vertices of $V_{ij}$ that 
form an edge with $\gamma_{ik}$ and at least $\varepsilon|V_{jk}|/2$ vertices of $B_{ijk}$.
Informally speaking, $B_{ijk}\subseteq V_{jk}$ consist of the vertices of $V_{jk}$ not yet ``covered'' and
$A_{ijk} \subseteq V_{ij}$ consist of all vertices that cover many uncovered vertices.
We next construct an auxiliary $2$-edge-colored $3$-uniform complete hypergraph $G_{\ell}$ with $n_\ell$ vertices $I_\ell$:
an edge $\{i,j,k\}$ of $G_{\ell}$, $i<j<k$, is colored red if $|A_{ijk}|<\varepsilon|V_{ij}|/2$ and blue otherwise.
If $\ell>0$, the choice of $n_\ell$ yields that
there exists an $N_{\ell}$-element subset $I'_{\ell}\subseteq I_\ell$ such that
all triples formed by the elements of $I'_{\ell}$ have the same color.
If $\ell=0$, then all edges of $H_0$ are actually colored red as
at least $(1-\varepsilon/2)|V_{jk}|$ vertices in $V_{jk}$
form an edge with $\gamma_{ik}$ and one of the vertices $\alpha_{ij}^1,\ldots,\alpha_{ij}^m$,
which yields that $A_{ijk}=\emptyset$; so, we set $I'_0=I_0$.

If the common color of the triples of $I'_{\ell}$ is blue, which happens only if $\ell\ge 1$ (as explained above),
we apply Lemma~\ref{lm:extend0a} with the $N_\ell$-partitioned subhypergraph of $H_{\ell}$ induced by $I'_{\ell}$ and sets $A_{ijk}$, $i<j<k$, $i,j,k\in I'_{\ell}$ to obtain
an $n_{\ell-1}$-partitioned induced subhypergraph $H_{\ell-1}$ of $H_{\ell}$ with index set $I_{\ell-1}$ and
$\alpha^{\ell}_{ij}\in A_{ijk}$ for all $i<j<k$, $i,j,k\in I_{\ell-1}$.
Observe that $H_{\ell-1}$, $I_{\ell-1}$ and $\alpha_{ij}^{\ell},\ldots,\alpha_{ij}^m$
satisfy the properties given in the description of the iterative procedure, and
so we may proceed to the next step of the iterative procedure.

If the common color of the triples of $I'_{\ell}$ is red,
we stop the process and
consider the $n$-partitioned hypergraph induced by any $n$-element subset $I$ of $I'_{\ell}$ and
remove from each $(i,j,k)$-triad, $i<j<k$, $i,j,k\in I$,
all edges containing $\gamma_{ik}$ and a vertex from $B_{ijk}$.
We set $H_0$ to be the resulting $n$-partitioned subhypergraph of $H$ and
all $\alpha^1_{ij},\ldots,\alpha^{\ell}_{ij}$ to $\alpha^{\ell+1}_{ij}$ (if $\ell=m$, choose these vertices arbitrarily).

We now argue that $H_0$ and $\alpha^1_{ij},\ldots,\alpha^m_{ij}$ satisfy the conclusion of the lemma.
For any $i<j<k$, $i,j,k\in I\subseteq I'_{\ell}$ (recall that $\ell$ is the iteration number when we stopped the process),
since the triple formed by $i$, $j$ and $k$ is red,
the set $A_{ijk}$ contains at most $\varepsilon|V_{ij}|/2$ vertices.
Because each of the vertices of $A_{ijk}$ is contained in at most $|V_{jk}|$ edges together with the vertex $\gamma_{ik}$,
we have removed at most $\varepsilon|V_{ij}|\cdot|V_{jk}|/2$ edges containing a vertex from $A_{ijk}$.
Since $A_{ijk}$ consists precisely of the vertices in $V_{ij}$ that are contained in at least $\varepsilon |V_{jk}|/2$ removed edges,
at most $\varepsilon|V_{ij}|\cdot|V_{jk}|/2$ edges containing a vertex from $V_{ij} \setminus A_{ijk}$ have been removed.
It follows that the degree of $\gamma_{ik}$ in each triad of $H_0$ is smaller by at most $\varepsilon$ compared to its degree in $H$.
In addition, the definition of $B_{ijk}$ and
the fact that the vertex $\gamma_{ik}$ has no right neighbor in $B_{ijk}$ in $H_0$ imply that
every right neighbor $w$ of $\gamma_{ik}$ in an $(i,j,k)$-triad of $H_0$ forms an edge
with $\gamma_{ik}$ and at least one of the vertices $\alpha^{\ell+1}_{ij},\ldots,\alpha^m_{ij}$.
Hence,
the $n$-partitioned subhypergraph $H_0$ of $H$ and the vertices $\alpha^1_{ij},\ldots,\alpha^m_{ij}$
satisfy the conclusion of the lemma.
\end{proof}

The next lemma is symmetric to Lemma~\ref{lm:Acover} and
can be proven by applying Lemma~\ref{lm:Acover} to a partitioned hypergraph with the order of parts reversed.

\begin{lemma}
\label{lm:Bcover}
For every $\varepsilon>0$ and $n\in\NN$,
there exist $N\in\NN$ and $m\le 2/\varepsilon$ such that
the following holds for every $N$-partitioned hypergraph $H$ and
every choice of vertices $\gamma_{ik}\in V_{ik}$, $i,k\in [N]$.
There exists an $n$-partitioned subhypergraph $H_0$ with index set $I\subseteq [N]$ and
there exist (not necessarily distinct) vertices $\beta_{ij}^1,\ldots,\beta_{ij}^m\in V_{ij}$, $i<j$, $i,j\in I$, such that
the following holds in every $(i,j,k)$-triad with $i<j<k$, $i,j,k\in I$:
\begin{itemize}
\item the degree of $\gamma_{ik}$ in $H_0$
      is smaller by at most $\varepsilon$ compared to its degree in $H$,
\item the degree of any vertex of $V_{ik}$ different from $\gamma_{ik}$ in $H_0$ is the same as in $H$, and
\item for every left neighbor $\alpha$ of $\gamma_{ik}$, there exists $\ell\in [m]$ such that
      $\{\alpha,\beta_{jk}^\ell,\gamma_{ik}\}$ is an edge.
\end{itemize}
\end{lemma}

Our next lemma asserts that given a choice of top vertices,
either we find a small collection of particularly well-structured edges, or
we pass to subhypergraph with slightly smaller density such that
there is no edge involving two right neighbors of the chosen top vertices.
Note that the vertices $\gamma^0_{ik}$, $\gamma_{ik}$, $\gamma^1_{ik}$ and $\gamma^2_{ik}$
in the statement of the lemma need not be distinct.

\begin{lemma}
\label{lm:AAedges}
For every $\varepsilon>0$ and $n\in\NN$,
there exists $N\in\NN$ such that
one of the following holds for every $N$-partitioned hypergraph $H$ with density $d$ and choice of vertices $\gamma^0_{ik},\gamma_{ik},\gamma^1_{ik},\gamma^2_{ik}\in V_{ik}$, $1\le i<k\le N$.
\begin{itemize}
\item There exists an induced $n$-partitioned subhypergraph $H_0$ and $\alpha_{ij},\beta_{ij},\omega_{ij},\alpha^{\omega}_{ij},\gamma^{\omega}_{ij}\in V_{ij}$ such that
      for every $1\le i<j<k\le n$, $i,j,k\in I$, 
      $\{\alpha_{ij},\beta_{jk},\gamma_{ik}^0\}$ is an edge,
      $\{\omega_{ij},\beta_{jk},\gamma^{\omega}_{ik}\}$ is an edge, and
      $\{\alpha^{\omega}_{ij},\omega_{jk},\gamma_{ik}^0\}$ is an edge.
\item There exists an $n$-partitioned subhypergraph $H_0$ with index set $I\subseteq [N]$ such that
      \begin{itemize}
      \item the density of $H_0$ is at least $d-\varepsilon$,
      \item the degrees of $\gamma_{ik}$, $\gamma^1_{ik}$ and $\gamma^2_{ik}$ in the $(i,j,k)$-triad for $i<j<k$, $i,j,k\in I$,
            are smaller by at most $\varepsilon$ compared to their degrees in the same triad of $H$, and
      \item there is no $i<j<k<\ell$, $i,j,k,\ell\in I$, such that
            the $(j,k,\ell)$-triad contains an edge that
	    contains both a right neighbor of $\gamma^0_{ik}$ in the $(i,j,k)$-triad and
	    a right neighbor of $\gamma^0_{j\ell}$ in the $(j,k,\ell)$-triad,
	    i.e., no vertex of $V_{jk}$ is both a right neighbor of $\gamma^0_{ik}$ and a left neighbor of $\gamma^0_{j\ell}$.
      \end{itemize}
\end{itemize}
\end{lemma}

\begin{proof}
Fix $\varepsilon>0$ and $n\in\NN$.
Let $m=\lceil 6/\varepsilon\rceil$ and set $n_1,\ldots,n_7$ as follows:
$n_1$ is $R_{m^2}^3(n)$,
$n_2=n_1+1$,
$n_3$ is the value of $N$ from Lemma~\ref{lm:extend0b} applied with $n_2$ and $\delta=\varepsilon/(6n)$,
$n_4$ is the value of $N$ from Lemma~\ref{lm:extendKST2} applied with $n_3$ and $\delta=\varepsilon/(6n)$,
$n_5$ is the value of $N$ from Lemma~\ref{lm:extendKST1rev} applied with $n_4$ and $\delta=\varepsilon/(3n)$,
$n_6$ is $R_5^4(n_5)$, and,
finally, $n_7$ is the value of $N$ from Lemma~\ref{lm:Acover} applied with $n_6$ and $\varepsilon/3$.
We will show that the statement of the lemma holds for $N=n_7$. 

We first apply Lemma~\ref{lm:Acover} with $\gamma^0_{ik}$ to get an $n_6$-partitioned subhypergraph $H_6$ of $H$ and
vertices $\alpha_{ij}^1,\ldots,\alpha_{ij}^{m}$ such that
the degree of each $\gamma^0_{ik}$ is smaller by at most $\varepsilon/3$ compared to $H$ and
every right neighbor of $\gamma^0_{ik}$ in an $(i,j,k)$-triad of $H_6$
forms an edge together with $\gamma^0_{ik}$ and (at least) one of the vertices $\alpha_{ij}^1,\ldots,\alpha_{ij}^{m}$.

For $i<j<k$, $i,j,k\in I_6$,
let $B_{ijk}\subseteq V_{jk}$ be the set of right neighbors of $\gamma^0_{ik}$ in the $(i,j,k)$-triad.
We next construct an auxiliary $5$-edge-colored $4$-uniform complete hypergraph $G$ with vertex set $I_6$.
For any $i<j<k<\ell$, $i,j,k,\ell\in I_6$,
let $W_{ijk\ell}$ be the set of vertices $w\in V_{j\ell}$ such that
the number of edges formed by the vertex $w$, a vertex of $B_{ijk}\subseteq V_{jk}$ and a vertex of $B_{jk\ell}\subseteq V_{k\ell}$
is at least $\varepsilon |V_{jk}||V_{k\ell}|/(3n)$.
An edge formed by $i<j<k<\ell$, $i,j,k,\ell\in I_6$ is colored (we use the first option that applies)
\begin{itemize}
\item red if $\gamma_{j\ell}\in W_{ijk\ell}$,
\item orange if $\gamma^1_{j\ell}\in W_{ijk\ell}$,
\item amber if $\gamma^2_{j\ell}\in W_{ijk\ell}$, 
\item blue if $|W_{ijk\ell}|\ge\varepsilon|V_{j\ell}|/(3n)$, and
\item green, otherwise.
\end{itemize}

By Ramsey's Theorem,
there exists an $n_5$-element subset $I_5\subseteq I_6$ such that
all edges of $G$ formed by the elements of $I_5$ have the same color.

If the common color of the edges of $G$ formed by the elements of $I_5$ is green, we proceed as follows.
Let $I$ be any $n$-element subset of $I_5$ and
let $H_0$ be the $n$-partitioned hypergraph obtained
from the $n$-partitioned subhypergraph of $H_6$ induced by $I$ by removing
for all $i<j<k<\ell$, $i,j,k,\ell\in I$,
all edges containing a vertex of $B_{ijk}$ and a vertex of $B_{jk\ell}$ from the $(j,k,\ell)$-triad.
For any $i<j<k<\ell, i,j,k,\ell\in I$,
the degree of any top vertex of the $(j,k,\ell)$-triad not contained in $W_{ijk\ell}$
is decreased by at most $\varepsilon/(3n)$ (this follows from the definition of $W_{ijk\ell}$).
In particular, for a fixed $(j,k,\ell)$-triad,
there are at most $n$ choices of $i$ and
so the degree of any such top vertex is decreased by at most $\varepsilon/3$ compared to $H_6$.
Since the common color of the edges is green,
neither of the vertices $\gamma_{j\ell}$, $\gamma^1_{j\ell}$ and $\gamma^2_{j\ell}$ is contained in $W_{ijk\ell}$ and
so the degree in the $(j,k,\ell)$-triad of each of them
is decreased by at most $\varepsilon/3+\varepsilon/3<\varepsilon$ compared to $H$.
Since the number of edges of a $(j,k,\ell)$-triad consisting of a vertex from $B_{ijk}$, $B_{jk\ell}$ and $W_{ijk\ell}$, for some $i$,
is at most $n \cdot \varepsilon/(3n)=\varepsilon/3$ (as the common color of edges of $G$ is green,
it holds that $|W_{ijk\ell}|<\varepsilon|V_{j\ell}|/(3n)$),
the density of any $(j,k,\ell)$-triad has decreased by at most $2\varepsilon/3$ compared to $H_6$ and
so by at most $\varepsilon$ compared to $H$.
It follows that the $n$-partitioned subhypergraph $H_0$ of $H$ satisfies the second conclusion of the lemma.

If the common color of the edges of $G$ formed by the elements of $I_5$ is blue,
we apply Lemma~\ref{lm:extendKST1rev} with the sets $W_{ijk\ell}$, $i,j,k,\ell\in I_5$, and
obtain an $n_4$-element subset $I_4\subseteq I_5$ and vertices $\gamma^{\omega}_{j\ell}$ such that
$\gamma^{\omega}_{j\ell}\in W_{ijk\ell}$ for all $i<j<k<\ell$, $i,j,k,\ell\in I_4$.
If the common color of the edges is red, orange or amber,
we set $I_4$ to be any $n_4$-element subset of $I_5$, and
we set $\gamma^{\omega}_{j\ell}=\gamma_{j\ell}$ if the common color is red,
$\gamma^{\omega}_{j\ell}=\gamma^1_{j\ell}$ if the common color is orange, and
$\gamma^{\omega}_{j\ell}=\gamma^2_{j\ell}$ if the common color is amber.
In each of these four cases,
it holds that $\gamma^{\omega}_{j\ell}\in W_{ijk\ell}$ for all $i<j<k<\ell$, $i,j,k,\ell\in I_4$, and
we next treat them together.
Let now $X_{ijk\ell}\subseteq B_{ijk}\subseteq V_{jk}$ be the set of vertices $w$ that
form an edge with $\gamma_{j\ell}^{\omega}$ and at least $\varepsilon|V_{k\ell}|/(6n)$ vertices of $B_{jk\ell}\subseteq V_{k\ell}$.
Observe that $|X_{ijk\ell}| \ge \varepsilon|V_{jk}|/(6n)$:
otherwise, the number of edges formed by a vertex of $B_{ijk}$, a vertex of $B_{jk\ell}$, and the vertex $\gamma_{j\ell}^{\omega}$
would be smaller than $\varepsilon|V_{jk}|/(6n) \cdot |V_{k\ell}|+|V_{jk}| \cdot \varepsilon|V_{k\ell}|/(6n)=\varepsilon|V_{jk}| \cdot |V_{k\ell}|/(3n)$ (the first term counts the edges containing a vertex from $X_{ijk\ell}$ and the latter the edges not containing a vertex from $X_{ijk\ell}$),
which is impossible since $\gamma_{j\ell}^{\omega} \in W_{ijk\ell}$.
We next apply Lemma~\ref{lm:extendKST2} to obtain an $n_3$-element subset $I_3\subseteq I_4$ and vertices $\omega_{jk}$ such that
$\omega_{jk}\in X_{ijk\ell}\subseteq B_{ijk}$ for all $i<j<k<\ell$, $i,j,k,\ell\in I_3$.
For any $j<k<\ell, j,k,\ell \in I_3$,
we set $B'_{jk\ell}\subseteq B_{jk\ell}\subseteq V_{k\ell}$ to be the set of vertices contained in $B_{jk\ell}$
that form an edge together with $\omega_{jk}$ and $\gamma_{j\ell}^{\omega}$.
Since $\omega_{jk}$ belongs to $X_{ijk\ell}$, the set $B'_{jk\ell}$ has at least $\varepsilon|V_{k\ell}|/(6n)$ elements.
Hence, we can apply Lemma~\ref{lm:extend0b} and
obtain an $n_2$-element subset $I_{2}\subseteq I_3$ and vertices $\beta_{k\ell}$ such that
$\beta_{k\ell}\in B'_{jk\ell}\subseteq B_{jk\ell}$ for all $j<k<\ell$, $j,k,\ell\in I_{2}$.
Finally, we let $I_1$ be the subset of $I_2$ obtained by removing its smallest index. 

To summarize,
if the common color of the edges of $G$ formed by the elements of $I_5$ is not green,
we have found an $n_1$-element subset $I_1\subseteq I_5$ and
vertices $\gamma_{jk}^{\omega},\omega_{jk}, \beta_{jk} \in V_{jk}$ for all $j,k\in I_1$ such that
for all $j<k<\ell$, $j,k,\ell\in I_{1}$,
both the vertices $\beta_{k\ell}$ and $\omega_{k\ell}$ are contained in $B_{jk\ell}$,
i.e., they are right neighbors of $\gamma^0_{j\ell}$, and
the $(j,k,\ell)$-triad contains an edge formed by the vertices $\omega_{jk}$, $\beta_{k\ell}$ and $\gamma^{\omega}_{j\ell}$.

Let $H_1$ be the $n_1$-partitioned subhypergraph of $H_6$ induced by $I_1$. 
We now define an auxiliary $m^2$-edge-colored $3$-uniform complete hypergraph $G'$ with vertex set $I_1$;
an edge formed by $j<k<\ell$, $j,k,\ell\in I_1$ is assigned a color $(a,b)\in [m]^2$
if the $(j,k,\ell)$-triad of $H_6$ contains the edges $\{\alpha^a_{jk}, \omega_{k\ell}, \gamma_{j\ell}^0\}$ and $\{\alpha_{jk}^b, \beta_{k\ell}, \gamma_{j\ell}^0\}$;
if there are multiple choices of $a$ and $b$, we choose one arbitrarily.
Since both $\beta_{k\ell}$ and $\omega_{k\ell}$ are contained in $B_{jk\ell}$ and
$B_{jk\ell}$ contains only right neighbors of $\gamma_{j\ell}^0$ in $H_6$,
it follows that such $a$ and $b$ always exist.
By Ramsey's Theorem, there exists an $n$-element subset $I$ of $I_1$ such that
all edges of $G'$ formed by the elements of $I$ have the same color, say $(a,b)$.
We set $H_0$ to be the $n$-partitioned subhypergraph of $H$ induced by $I$, and
$\alpha_{jk}^{\omega}=\alpha^a_{jk}$ and $\alpha_{jk}=\alpha^b_{jk}$ for all $j<k$, $j,k\in I$.
It follows that the induced $n$-partitioned subhypergraph $H_0$ of $H$ (not of $H_6$) and
the vertices $\alpha_{jk},\alpha^{\omega}_{jk},\omega_{jk},\beta_{jk},\gamma^{\omega}_{jk}\in V_{jk}$, $j<k$, $j,k\in I$,
satisfy the first conclusion of the lemma.
\end{proof}

The next lemma is symmetric to Lemma~\ref{lm:AAedges} and
can be proven by applying Lemma~\ref{lm:AAedges} to a partitioned hypergraph with the order of parts reversed.

\begin{lemma}
\label{lm:BBedges}
For every $\varepsilon>0$ and $n\in\NN$,
there exists $N\in\NN$ such that
one of the following holds for every $N$-partitioned hypergraph $H$ with density $d$ and choice of vertices $\gamma^0_{ik},\gamma_{ik},\gamma^1_{ik},\gamma^2_{ik}\in V_{ik}$, $1\le i<k\le N$.
\begin{itemize}
\item There exists an induced $n$-partitioned subhypergraph $H_0$ and $\alpha_{ij},\beta_{ij},\omega_{ij},\beta^{\omega}_{ij},\gamma^{\omega}_{ij}\in V_{ij}$ such that
      for every $1\le i<j<k\le n$, $i,j,k\in I$,
      $\{\alpha_{ij},\beta_{jk},\gamma_{ik}^0\}$ is an edge,
      $\{\alpha_{ij},\omega_{jk},\gamma^{\omega}_{ik}\}$ is an edge, and
      $\{\omega_{ij},\beta^{\omega}_{jk},\gamma_{ik}^0\}$ is an edge.
\item There exists an $n$-partitioned subhypergraph $H_0$ with index set $I\subseteq [N]$ such that
      \begin{itemize}
      \item the density of $H_0$ is at least $d-\varepsilon$,
      \item the degrees of $\gamma_{ik}$, $\gamma^1_{ik}$ and $\gamma^2_{ik}$ in the $(i,j,k)$-triad for $i<j<k$, $i,j,k\in I$,
            are smaller by at most $\varepsilon$ compared to their degrees in the same triad of $H$ and
      \item there is no $i<j<k<\ell$, $i,j,k,\ell\in I$, such that
            the $(i,j,k)$-triad contains an edge that
	    contains both a left neighbor of $\gamma^0_{ik}$ in the $(i,j,k)$-triad and
	    a left neighbor of $\gamma^0_{j\ell}$ in the $(j,k,\ell)$-triad,
	    i.e., no vertex of $V_{jk}$ is both a right neighbor of $\gamma^0_{ik}$ and a left neighbor of $\gamma^0_{j\ell}$.
      \end{itemize}
\end{itemize}
\end{lemma}

The next lemma concerns edges containing both right and left neighbors of chosen top vertices instead of two right neighbors.
The proof is similar to that of Lemma~\ref{lm:AAedges},
although the first conclusion is more complicated.

\begin{lemma}
\label{lm:BAedges}
For every $\varepsilon>0$ and $n\in\NN$,
there exists $N\in\NN$ such that
one of the following holds for every $N$-partitioned hypergraph $H$ with density $d$ and choice of vertices $\gamma^0_{ik},\gamma_{ik},\gamma^1_{ik},\gamma^2_{ik}\in V_{ik}$, $1\le i<k\le N$.
\begin{itemize}
\item There exists an induced $n$-partitioned subhypergraph $H_0$ with index set $I\subseteq [N]$ and
      vertices $\alpha_{ij},\beta_{ij},$ $\alpha'_{ij},\beta'_{ij},\gamma^{\bullet}_{ij}\in V_{ij}$, $i<j$, $i,j\in I$, and
      vertices $\beta^{\bullet}_{ijkst}\in V_{jk}$ and $\alpha^{\bullet}_{ijkst}\in V_{ks}$ such that
      the following holds for all $i<j<k<s<t$, $i,j,k,s,t\in I$,
      \begin{itemize}
      \item $\{\alpha_{ij},\beta_{jk},\gamma^0_{ik}\}$ is an edge in the $(i,j,k)$-triad,
      \item $\{\alpha'_{ij},\beta^{\bullet}_{ijkst},\gamma^0_{ik}\}$ is an edge in the $(i,j,k)$-triad,
      \item $\{\alpha^{\bullet}_{ijkst},\beta'_{st},\gamma^0_{kt}\}$ is an edge in the $(k,s,t)$-triad, and
      \item $\{\beta^{\bullet}_{ijkst},\alpha^{\bullet}_{ijkst},\gamma^{\bullet}_{js}\}$ is an edge in the $(j,k,s)$-triad.
      \end{itemize}
\item There exists an $n$-partitioned subhypergraph $H_0$ with index set $I\subseteq [N]$ such that
      \begin{itemize}
      \item the density of $H_0$ is at least $d-\varepsilon$,
      \item the degrees of $\gamma_{ik}$, $\gamma^1_{ik}$ and $\gamma^2_{ik}$ in the $(i,j,k)$-triad for $i<j<k$, $i,j,k\in I$,
            are smaller by at most $\varepsilon$ than their degrees in the same triad of $H$ and
      \item there is no $i<j<k<s<t$, $i,j,k,s,t\in I$, such that
            the $(j,k,s)$-triad contains an edge that
	    contains both a right neighbor of $\gamma^0_{ik}$ in the $(i,j,k)$-triad and
	    a left neighbor of $\gamma^0_{kt}$ in the $(k,s,t)$-triad.
      \end{itemize}
\end{itemize}
\end{lemma}

\begin{proof}
Fix $\varepsilon>0$ and $n\in\NN$.
Let $m=\lceil 12/\varepsilon\rceil$ and set $n_1,\ldots,n_7$ as follows:
$n_1$ is $R_{m^2}^5(n)$,
$n_2$ is the value of $N$ from Lemma~\ref{lm:intersect} applied with $n_1$ and $\delta=\varepsilon/(2n^2)$,
$n_3$ is $R_5^5(n_2)$,
$n_4$ is the value of $N$ from Lemma~\ref{lm:extend0ab} applied with $n_3$ and $2\varepsilon/3$,
$n_5$ is $R_2^3(n_4)$,
$n_6$ is the value of $N$ from Lemma~\ref{lm:Bcover} applied with $n_5$ and $\varepsilon/6$, and,
finally, $n_7$ is the value of $N$ from Lemma~\ref{lm:Acover} applied with $n_6$ and $\varepsilon/6$.
We will show that the statement of the lemma holds for $N=n_7.$ 

We first apply Lemma~\ref{lm:Acover} with $\gamma^0_{ik}$ to get an $n_6$-partitioned subhypergraph $H_6$ of $H$ and
vertices $\alpha_{ij}^1,\ldots,\alpha_{ij}^{m}$ such that
the degree of each $\gamma^0_{ik}$ is smaller by at most $\varepsilon/6$ compared to $H$ and
every right neighbor of $\gamma^0_{ik}$ in an $(i,j,k)$-triad of $H_6$
forms an edge together with $\gamma^0_{ik}$ and (at least) one of the vertices $\alpha_{ij}^1,\ldots,\alpha_{ij}^{m}$.
We next apply Lemma~\ref{lm:Bcover} with $\gamma_{ik}^0$ to get an $n_5$-partitioned subhypergraph $H_5$ of $H_6$ and
vertices $\beta_{ij}^1,\ldots,\beta_{ij}^{m}$ such that
the degree of each $\gamma^0_{ik}$ is smaller by at most $\varepsilon/6$ compared to $H_5$ and
every left neighbor of $\gamma^0_{ik}$ in an $(i,j,k)$-triad of $H_5$
forms an edge together with $\gamma^0_{ik}$ and (at least) one of the vertices $\beta_{jk}^1,\ldots,\beta_{jk}^{m}$.
Note that
the degree of each $\gamma^0_{ik}$ in $H_5$ is smaller by at most $\varepsilon/3$ compared to $H$ and
the degrees of other top vertices are the same in $H_5$ and $H$.

Let $I_5$ be the index set of the $n_5$-partitioned subhypergraph $H_5$.
We now construct an auxiliary $2$-edge-colored $3$-uniform complete hypergraph $G$ with vertex set $I_5$.
An edge formed by $i<j<k$, $i,j,k\in I_5$ is colored
\begin{itemize}
\item blue if the degree of $\gamma^0_{ik}$ in the $(i,j,k)$-triad is smaller than $2\varepsilon/3$, and
\item green, otherwise.
\end{itemize}
By Ramsey's Theorem, there exists an $n_4$-element subset $I_4$ of $I_5$ such that
all edges formed by $i<j<k$, $i,j,k\in I_4$, have the same color.
If the common color of the edges is blue, we set $I$ to be any $n$-element subset of $I_4$ and
$H_0$ to be the $n$-partitioned hypergraph obtained from the $n$-partitioned subhypergraph of $H_5$ induced by $I$
by removing all edges containing $\gamma^0_{ik}$ from each $(i,j,k)$-triad, $i<j<k$, $i,j,k\in I$. 
Since the degree of $\gamma^0_{ik}$ in the $(i,j,k)$-triad
dropped by at most an additional $2\varepsilon/3$ compared to $H_5$ for all $i<j<k$, $i,j,k\in I$,
the hypergraph $H_0$ satisfies the second conclusion of the lemma.

If the common color of the edges formed by the elements of $I_4$ is green,
we apply Lemma~\ref{lm:extend0ab} with the $n_4$-partitioned subhypergraph of $H_4$ induced by $I_4$ and
the vertices $\gamma^0_{ik}$, $i<k$, $i,k\in I_4$,
to obtain an $n_3$-partitioned subhypergraph $H_3$ of $H_4$ induced by $I_3\subseteq I_4$ and
vertices $\alpha_{ij}\in V_{ij}$ and $\beta_{ij}\in V_{ij}$ such that
$\{\alpha_{ij},\beta_{jk},\gamma^0_{ik}\}$ is an edge for all $i<j<k$, $i,j,k\in I_3$.

For $i<j<k$, $i,j,k\in I_3$,
let $A_{ijk}\subseteq V_{ij}$ be the set of left neighbors of $\gamma^0_{ik}$ in the $(i,j,k)$-triad and
$B_{ijk}\subseteq V_{jk}$ the set of right neighbors of $\gamma^0_{ik}$.
We next construct an auxiliary $5$-edge-colored $5$-uniform complete hypergraph $G'$ with vertex set $I_3$.
Let $W_{ijkst}$ for $i<j<k<s<t$, $i,j,k,s,t\in I_3$, be the set of vertices $w\in V_{js}$ such that
there is an edge formed by the vertex $w$, a vertex of $B_{ijk}$ and a vertex of $A_{kst}$.
An edge of $G'$ formed by $i<j<k<s<t$, $i,j,k,s,t\in I_3$ is colored (we use the first option that applies)
\begin{itemize}
\item red if $\gamma_{js}\in W_{ijkst}$,
\item orange if $\gamma^1_{js}\in W_{ijkst}$,
\item amber if $\gamma^2_{js}\in W_{ijkst}$,
\item blue if $|W_{ijkst}|\ge\varepsilon|V_{js}|/(2n^2)$, and
\item green, otherwise.
\end{itemize}

By Ramsey's Theorem,
there exists an $n_2$-element subset $I_2\subseteq I_3$ such that
all edges of $G'$ formed by the elements of $I_2$ have the same color.
If the common color of the edges of $G'$ formed by the elements of $I_2$ is green,
we proceed as follows.
Let $I$ be any $n$-element subset of $I_2$ and
let $H_0$ be the $n$-partitioned hypergraph obtained
from $n$-partitioned subhypergraph of $H_3$ induced by $I$ by removing
for all $i<j<k<s<t$, $i,j,k,s,t\in I$,
all edges containing a vertex of $B_{ijk}$ and a vertex of $A_{kst}$ from the $(j,k,s)$-triad.
Since the edge of $G'$ formed by $i<j<k<s<t$ is not blue,
the density of the $(j,k,s)$-triad drops by at most $\varepsilon/(2n^2)$ for each choice of $i$ and $t$.
It follows that the density of $H_0$ is smaller by at most $\varepsilon/2$ compared to the density of $H_3$ and
so by at most $\varepsilon$ compared to $H$.
Moreover, since the edge of $G'$ formed by $i<j<k<s<t$ is not red, orange or amber
the degrees of $\gamma_{js}, \gamma^1_{js}$ and $\gamma^2_{js}$ in the $(j,k,s)$-triad remain the same as they were in $H_5$.
It follows that the $n$-partitioned subhypergraph $H_0$ of $H$ satisfies the second conclusion of the lemma.

If the common color of the edges of $G'$ formed by the elements of $I_2$ is blue,
we apply Lemma~\ref{lm:intersect} with the sets $W_{ijkst}$, $i,j,k,s,t\in I_2$, and
obtain an $n_1$-element subset $I_1\subseteq I_2$ and vertices $\gamma^{\bullet}_{js}$ such that
$\gamma^{\bullet}_{js}\in W_{ijkst}$ for all $i<j<k<s<t$, $i,j,k,s,t\in I_1$.
If the common color of the edges is red, orange or amber,
we set $I_1$ to be any $n_1$-element subset of $I_2$, and
we set $\gamma^{\bullet}_{js}=\gamma_{js}$ if the common color is red,
$\gamma^{\bullet}_{js}=\gamma^1_{js}$ if the common color is orange, and
$\gamma^{\bullet}_{js}=\gamma^2_{js}$ if the common color is amber.
Note that, regardless whether the common color of the edges of $G'$
formed by the elements of $I_2$ is blue, red, orange or amber,
it holds for all $i<j<k<s<t$, $i,j,k,s,t\in I_1$, that
the vertex $\gamma^{\bullet}_{js}$ forms an edge with a vertex of $B_{ijk}$ and a vertex of $A_{kst}$.

Let $H_1$ be the $n_1$-partitioned subhypergraph of $H_3$ induced by $I_1$.
For each $i<j<k<s<t$, $i,j,k,s,t\in I_1$,
fix a vertex $\beta^{\bullet}_{ijkst}\in B_{ijk} \subseteq V_{jk}$ and
a vertex $\alpha^{\bullet}_{ijkst}\in A_{kst} \subseteq V_{ks}$ such that
$\{\beta^{\bullet}_{ijkst},\alpha^{\bullet}_{ijkst},\gamma^{\bullet}_{js}\}$ is an edge of the $(j,k,s)$-triad.
We now define an auxiliary $m^2$-edge-colored $5$-uniform complete hypergraph $G''$ with vertex set $I_1$;
an edge formed by $i<j<k<s<t$, $i,j,k,s,t\in I_1$ is given a color $(a,b)\in [m]^2$
if the $(i,j,k)$-triad of the original partitioned hypergraph $H$ (i.e., not necessarily in the ``cleaned up'' partitioned hypergraph $H_1$)
contains the edge formed by $\alpha^a_{ij}$, $\beta^{\bullet}_{ijkst}$ and $\gamma_{ik}^0$, and
the $(k,s,t)$-triad of $H$ contains the edge formed by $\alpha^{\bullet}_{ijkst}$, $\beta^b_{st}$ and $\gamma_{kt}^0$.
In case of multiple possible choices, we choose $a$ and $b$ arbitrarily.
The existence of $a\in [m]$ and $b\in [m]$ is guaranteed by the conclusions of Lemmas~\ref{lm:Acover} and~\ref{lm:Bcover},
which have been applied earlier in the proof,
since $\beta^{\bullet}_{ijkst}$ is a right neighbor of $\gamma_{ik}^0$ in the $(i,j,k)$-triad and
$\alpha^{\bullet}_{ijkst}\in A_{kst}$ is a left neighbor of $\gamma_{kt}^0$ in the $(k,s,t)$-triad.
By Ramsey's Theorem, there exists an $n$-element subset $I$ of $I_1$ such that
all edges of $G''$ formed by the elements of $I$ have the same color, say $(a,b)$.
We set $H_0$ to be the $n$-partitioned subhypergraph of $H$ induced by $I$,
$\alpha'_{ij}=\alpha^a_{ij}$ and
$\beta'_{ij}=\beta^b_{ij}$ for all $i<j$, $i,j\in I$.
It follows that the induced $n$-partitioned subhypergraph $H_0$ of $H$ and
the vertices $\alpha_{ij},\beta_{ij},\alpha_{ij}',\beta'_{ij},\gamma^{\bullet}_{ij}\in V_{ij}$, $i<j$, $i,j\in I$, and
the vertices $\beta^{\bullet}_{ijkst}\in V_{jk}$ and $\alpha^{\bullet}_{ijkst}\in V_{ks}$
satisfy the first conclusion of the lemma.
\end{proof}

The final lemma of this subsection concerns left and right neighbors of given pairs top vertices.

\begin{lemma}
\label{lm:gamma12ramsey}
For every $\varepsilon>0$,
there exist $N\in\NN$ such that
the following holds for every $N$-partitioned hypergraph $H$ and
every choice of vertices $\gamma^1_{ij}$, $\gamma^2_{ij}\in V_{ij}$, $i<j$, $i,j\in [N]$. There exists $n\in\NN$ and $k\in [n-2]$ and
an $n$-partitioned subhypergraph $H_0$ with index set $I=\{i_1,\ldots,i_{n}\}\subseteq [N]$, $i_1<\cdots<i_{n}$, such that
\begin{itemize}
\item all triads different from the $(i_k,i_{k+1},i_{k+2})$-triad are the same in $H_0$ and $H$,
\item the degrees of all top vertices in the $(i_k,i_{k+1},i_{k+2})$-triad
      are smaller by at at most $\varepsilon$ compared to their degrees in $H$,
\item every right neighbor of $\gamma^1_{i_k,i_{k+2}}$ in the $(i_k,i_{k+1},i_{k+2})$-triad
      is also a right neighbor of $\gamma^1_{i_j,i_{k+2}}$ in the $(i_j,i_{k+1},i_{k+2})$-triad for some $j<k$, and
\item every left neighbor of $\gamma^2_{i_k,i_{k+2}}$ in the $(i_k,i_{k+1},i_{k+2})$-triad
      is also a left neighbor of $\gamma^2_{i_k,i_j}$ in the $(i_k,i_{k+1},i_j)$-triad for some $j>k+2$.
\end{itemize}
\end{lemma}

\begin{proof}
Without loss of generality, we may assume that $1/\varepsilon$ is an integer and
set $N=N'+2$ where $N'=R_2^3(3+2/\varepsilon)$.

Suppose that there exist $1<a<b<c<N'+2$, $a,b,c\in [N'+2]$, such that
every right neighbor of $\gamma^1_{ac}$ in the $(a,b,c)$-triad of $H'$
with the exception of at most $\varepsilon|V'_{bc}|/2$ vertices of $V'_{bc}$
is also a right neighbor of $\gamma^1_{kc}$ in the $(k,b,c)$-triad for some $k<a$, and
every left neighbor of $\gamma^2_{ac}$ in  the $(a,b,c)$-triad of $H'$,
with the exception of at most $\varepsilon|V'_{ab}|/2$ vertices of $V'_{ab}$,
is also a left neighbor of $\gamma^2_{ak}$ in the $(a,b,k)$-triad for some $k>c$.
We set $I=\{1,\ldots,a-1,a,b,c,c+1,\ldots,N'+2\}$ and $n=|I|=N'+2-(c-a)+2$.
Let $H_0$ be the $n$-partitioned hypergraph obtained
from the $n$-partitioned subhypergraph of $H'$ induced by $I$
by removing from the $(a,b,c)$-triad all edges that contain:
\begin{itemize}
\item $\gamma^1_{ac}$ and a vertex $w\in V'_{bc}$ that
      is not a right neighbor of $\gamma^1_{kc}$ in any $(k,b,c)$-triad for $k<a$, or
\item $\gamma^2_{ac}$ and a vertex $w\in V'_{ab}$ that
      is not a left neighbor of $\gamma^2_{ak}$ in any $(a,b,k)$-triad for $k>c$.
\end{itemize}
Observe that the degree of $\gamma^1_{ac}$ has decreased by at most $\varepsilon/2$ and
the same holds for the degree of $\gamma^2_{ac}$,
unless $\gamma^1_{ac}=\gamma^2_{ac}$, in which case,
the degree of $\gamma^1_{ac}=\gamma^2_{ac}$ has decreased by at most $\varepsilon$.
Hence, the $n$-partitioned hypergraph $H_0$ with $k=a$ satisfies the conclusion of the lemma.

In the rest, we will assume for the sake of contradiction that
there is no triple $1<a<b<c<N'+2$ with the properties given in the previous paragraph.
We construct a $2$-edge-colored $3$-uniform complete hypergraph $G$ with vertex set $\{2,\ldots,N'+1\}$ as follows.
The edge formed by $2\le a<b<c\le N'+1$, is red
if at least $\varepsilon|V'_{bc}|/2$ right neighbors of $\gamma^1_{ac}$ in the $(a,b,c)$-triad
are not right neighbors of any  $\gamma^1_{kc}$ with $k<a$;
otherwise, the edge is colored blue.
Note that, by the assumption made at the beginning of this paragraph,
if the edge formed by $a<b<c$ is blue,
then there are at least $\varepsilon|V'_{ab}|/2$ left neighbors of $\gamma^2_{ac}$ in the $(a,b,c)$-triad that
are not a left neighbor of $\gamma^2_{ak}$ for any $k>c$.

By Ramsey's Theorem, there exist $a_1,\ldots,a_{3+2/\varepsilon}\in\{2,\ldots,N'+1\}$ such that
any triple formed by them has the same color in $G$.
Suppose that the color is blue and for $i=3,\ldots,3+2/\varepsilon$,
let $X_i$ be the set of left neighbors of $\gamma^2_{a_1a_i}$ in the $(a_1,a_2,a_i)$-triad that
are not a left neighbor of $\gamma^2_{a_1a_j}$ for any $j>i$.
The definition of the sets $X_3,\ldots,X_{3+2/\varepsilon}$ implies that
the sets $X_3,\ldots,X_{3+2/\varepsilon}$ are disjoint.
On the other hand, since the triple formed by $a_1$, $a_2$ and $a_i$ is colored blue,
the size of each set $X_i$, $i=3,\ldots,3+2/\varepsilon$ is at least $\varepsilon|V_{a_1a_2}|/2$,
which is impossible since the $1+2/\varepsilon$ sets $X_i$'s are disjoint subsets of $V_{a_1a_2}$.
Hence, the common color cannot be blue.
A symmetric argument yields that the common color cannot be red.
It follows that a triple $a<b<c$ with the properties given in the second paragraph always exists.
\end{proof}

\section{Neighborhoods of top vertices}
\label{sec:optim}

The core of our argument is analyzing neighborhoods of three selected top vertices in triads.
In a partitioned hypergraph, we select using Lemmas \ref{lm:gamma} and \ref{lm:gamma12}
vertices $\gamma_{ik}$, $\gamma^1_{ik}$ and $\gamma^2_{ik}$ such that
\begin{itemize}
\item the vertex $\gamma_{ik}$ is a top vertex of almost maximum degree among top vertices in all $(i,j,k)$-triads with $i<j<k$,
\item the vertex $\gamma^1_{ik}$ is a top vertex of almost maximum degree in all $(i,j,k)$-triads with $i<j<k$
      among top vertices that are left neighbors of some $\gamma_{i\ell}$ with $\ell>k$, and
\item the vertex $\gamma^1_{ik}$ is a top vertex of almost maximum degree in all $(i,j,k)$-triads with $i<j<k$
      among top vertices that are right neighbors of some $\gamma_{\ell{}k}$ with $\ell<i$.
\end{itemize}
We use Lemmas~\ref{lm:AAedges} and~\ref{lm:BBedges} to guarantee that for all $i<j<k<\ell$
every vertex of $V_{jk}$ is either a left neighbor of $\gamma_{j\ell}$, a right neighbor of $\gamma_{ik}$ or neither,
in particular, no vertex of $V_{jk}$ is both a left neighbor of $\gamma_{j\ell}$ and a right neighbor of $\gamma_{ik}$.
The same holds with respect to vertices $\gamma^1_{j\ell}$, $\gamma^1_{ik}$, $\gamma^2_{j\ell}$ and $\gamma^2_{ik}$.
It follows that each vertex of $V_{jk}$ is one of $3^3=27$ types,
depending on its relation to the three selected top vertices in triads.
The following lemma will let us assume that 
the proportion of each of these $27$ types of vertices in $V_{jk}$ is approximately the same in all $V_{jk}$.

\begin{lemma}
\label{lm:gammaABX}
For every $\varepsilon>0$ and $n\in\NN$,
there exists $N\in\NN$ such that
the following holds for every $N$-partitioned hypergraph $H$ and
every choice of vertices $\gamma_{ik},\gamma^1_{ik},\gamma^2_{ik}\in V_{ik}$, $i<k$.

For $i<j<k<\ell$,
let $A_{ijk\ell}\subseteq V_{jk}$ be the set of left neighbors of $\gamma_{j\ell}$ in the $(j,k,\ell)$-triad,
$B_{ijk\ell}\subseteq V_{jk}$ the set of right neighbors of $\gamma_{ik}$ in the $(i,j,k)$-triad, and
$X_{ijk\ell}=V_{jk}\setminus(A_{ijk\ell}\cup B_{ijk\ell})$.
We define analogously $A^1_{ijk\ell}$, $B^1_{ijk\ell}$ and $X^1_{ijk\ell}$ with respect to $\gamma^1_{j\ell}$ and $\gamma^1_{ik}$, and
$A^2_{ijk\ell}$, $B^2_{ijk\ell}$ and $X^2_{ijk\ell}$ with respect to $\gamma^2_{j\ell}$ and $\gamma^2_{ik}$.

There exists an induced $n$-partitioned subhypergraph $H_0$ with the index set $I$ and
reals $d_{\kappa\kappa^1\kappa^2}\in [0,1]$ indexed by $\kappa,\kappa^1,\kappa^2\in\{A,B,X\}$ such that the following holds 
\[\left|\;\frac{\left|\kappa_{ijk\ell}\cap (\kappa^1)^1_{ijk\ell}\cap (\kappa^2)^2_{ijk\ell}\right|}{|V_{jk}|}-d_{\kappa\kappa^1\kappa^2}\right|
  \le\varepsilon\]
for all $\kappa,\kappa^1,\kappa^2\in\{A,B,X\}$ and all $i<j<k<\ell$, $i,j,k,\ell\in I$.

In addition, if the sets $A_{ijk\ell}$ and $B_{ijk\ell}$ are disjoint,
the sets $A^1_{ijk\ell}$ and $B^1_{ijk\ell}$ are disjoint, and
the sets $A^2_{ijk\ell}$ and $B^2_{ijk\ell}$ are disjoint for all $i<j<k<\ell$,
then it holds that
\[\sum_{\kappa,\kappa^1,\kappa^2\in\{A,B,X\}} d_{\kappa\kappa^1\kappa^2} \le 1+27\varepsilon.\]
\end{lemma}

\begin{proof}
Without loss of generality, we may assume that $1/\varepsilon$ is an integer.
Set $K=\varepsilon^{-27}$ and $N=R_K^4\!\big(n\big)$.

Let $H$ be an $N$-partitioned hypergraph with properties given in the statement of the lemma.
We construct an auxiliary $K$-edge-colored $4$-uniform complete hypergraph $G$ with vertex set $[N]$ as follows:
the edge formed by $i<j<k<\ell$ is colored by a $27$-tuple $\Delta\in [\varepsilon^{-1}]^{27}$,
whose elements are indexed by $\kappa,\kappa^1,\kappa^2\in\{A,B,X\}$, such that
\[\left|\;\frac{\left|\kappa_{ijk\ell}\cap (\kappa^1)^1_{ijk\ell}\cap (\kappa^2)^2_{ijk\ell}\right|}{|V_{jk}|}-\varepsilon\Delta_{\kappa\kappa^1\kappa^2}\right|
  \le\varepsilon.\]

By Ramsey's Theorem, there exists an $n$-element subset $I\subseteq [N]$ such that
all edges of $G$ formed by vertices of $I$ have the same color, and
let $\Delta\in [\varepsilon^{-1}]^{27}$ be this color.
It follows that the statement of the lemma holds for the $n$-partitioned subhypergraph $H_0$ of $H$ induced by $I$ and
for $d_{\kappa\kappa^1\kappa^2}=\varepsilon\Delta_{\kappa\kappa^1\kappa^2}$ where $\kappa,\kappa^1,\kappa^2\in\{A,B,X\}$.

We next consider that the additional assumption holds.
If the sets $A_{ijk\ell}$ and $B_{ijk\ell}$ are disjoint,
then the sets $A_{ijk\ell}$, $B_{ijk\ell}$ and $X_{ijk\ell}$ partition $V_{jk}$.
Similarly, if $A^1_{ijk\ell}$ and $B^1_{ijk\ell}$ are disjoint,
then $A^1_{ijk\ell}$, $B^1_{ijk\ell}$ and $X^1_{ijk\ell}$ partition $V_{jk}$, and
if $A^2_{ijk\ell}$ and $B^2_{ijk\ell}$ are disjoint,
then $A^2_{ijk\ell}$, $B^2_{ijk\ell}$ and $X^2_{ijk\ell}$ partition $V_{jk}$.
This in turn implies that
the $27$ sets $\kappa_{ijk\ell}\cap (\kappa^1)^1_{ijk\ell}\cap (\kappa^2)^2_{ijk\ell}$ with $\kappa,\kappa^1,\kappa^2\in\{A,B,X\}$
partition the set $V_{jk}$.
Since it holds that
\[d_{\kappa\kappa^1\kappa^2} \le \frac{\left|\kappa_{ijk\ell}\cap (\kappa^1)^1_{ijk\ell}\cap (\kappa^2)^2_{ijk\ell}\right|}{|V_{jk}|}+\varepsilon\]
for every $\kappa,\kappa^1,\kappa^2\in\{A,B,X\}$, and
the sum of the right sides over $\kappa,\kappa^1,\kappa^2\in\{A,B,X\}$ is equal to $1+27\varepsilon$,
the bound on the sum of $d_{\kappa\kappa^1\kappa^2}$ follows.
\end{proof}

We next perform the quantitative analysis of neighborhood sizes of the three selected top vertices in triads in an idealized setting,
where all triads posses exactly the same structure.
The latter means that 
each set $V_{ij}$ has exactly $m$ elements, which we identify with $[m]$, and
if \emph{some} triad contains an edge with a left vertex $a$, a right vertex $b$ and a top vertex $c$,
then \emph{every} triad contains an edge formed by the left vertex $a$, the right vertex $b$ and the top vertex $c$.
While this assumption is not implied by Lemma~\ref{lm:gammaABX},
the analysis performed in Section~\ref{sec:main} will yield a partition subhypergraph such that
all triads have ``approximately the same'' structure,
which will be sufficient for the quantitative analysis performed in this section to apply.
The purpose of considering the idealized setting here is to motivate an optimization problem,
which we describe in the next paragraph.

We start with giving intuition behind the set-up of the problem.
Suppose that $\gamma$ is a top vertex with the maximum degree and
$\gamma_A$ is its left neighbor with the maximum degree when viewed as the top vertex in other triads, and
$\gamma_B$ is its right neighbor with the maximum degree when viewed as the top vertex in other triads.
The vertices $\gamma$, $\gamma_A$ and $\gamma_B$ are ``idealized'' vertices $\gamma_{ik}$, $\gamma^1_{ik}$ and $\gamma^2_{ik}$.
Each vertex can be a left neighbor, a right neighbor or a non-neighbor of $\gamma$ (Lemmas~\ref{lm:AAedges} and~\ref{lm:BBedges} can be used to guarantee that exactly one of these cases occur).
Similarly, each vertex is a left neighbor, a right neighbor or a non-neighbor of $\gamma_A$ and
is a left neighbor, a right neighbor or a non-neighbor of $\gamma_B$.
For each of these 27 possibilities, we set up a variable that measures the proportion of such vertices:
there will be $27$ variables $x_{\kappa\kappa_A\kappa_B}$ where $\kappa,\kappa_A,\kappa_B\in\{A,B,X\}$,
and $\kappa$ determines the relation to $\gamma$, $\kappa_A$ to $\gamma_A$ and $\kappa_B$ to $\gamma_B$;
the values of these $27$ variables are ``approximated'' by the values of $d_{\kappa\kappa_A\kappa_B}$ from Lemma~\ref{lm:gammaABX}.
For example, $x_{ABX}$ is the proportion of vertices that
are a left neighbor of $\gamma$, a right neighbor of $\gamma_A$ and 
neither a left neighbor nor a right neighbor of $\gamma_B$.

The set of top vertices of a triad
can be split into non-neighbors of $\gamma$, left neighbors of $\gamma$ and right neighbors of $\gamma$.
Hence, considering separately edges based on the type of their top vertices,
we conclude that the density of the triad can be upper bounded by the sum of
\begin{itemize}
\item the product of one minus the sum of the left and the right degrees of $\gamma$ with the degree of $\gamma$,
\item the product of the left degree of $\gamma$ and the degree of $\gamma_A$ as the top vertex, and
\item the product of the right degree of $\gamma$ and the degree of $\gamma_B$ as the top vertex.
\end{itemize}
Note that the proportion of non-neighbors of $\gamma$ is equal to $\sum_{\kappa_A\kappa_B}x_{X\kappa_A\kappa_B}$,
the proportion of left neighbors is equal to $\sum_{\kappa_A\kappa_B}x_{A\kappa_A\kappa_B}$ and
the proportion of right neighbors is equal to $\sum_{\kappa_A\kappa_B}x_{B\kappa_A\kappa_B}$.

Consider a left vertex $w$ of the type $\kappa\kappa_A\kappa_B$ and
a right vertex $w'$ of the type $\lambda\lambda_A\lambda_B$.
The vertices $w$ and $w'$ can be contained in the same edge only
if each of the pairs $(\kappa,\lambda)$, $(\kappa_A,\lambda_A)$ and $(\kappa_B,\lambda_B)$
equals one of the following six pairs:
\[(A,B),\, (A,X),\, (B,X),\, (X,A),\, (X,B) \mbox{ and } (X,X)\]
unless the first conclusion of one of  Lemmas~\ref{lm:AAedges}, \ref{lm:BBedges} or~\ref{lm:BAedges}
applies for one of the vertices $\gamma$, $\gamma_A$ and $\gamma_B$ (the lemmas
exclude the pairs $(A,A), (B,B)$ and $(B,A)$, respectively).
In addition, unless the conclusion of Lemma~\ref{lm:gamma12ramsey} applies,
it must also hold that $(\kappa_B,\lambda_A)\not=(A,B)$.
Hence, we will write $\kappa\kappa_A\kappa_B\to\lambda\lambda_A\lambda_B$
if each of the pairs $(\kappa,\lambda)$, $(\kappa_A,\lambda_A)$ and $(\kappa_B,\lambda_B)$
is one of the pairs 
\[(A,B),\, (A,X),\, (B,X),\, (X,A),\, (X,B) \mbox{ and } (X,X),\]
and it also holds that $(\kappa_B,\lambda_A)\not=(A,B)$.
So, in the considered idealized setting, it holds that
\begin{itemize} 
\item the degree of $\gamma$ is at most $\sum_{A\kappa_A\kappa_B\to B\lambda_A\lambda_B}x_{A\kappa_A \kappa_B}x_{B\lambda_A\lambda_B}$,
\item the degree of $\gamma_A$ is at most $\sum_{\kappa A\kappa_B\to \lambda B\lambda_B}x_{\kappa A\kappa_B}x_{\lambda B\lambda_B}$ and
\item the degree of $\gamma_B$ is at most $\sum_{\kappa\kappa_AA\to \lambda\lambda_AB}x_{\kappa\kappa_AA}x_{\lambda\lambda_AB}$.
\end{itemize}
In addition, 
since $\gamma$ is a top vertex of maximum degree,
the degrees of $\gamma_A$ and $\gamma_B$ are bounded by that of $\gamma$,
in particular, they are also upper bounded by the first expression above.

So, we can now bound the density of a triad by the proportion of top vertices of each of the three types multiplied
by the upper bounds on the degrees of $\gamma$, $\gamma_A$ and $\gamma_B$,
i.e., the density of a triad is at most
\begin{align*}
& \left(\sum_{\kappa_A\kappa_B}x_{X\kappa_A\kappa_B}\right)
  \sum_{A\kappa_A\kappa_B\to B\lambda_A\lambda_B}x_{A\kappa_A\kappa_B}x_{B\lambda_A\lambda_B}+\\
& \left(\sum_{\kappa_A\kappa_B}x_{A\kappa_A\kappa_B}\right)
  \min\left\{\sum_{A\kappa_A\kappa_B\to B\lambda_A\lambda_B}x_{A\kappa_A\kappa_B}x_{B\lambda_A\lambda_B},
             \sum_{\kappa A\kappa_B\to \lambda B\lambda_B}x_{\kappa A\kappa_B}x_{\lambda B\lambda_B}\right\}+\\
& \left(\sum_{\kappa_A\kappa_B}x_{B\kappa_A\kappa_B}\right)
  \min\left\{\sum_{A\kappa_A\kappa_B\to B\lambda_A\lambda_B}x_{A\kappa_A\kappa_B}x_{B\lambda_A\lambda_B},
             \sum_{\kappa\kappa_AA\to \lambda\lambda_AB}x_{\kappa\kappa_AA}x_{\lambda\lambda_AB}\right\}.
\end{align*}
The sums above should be interpreted as sums over all possible choices,
in particular, an expression $A\kappa_A\kappa_B\to B\lambda_A\lambda_B$
is the sum over all $32$ choices of $\kappa_A$, $\kappa_B$, $\lambda_A$ and $\lambda_B$
such that $A\kappa_A\kappa_B\to B\lambda_A\lambda_B$.
We now define \optlink{} to be the optimization problem to maximize the above expression
subject to the constraints that all the variables $x_{AAA},\ldots,x_{XXX}$ are non-negative and their sum is one.

The main result of this section is that the optimum value of the problem \optlink{} is $4/27$.
The lower bound follows, for example, from setting $x_{AAX}=2/3$, $x_{BBX}=1/3$ and all other variables to zero.
We will start our analysis of the problem \optlink{} by proving the next lemma,
which allows eliminating the two minima in the definition of the problem \optlink.
We point out here that while we do know that the degree of $\gamma$ is larger than that of $\gamma_A$ and $\gamma_B$,
the expressions appearing in \optlink{} are upper bounds on their degrees and
so it is possible that the first term in the minima appearing the problem can be smaller in principle.
The following lemma asserts that
it is possible to restrict to instances where the first term in the two minima is at least as large as the second;
the corresponding optimization problem with be referred to as \optineq.

\begin{lemma}
\label{lm:nomin}
The optimum value of the problem \optlink{} is equal to the maximum of 
\begin{align}
\begin{split}
& \left(\sum_{\kappa_A\kappa_B}x_{X\kappa_A\kappa_B}\right)
  \sum_{A\kappa_A\kappa_B\to B\lambda_A\lambda_B}x_{A\kappa_A\kappa_B}x_{B\lambda_A\lambda_B}+\\
& \left(\sum_{\kappa_A\kappa_B}x_{A\kappa_A\kappa_B}\right)
             \sum_{\kappa A\kappa_B\to \lambda B\lambda_B}x_{\kappa A\kappa_B}x_{\lambda B\lambda_B}+\\
& \left(\sum_{\kappa_A\kappa_B}x_{B\kappa_A\kappa_B}\right)
             \sum_{\kappa\kappa_AA\to \lambda\lambda_AB}x_{\kappa\kappa_AA}x_{\lambda\lambda_AB}
\end{split}\tag{\optineq}
\end{align}
subject to that all the variables $x_{AAA},\ldots,x_{XXX}$ are non-negative, their sum is one and
\begin{align*}
\sum_{\kappa A\kappa_B\to \lambda B\lambda_B}x_{\kappa A\kappa_B}x_{\lambda B\lambda_B}
& \le \sum_{A\kappa_A\kappa_B\to B\lambda_A\lambda_B}x_{A\kappa_A\kappa_B}x_{B\lambda_A\lambda_B}\mbox{ and}\\
\sum_{\kappa\kappa_AA\to \lambda\lambda_AB}x_{\kappa\kappa_AA}x_{\lambda\lambda_AB}
& \le \sum_{A\kappa_A\kappa_B\to B\lambda_A\lambda_B}x_{A\kappa_A\kappa_B}x_{B\lambda_A\lambda_B}\mbox{.}
\end{align*}
\end{lemma}

\begin{proof}
As any feasible solution of the problem \optineq{} is also a feasible solution of the problem \optlink,
the optimum value of the problem \optineq{} is at most that of \optlink.
In the rest, we establish the opposite inequality,
i.e., that the optimum value of the problem \optineq{} is at least that of \optlink.

Let $x_{AAA},\ldots,x_{XXX}$ be a solution of the problem \optlink{} that maximizes the objective function and
among all solutions that maximize the objective function choose the one that
maximizes
\begin{equation}
\sum_{\kappa_A\kappa_B}x_{X\kappa_A\kappa_B}+\sum_{\kappa\kappa_B}x_{\kappa{}X\kappa_B}+\sum_{\kappa\kappa_A}x_{\kappa\kappa_A{}X}.
\label{eq:weighted-xs}
\end{equation}
Note that such a solution exists by compactness.

We claim that $x_{AAA},\ldots,x_{XXX}$ is also a feasible solution of \optineq{}.
If not, then one of the two additional inequalities from the definition of the problem \optineq{} is violated.
We analyze the case that the first inequality is violated; the other case is symmetric.
Let $T$ and $S$ be the left and right sides of the inequality, i.e.,
\begin{equation*}
  S=\sum_{\kappa A\kappa_B\to \lambda B\lambda_B}x_{\kappa A\kappa_B}x_{\lambda B\lambda_B}
  \qquad\mbox{and}\qquad
  T=\sum_{A\kappa_A\kappa_B\to B\lambda_A\lambda_B}x_{A\kappa_A\kappa_B}x_{B\lambda_A\lambda_B}
  \mbox{,}
\end{equation*}  
and let $\xi=T/S$.
As we have assumed that the inequality is violated, it holds that $\xi<1$.
For all $\kappa,\kappa_B\in\{A,B,X\}$,
set 
\[x'_{\kappa A{}\kappa_B}=\xi\cdot x_{\kappa A{}\kappa_B},\quad
  x'_{\kappa B{}\kappa_B}=x_{\kappa B{}\kappa_B}\quad\mbox{and}\quad
  x'_{\kappa X{}\kappa_B}=x_{\kappa X{}\kappa_B}+(1-\xi) x_{\kappa A{}\kappa_B}.\]
We will show that $x'_{AAA},\ldots,x'_{XXX}$ is a solution of the problem \optlink{},
the value of the objective function is at least the value for $x_{AAA},\ldots,x_{XXX}$, and
the expression \eqref{eq:weighted-xs} is larger for $x'_{AAA},\ldots,x'_{XXX}$ than for $x_{AAA},\ldots,x_{XXX}$.
This would contradict the choice of the solution $x_{AAA},\ldots,x_{XXX}$.

We first observe that
\begin{align*}
& \sum_{\kappa_A\kappa_B}x_{A\kappa_A\kappa_B}=\sum_{\kappa_A\kappa_B}x'_{A\kappa_A\kappa_B},\quad
  \sum_{\kappa_A\kappa_B}x_{B\kappa_A\kappa_B}=\sum_{\kappa_A\kappa_B}x'_{B\kappa_A\kappa_B}\quad\mbox{and} & \sum_{\kappa_A\kappa_B}x_{X\kappa_A\kappa_B}=\sum_{\kappa_A\kappa_B}x'_{X\kappa_A\kappa_B}.
\end{align*}
In particular, the sum of the variables $x'_{AAA},\ldots,x'_{XXX}$ is one,
which implies that $x'_{AAA},\ldots,x'_{XXX}$ is a feasible solution of the problem \optlink{}.

We now establish that the value of the objective function for $x'_{AAA},\ldots,x'_{XXX}$
is at least that for $x_{AAA},\ldots,x_{XXX}$,
i.e., the solution $x'_{AAA},\ldots,x'_{XXX}$ is optimal.
We start with analyzing the change of the sums defining $T$ and $S$.
Let $T'$ and $S'$ be the new values of the sums corresponding to $T$ and $S$, respectively.
Fix $(\kappa_B,\lambda_B)$ to be one of the six pairs $(A,B)$, $(A,X)$, $(B,X)$, $(X,A)$, $(X,B)$ and $(X,X)$, and
observe that
\begin{align}
\sum_{\substack{\kappa_A,\lambda_A\\A\kappa_A\kappa_B\to B\lambda_A\lambda_B}} x'_{A\kappa_A\kappa_B}x'_{B\lambda_A\lambda_B}
 & = \sum_{\substack{\kappa_A,\lambda_A\\A\kappa_A\kappa_B\to B\lambda_A\lambda_B}} x_{A\kappa_A\kappa_B}x_{B\lambda_A\lambda_B}\nonumber\\ 
 & + (1-\xi^2)x_{AA\kappa_B}x_{AA\lambda_B}+(1-\xi)x_{AB\kappa_B}x_{AA\lambda_B}\label{eq:partsum}\\
 & \ge \sum_{\substack{\kappa_A,\lambda_A\\A\kappa_A\kappa_B\to B\lambda_A\lambda_B}} x_{A\kappa_A\kappa_B}x_{B\lambda_A\lambda_B}.\nonumber
\end{align}
To see why the equality in \eqref{eq:partsum} holds,
note that the left side of the sum (recall that $\kappa_B$ and $\lambda_B$ is fixed)
is equal to
\[x'_{AA\kappa_B}x'_{BX\lambda_B}+x'_{AX\kappa_B}x'_{BA\lambda_B}+x'_{AX\kappa_B}x'_{BX\lambda_B}+x'_{AB\kappa_B}x'_{BX\lambda_B}\]
if $\kappa_B=A$, or to
\[x'_{AA\kappa_B}x'_{BX\lambda_B}+x'_{AX\kappa_B}x'_{BA\lambda_B}+x'_{AX\kappa_B}x'_{BX\lambda_B}+x'_{AB\kappa_B}x'_{BX\lambda_B}+(x'_{AA\kappa_B}+x'_{AX\kappa_B})x'_{BB\lambda_B},\]
otherwise.
Observe that the additional term in the latter expression has the same value as in the original solution,
i.e., $(x'_{AA\kappa_B}+x'_{AX\kappa_B})x'_{BB\lambda_B}=(x_{AA\kappa_B}+x_{AX\kappa_B})x_{BB\lambda_B}$.
Since the values of $(x_{AA\kappa_B}+x_{AX\kappa_B})(x_{BA\lambda_B}+x_{BX\lambda_B})$ and
$(x'_{AA\kappa_B}+x'_{AX\kappa_B})(x'_{BA\lambda_B}+x'_{BX\lambda_B})$ are the same,
it follows that the sum of the first three terms
has increased by $(1-\xi^2)x_{AA\kappa_B}x_{AA\lambda_B}$ compared to the original solution.
Finally, the difference between $x'_{AB\kappa_B}x'_{BX\lambda_B}$ and $x_{AB\kappa_B}x_{BX\lambda_B}$
is $(1-\xi)x_{AB\kappa_B}x_{AA\lambda_B}$.
Hence, the equality in \eqref{eq:partsum} is now established.
The estimate \eqref{eq:partsum} summed over possible choices of $\kappa_B$ and $\lambda_B$ yields that $T'\ge T$.
Next observe that
\[S'=\sum_{\kappa A\kappa_B\to \lambda B\lambda_B}x'_{\kappa A\kappa_B}x'_{\lambda B\lambda_B}=
    \sum_{\kappa A\kappa_B\to \lambda B\lambda_B}\xi x_{\kappa A\kappa_B}x_{\lambda B\lambda_B}=
    \xi S=T.
  \]
We conclude that it holds that $\min\{T',S'\}=\min\{T,S\}=T$.

We next analyze the first sum from the second minimum in the definition of the problem \optlink{};
let
\[R=\sum_{\kappa\kappa_AA\to \lambda\lambda_AB}x_{\kappa\kappa_AA}x_{\lambda\lambda_AB}\qquad\mbox{and}\qquad
  R'=\sum_{\kappa\kappa_AA\to \lambda\lambda_AB}x'_{\kappa\kappa_AA}x'_{\lambda\lambda_AB}.\]
Fix $(\kappa,\lambda)$ to be one of the six pairs $(A,B)$, $(A,X)$, $(B,X)$, $(X,A)$, $(X,B)$ and $(X,X)$, and
observe (the proof of the equality is analogous to that of \eqref{eq:partsum}) that
\begin{align*}
\sum_{\substack{\kappa_A,\lambda_A\\\kappa\kappa_AA\to\lambda\lambda_AB}} x'_{\kappa\kappa_AA}x'_{\lambda\lambda_AB}
 & = \sum_{\substack{\kappa_A,\lambda_A\\\kappa\kappa_AA\to\lambda\lambda_AB}} x_{\kappa\kappa_AA}x_{\lambda\lambda_AB}\\
 & + (1-\xi^2) x_{\kappa AA}x_{\lambda AB}+(1-\xi)x_{\kappa BA}x_{\lambda AB}\\
 & \ge \sum_{\substack{\kappa_A,\lambda_A\\\kappa\kappa_AA\to\lambda\lambda_AB}} x_{\kappa\kappa_AA}x_{\lambda\lambda_AB}.
\end{align*}
It follows that $R'\ge R$ and so $\min\{R',T'\}\ge\min\{R,T\}$ (recall that we have already established that $T'\ge T$).
Since it holds that $T\le T'$, $\min\{T,S\}=\min\{T',S'\}$, $\min\{R,T\}\le\min\{R',T'\}$ and
\[\sum_{\kappa_A\kappa_B}x_{\kappa\kappa_A\kappa_B}=\sum_{\kappa_A\kappa_B}x'_{\kappa\kappa_A\kappa_B}\]
for every $\kappa\in\{A,B,X\}$,
the value of the objective function for $x'_{AAA},\ldots,x'_{XXX}$
is at least that for $x_{AAA},\ldots,x_{XXX}$,
which implies that $x'_{AAA},\ldots,x'_{XXX}$ is an optimal solution of the problem \optlink.

To complete the proof, we note that the value of \eqref{eq:weighted-xs} for $x'_{AAA},\ldots,x'_{XXX}$
compared to its value for $x_{AAA},\ldots,x_{XXX}$
is larger by precisely $(1-\xi)\sum_{\kappa,\kappa_B}x_{\kappa A\kappa_B}$,
which is positive (otherwise all $x_{\kappa A\kappa_B}$ are equal to $0$, which would imply $S=0$ but $S>T\ge 0$).
This contradicts the choice of $x_{AAA},\ldots,x_{XXX}$.
\end{proof}

Before proving the main theorem of this section, we need to establish the following auxiliary lemma.

\begin{lemma}
\label{lm:strange-ineq}
The following inequality holds for all non-negative reals $a$, $a_1$, $a_2$, $b$, $b_1$, $b_2$, $x$, $x_1$ and $x_2$ such that
$x_1+x_2\le x$, $a_1+a_2\le a,$ $b_1+b_2\le b$ and $x\le \min \{a,b\}$:
\begin{equation}
a_1b_2+a_1x_2+b_1a_2+b_1x_2+x_1a_2+x_1b_2 \le ab+x\max \{a,b\}.\label{eq:strange-ineq}
\end{equation}
\end{lemma}   

\begin{proof}
First, note that we may assume $a_1+a_2=a, b_1+b_2=b$ and $x_1+x_2=x$ by increasing $a_1$, $b_1$ or $x_1$ if necessary,
since the variables $a_1$, $b_1$ or $x_1$ appear in the left side of \eqref{eq:strange-ineq} only.
By symmetry, we can also assume that $b\le a$.
Hence, the right side of \eqref{eq:strange-ineq} is $ab+ax=(a_1+a_2)(b_1+b_2+x_1+x_2)$.
Expanding and canceling the common terms transforms \eqref{eq:strange-ineq} to
\[ b_1x_2+x_1b_2 \le a_1b_1+a_1x_1+a_2b_2+a_2x_2,\]
which is equivalent to
\[ (b_2-a_1)(x_1-a_2)+(b_1-a_2)(x_2-a_1)\le 2a_1a_2.\]
Using that $b_1+b_2=b\le a= a_1+a_2$ and $0\le b_2$, we obtain that $-a_1\le b_2-a_1\le a_2$, and
similarly, using $x_1+x_2\le a_1+a_2$ and $0\le x_1$, we obtain that $-a_2\le x_1-a_2\le a_1$.
It follows that $(b_2-a_1)(x_1-a_2)\le a_1a_2$.
An analogous argument yields that $(b_1-a_2)(x_2-a_1)\le a_1a_2$,
which completes the proof of the lemma.
\end{proof}

We are now ready to determine the optimum value of the problem \optlink{}.

\begin{theorem}
\label{thm:optim}
The optimum value of the problem \optlink{} is $4/27$.
\end{theorem}

\begin{proof}
Since the objective value of the problem \optlink{} is $4/27$
for $x_{AAX}=2/3$, $x_{BBX}=1/3$ and the remaining variables equal to zero,
the optimum value of the problem \optlink{} is at least $4/27$.
By Lemma~\ref{lm:nomin},
it is enough to show that the optimum value of the problem \optineq{} is at most $4/27$.

Consider $x_{AAA},\ldots,x_{XXX}$ that maximize the objective value of the problem \optineq.

For the rest of the proof, set
\[ a = \sum_{\kappa_A\kappa_B}x_{A\kappa_A\kappa_B},\quad
   b = \sum_{\kappa_A\kappa_B}x_{B\kappa_A\kappa_B}\quad\mbox{and}\quad
   x = \sum_{\kappa_A\kappa_B}x_{X\kappa_A\kappa_B}.
   \]
Observe that $a+b+x=1$.  
In addition, we use $T$, $R$ and $S$ for
\begin{align*}
T & = \sum_{A\kappa_A\kappa_B\to B\lambda_A\lambda_B}x_{A\kappa_A\kappa_B}x_{B\lambda_A\lambda_B},\quad
R & = \sum_{\kappa A\kappa_B\to \lambda B\lambda_B}x_{\kappa A\kappa_B}x_{\lambda B\lambda_B}\quad\mbox{and}\quad
S & = \sum_{\kappa\kappa_A A\to \lambda\lambda_A B}x_{\kappa\kappa_A A}x_{\lambda\lambda_A B}.
\end{align*}
Note that the value of the objective function of \optineq{} is $xT+aR+bS$ and it holds that $R\le T$ and $S\le T$.
Also note that $T\le ab$.
The rest of the proof is devoted to establishing that $xT+aR+bS$ is at most $4/27$.

We start with showing that we can assume that $x\le a$ and $x\le b$.
If $x\ge a$, then it would hold that
\[xT+aR+bS\le (x+a+b)T=T\le ab\le a(1-x-a)\le a(1-2a)\le 1/8<4/27.\]
Similarly, if $x\ge b$, then $xT+aR+bS<4/27$.
Hence, we will assume that $x\le a$ and $x\le b$ in the rest of the proof.

Our next step is to find suitable upper bounds on $R$ and $S$;
we will give bounds separately for the parts where the first entries are both equal to $X$ and the rest.
So, we define
\[
R' = R - \sum_{XA\kappa_B\to XB\lambda_B}x_{XA\kappa_B}x_{XB\lambda_B}\quad\mbox{and}\quad
S' = S - \sum_{X\kappa_A A\to X\lambda_A B}x_{X\kappa_A A}x_{X\lambda_A B}.
\]
We derive using $pq\le (p+q)^2/4$ that
\begin{align}
R-R' & = x_{XAB}x_{XBX}+x_{XAX}x_{XBA}+x_{XAX}x_{XBB}+x_{XAX}x_{XBX} \nonumber \\
     & \le (x_{XAB}+x_{XAX})(x_{XBA}+x_{XBB}+x_{XBX}) \nonumber \\
     & \le (x_{XAB}+x_{XAX}+x_{XBA}+x_{XBB}+x_{XBX})^2/4 \le x^2/4, \label{eq:RR}
\end{align}
and similarly we derive that
\begin{equation}
S-S' = x_{XAA}x_{XXB}+x_{XBA}x_{XXB}+x_{XXA}x_{XAB}+x_{XXA}x_{XXB} \le x^2/4. \label{eq:SS}
\end{equation}

We next show that
\begin{equation}
R'+S'\le a_1b_2+a_1x_2+b_1a_2+b_1x_2+x_1a_2+x_1b_2, \label{eq:RS}
\end{equation}
where the quantities $a_1$, $a_2$, $b_1$, $b_2$, $x_1$ and $x_2$ (note that $a_1+a_2\le a$, $b_1+b_2\le b$ and $x_1+x_2\le x$) are defined as
\begin{align*}
    a_1&=x_{AAA}+x_{ABA}+x_{ABB}+x_{ABX}+x_{AXA}, \quad
    &
    a_2&=x_{AAB}+x_{AAX}+x_{AXB},
    \\ 
    b_1&=x_{BAA}+x_{BBA}+x_{BBB}+x_{BBX}+x_{BXA},  
    &
    b_2&=x_{BAB}+x_{BAX}+x_{BXB}, 
    \\
    x_1&=x_{XAA}+x_{XBA}+x_{XBB}+x_{XBX}+x_{XXA}, 
    &
    x_2&=x_{XAB}+x_{XAX}+x_{XXB}.
\end{align*}
First observe that the following inequalities hold (note that
$R'$ is equal to the sum of the five sums in the left column and
$S'$ to the sum of the five sums in the right column):
\begin{align*}
    \sum_{AA\kappa_B\to BB\lambda_B}x_{AA\kappa_B}x_{BB\lambda_B}&\le a_2b_1,
   &\sum_{A\kappa_AA\to B\lambda_AB}x_{A\kappa_AA}x_{B\lambda_AB}&\le a_1b_2,\\
    \sum_{AA\kappa_B\to XB\lambda_B}x_{AA\kappa_B}x_{XB\lambda_B}&\le a_2x_1,
   &\sum_{A\kappa_AA\to X\lambda_AB}x_{A\kappa_AA}x_{X\lambda_AB}&\le a_1x_2,\\
    \sum_{BA\kappa_B\to XB\lambda_B}x_{BA\kappa_B}x_{XB\lambda_B}&\le b_2x_1,
   &\sum_{B\kappa_AA\to X\lambda_AB}x_{B\kappa_AA}x_{X\lambda_AB}&\le b_1x_2,\\
    \sum_{XA\kappa_B\to AB\lambda_B}x_{XA\kappa_B}x_{AB\lambda_B}&\le x_2a_1,
   &\sum_{X\kappa_AA\to A\lambda_AB}x_{X\kappa_AA}x_{A\lambda_AB}&\le x_1a_2,\\
    \sum_{XA\kappa_B\to BB\lambda_B}x_{XA\kappa_B}x_{BB\lambda_B}&\le x_2b_1,
   &\sum_{X\kappa_AA\to B\lambda_AB}x_{X\kappa_AA}x_{B\lambda_AB}&\le x_1b_2.
\end{align*}
Each of the ten inequalities can be verified by writing out the terms forming the sum on its left side and terms arising from multiplying the two quantities on the right side.
Another way of verifying the inequalities is using the definition of $\to$ as follows:
consider one of the inequalities on the left side, say the second, and observe that
$\kappa_B$ cannot be equal to $A$ in the sum and so each $x_{AA\kappa_B}$ appears as a term in the sum defining $a_2$, and
each $x_{XB\lambda_B}$ appears as a term in the sum defining $x_1$ (note that the variables in the sums defining $a_1$, $b_1$ and $x_1$
are exactly those with the second entry of the subscript being $B$ or the third entry of the subscript being $A$).
An analogous argument applies to the five inequalities on the right side.
Since no term appears in both the sum defining $R$ and the sum defining $S$ (there is no $\kappa$ and $\lambda$ such that
$\kappa AA\to\lambda BB$), we actually obtain that the following holds:
\begin{align*}
 \sum_{XA\kappa_B\to AB\lambda_B}x_{XA\kappa_B}x_{AB\lambda_B}+
 \sum_{A\kappa_AA\to X\lambda_AB}x_{A\kappa_AA}x_{X\lambda_AB}&\le a_1x_2,\\
 \sum_{XA\kappa_B\to BB\lambda_B}x_{XA\kappa_B}x_{BB\lambda_B}+
 \sum_{B\kappa_AA\to X\lambda_AB}x_{B\kappa_AA}x_{X\lambda_AB}&\le b_1x_2,\\
 \sum_{AA\kappa_B\to XB\lambda_B}x_{AA\kappa_B}x_{XB\lambda_B}+
 \sum_{X\kappa_AA\to A\lambda_AB}x_{X\kappa_AA}x_{A\lambda_AB}&\le x_1a_2,\mbox{ and}\\
 \sum_{BA\kappa_B\to XB\lambda_B}x_{BA\kappa_B}x_{XB\lambda_B}+
 \sum_{X\kappa_AA\to B\lambda_AB}x_{X\kappa_AA}x_{B\lambda_AB}&\le x_1b_2.
\end{align*} 
The inequality \eqref{eq:RS} now follows.
Applying Lemma~\ref{lm:strange-ineq} to the right side of \eqref{eq:RS} yields that
\begin{equation}
R'+S'\le ab+x\cdot\max\{a,b\}. \label{eq:RS2}
\end{equation}
Recall that $R'\le R\le T$, $S'\le S\le T$ and $T\le ab$. It now follows that
\begin{align*}
aR'+bS' &\le\max\{a,b\}ab+\min\{a,b\}\max\{R'+S'-ab,0\}\\
        &\le\max\{a,b\}ab+\min\{a,b\}\max\{a,b\}x=\max\{a,b\}ab+abx.
\end{align*}
Using \eqref{eq:RR} and \eqref{eq:SS}, we obtain that
\[ a(R-R')+b(S-S')\le (a+b)\frac{x^2}{4}. \]
It follows (note that $T\le ab$) that
\begin{equation}
xT+aR+bS \le 2abx+\max\{a,b\}ab+\frac{(a+b)x^2}{4}. \label{eq:TRS}
\end{equation}

To complete the proof, we need to show that the right side of \eqref{eq:TRS}
is at most $4/27$ for all non-negative reals $a$, $b$ and $x$ such that
$a+b+x=1$, $x\le a$ and $x\le b$.
As the right side of \eqref{eq:TRS} is symmetric in $a$ and $b$,
we can also assume that $a\ge b$.
Hence, we need to show that it holds that
\begin{equation}
a^2b+2abx+\frac{(a+b)x^2}{4}\le\frac{4}{27} \label{eq:goal}
\end{equation}
assuming that $0\le x\le b\le a$ and $a+b+x=1$.
The left side of \eqref{eq:goal} can be rewritten as
\begin{align}
a^2b+2abx+\frac{(a+b)x^2}{4} & = b(a+x)^2 +\frac{(a-3b)x^2}{4} 
                             = b(1-b)^2 +\frac{(a-3b)x^2}{4}.\label{eq:goal2}
\end{align}
If $a\le 3b$, then the fraction in \eqref{eq:goal2} is non-positive,
which implies that the left side is at most $b(1-b)^2\le 4/27$.
Hence, we can assume that $a>3b$, which implies that $b<1/4$.
It follows that the term $b(1-b)^2$ is less than $9/64$ and that
\[\frac{(a-3b)x^2}{4} =
  \frac{(1-4b-x)x^2}{4} \le
  \frac{(1-5x)x^2}{4} \le \frac{1}{675}.\]
It follows that the right side of \eqref{eq:goal2} is less than $9/64+1/675<4/27$,
which establishes \eqref{eq:goal}.
\end{proof}

\section{Main result}
\label{sec:main}

We are now ready to prove our main structural result.

\begin{theorem}
\label{thm:base}
For every $\delta>0$ and every $n\in\NN$, there exists $N\in\NN$ such that
every $N$-partitioned hypergraph $H$ with density at least $4/27+\delta$ 
has a $4n$-partitioned subhypergraph $H_0$ with index set $I=\{i_1,\ldots,i_{4n}\}$, $i_1<i_2<\cdots<i_{4n}$, such that
one of the following holds.
\begin{itemize}
\item There exist vertices $\alpha_{ij},\beta_{ij},\gamma_{ij}^0,\omega_{ij},\alpha^{\omega}_{ij},\gamma^{\omega}_{ij}\in V_{ij}$ for all $i<j$, $i,j\in I$, such that
      the following holds for all $i<j<k$, $i,j,k\in I$:
      \begin{itemize}
      \item $\{\alpha_{ij},\beta_{jk},\gamma_{ik}^0\}$ is an edge,
      \item $\{\omega_{ij},\beta_{jk},\gamma^{\omega}_{ik}\}$ is an edge, and
      \item $\{\alpha^{\omega}_{ij},\omega_{jk},\gamma_{ik}^0\}$ is an edge.
      \end{itemize}
\item There exist vertices $\alpha_{ij},\beta_{ij},\gamma_{ij}^0,\omega_{ij},\beta^{\omega}_{ij},\gamma^{\omega}_{ij}\in V_{ij}$ for all $i<j$, $i,j\in I$, such that
      the following holds for all $i<j<k$, $i,j,k\in I$:
      \begin{itemize}
      \item $\{\alpha_{ij},\beta_{jk},\gamma_{ik}^0\}$ is an edge,
      \item $\{\alpha_{ij},\omega_{jk},\gamma^{\omega}_{ik}\}$ is an edge, and
      \item $\{\omega_{ij},\beta^{\omega}_{jk},\gamma_{ik}^0\}$ is an edge.
      \end{itemize}

\item There exist vertices $\alpha_{ij},\beta_{ij},\alpha'_{ij},\beta'_{ij},\gamma_{ij},\gamma^{\bullet}_{ij}\in V_{ij}$ for all $i<j$, $i,j\in I$, and
      $\beta^{\bullet}_{ijkst}\in V_{jk}$ and $\alpha^{\bullet}_{ijkst}\in V_{ks}$ for all $i<j<k<s<t$, $i,j,k,s,t\in I$, such that
      the following holds for all $i<j<k<s<t$, $i,j,k,s,t\in I$:
      \begin{itemize}
      \item $\{\alpha_{ij},\beta_{jk},\gamma_{ik}^0\}$ is an edge in the $(i,j,k)$-triad,
      \item $\{\alpha'_{ij},\beta^{\bullet}_{ijkst},\gamma_{ik}^0\}$ is an edge in the $(i,j,k)$-triad,
      \item $\{\beta^{\bullet}_{ijkst},\alpha^{\bullet}_{ijkst},\gamma^{\bullet}_{js}\}$ is an edge in the $(j,k,s)$-triad, and
      \item $\{\alpha^{\bullet}_{ijkst},\beta'_{st},\gamma_{kt}^0\}$ is an edge in the $(k,s,t)$-triad.
      \end{itemize}
\item There exist vertices $\alpha^*_{ij}, \beta^*_{ij}, \gamma_{ij},\alpha^1_{ij},\beta^1_{ij},\gamma^1_{ij},\alpha^2_{ij},\beta^2_{ij},\gamma^2_{ij}\in V_{ij}$ for all $i<j$, $i,j\in I$, such that
      the following holds for all $i<j<k$, $i,j,k\in I$:
      \begin{itemize}
      \item $\{\gamma^1_{ij},\beta^*_{jk},\gamma_{ik}\}$ is an edge,
      \item $\{\alpha^*_{ij},\gamma^2_{jk},\gamma_{ik}\}$ is an edge,
      \item $\{\alpha^1_{ij},\beta^1_{jk},\gamma^1_{ik}\}$ is an edge, and
      \item $\{\alpha^2_{ij},\beta^2_{jk},\gamma^2_{ik}\}$ is an edge.
      \end{itemize}
      In addition,
      there exist $\alpha\in V_{i_{2n}i_{2n+1}}$ and $\beta\in V_{i_{2n+1}i_{2n+2}}$ such that
      the vertex $\alpha$ is a left neighbor of $\gamma^2_{i_{2n}i_{3n+1}}$,
      the vertex $\beta$ is a right neighbor of $\gamma^1_{i_ni_{2n+2}}$, and
      the $(i_{2n},i_{2n+1},i_{2n+2})$-triad contains an edge containing $\alpha$ and $\beta$.
\end{itemize}
\end{theorem}

\begin{proof} 
Assume that the statement of the theorem is false for some $\delta>0$ and $n\in\NN$,
which we fix for the rest of the proof;
without loss of generality, we will assume that $\delta\le 1/27$.
We will show that the assumption that the statement is false for $\delta$ and $n$ implies that 
the problem \optlink{} has a solution of value larger than $4/27$,
which is impossible by Theorem~\ref{thm:optim}.

Fix a positive $\varepsilon\le\delta/250$ and define $n_1,\ldots,n_{11}$ as follows.
First, apply Lemma~\ref{lm:gamma12ramsey} with $\varepsilon^2$ to get $N$ and set $n_1=(N+1)n$.
Next, apply Lemma~\ref{lm:extend0ab} with $\varepsilon^2$ and $n_1$ (as the value of $n$) to get $n_2$, and
then again with $\varepsilon^2$ and $n_2$ to get $n_3$.
Set $n_4=R_3^3(n_3)$ and
apply Lemma~\ref{lm:gammaABX} with $\varepsilon$ and $n_4$ to get $n_5$.
We apply Lemma~\ref{lm:BAedges} with $\varepsilon$ and $n_5$ to get $n_5'$,
then Lemma~\ref{lm:BBedges} with $\varepsilon$ and $n_5'$ to get $n_{5}''$, and
eventually Lemma~\ref{lm:AAedges} with $\varepsilon$ and $n_{5}''$ to get $n_{5'''}$.
Yet again we apply Lemmas~\ref{lm:BAedges},~\ref{lm:BBedges} and~\ref{lm:AAedges} starting with $\varepsilon$ and $n_{5}'''$ to obtain $n_{6}$.
We next apply Lemma~\ref{lm:BAedges} with $\varepsilon$ and $n_6$ to get $n_7$,
then Lemma~\ref{lm:BBedges} with $\varepsilon$ and $n_7$ to get $n_{8}$, and
eventually Lemma~\ref{lm:AAedges} with $\varepsilon$ and $n_{8}$ to get $n_{9}$.
Finally, we apply Lemma~\ref{lm:gamma12} with $\varepsilon$ and $n_{9}$ to get $n_{10}$ and
Lemma~\ref{lm:gamma} with $\varepsilon$ and $n_{10}$ to get $n_{11}$.

Since the statement of the theorem is false for $\delta$ and $n$,
there exists an $n_{11}$-partitioned hypergraph $H$ with density $d=4/27+\delta$ such that
neither of the four conclusions of the theorem holds.
Fix such an $n_{11}$-partitioned hypergraph $H$ for the rest of the proof.

We first apply Lemma~\ref{lm:gamma} with the $n_{11}$-partitioned hypergraph $H$
to get an $n_{10}$-partitioned subhypergraph $H_{10}$ of $H$ with index set $I_{10}$ and
vertices $\gamma_{ik}\in V_{ik}$, $i<k$, $i,k\in I_{10}$, such that
\begin{itemize}
\item the density of $H_{10}$ is at least $d-\varepsilon$, and
\item the degree of each top vertex of an $(i,j,k)$-triad
      is at most that of $\gamma_{ik}$ for all $i<j<k$, $i,j,k\in I_{10}$.
\end{itemize}
Note that the degrees of the vertices $\gamma_{ik}$ in each $(i,j,k)$-triad are at least $d-\varepsilon\ge 4/27+\delta-\varepsilon\ge 5\delta-\varepsilon\ge 1224\varepsilon$.
Therefore, we can now apply Lemma~\ref{lm:gamma12} with $H_{10}$ and vertices $\gamma_{ik}$, $i<k$, $i,k\in I_{10}$,
to get an $n_{9}$-partitioned subhypergraph $H_{9}$ of $H_{10}$ with index set $I_{9}\subseteq I_{10}$ and
vertices $\alpha^*_{ik},\beta^*_{ik},\gamma^1_{ik},\gamma^2_{ik}\in V_{ik}$, $i<k$, $i,k\in I_{9}$, such that the density of $H_{9}$ is at least $d-2\varepsilon$,
\begin{itemize}
\item the degree of each top vertex of an $(i,j,k)$-triad
      is larger by at most $\varepsilon$ than that of $\gamma_{ik}$ for all $i<j<k$, $i,j,k\in I_{9}$,
\item $\{\gamma^1_{ij},\beta^*_{jk}, \gamma_{ik}\}$ and  $\{\alpha^*_{ij},\gamma^2_{jk}, \gamma_{ik}\}$ are edges in the $(i,j,k)$-triad for all $i<j<k$, $i,j,k\in I_{9}$,
\item for all $i<\ell<j<k$, $i,j,k,\ell\in I_{9},$
      among left neighbors of $\gamma_{ik}$ in the $(i,j,k)$-triad, $\gamma^1_{ij}$ has the largest degree in the $(i,\ell,j)$-triad of $H_{9}$, and
\item for all $i<j<\ell<k$, $i,j,k,\ell\in I_{9}$,
      among right neighbors of $\gamma_{ik}$ in the $(i,j,k)$-triad, $\gamma^2_{jk}$ has the largest degree in the $(j,\ell,k)$-triad of $H_{9}$.
\end{itemize}

As the next step, we apply Lemma~\ref{lm:AAedges} with the $n_{9}$-partitioned hypergraph $H_{9}$ and
the vertices $\gamma_{ik}^0=\gamma_{ik}$, for all $i<k$, $i,k\in I_{9}$.
Since the hypergraph $H$ does not satisfy the first conclusion of the theorem,
which matches the first outcome of Lemma~\ref{lm:AAedges},
we conclude that its second outcome must hold and
hence obtain an $n_{8}$-partitioned subhypergraph $H_{8}$ of $H_{9}$.
We next apply Lemma~\ref{lm:BBedges} with the $n_{8}$-partitioned hypergraph $H_{8}$ and
the vertices $\gamma_{ik}^0=\gamma_{ik}$, for all $i<k$, $i,k\in I_{8}$.
Since the hypergraph $H$ does not satisfy the second conclusion of the theorem,
which matches the first outcome of Lemma~\ref{lm:BBedges},
we conclude that its second outcome must hold and
hence obtain an $n_{7}$-partitioned subhypergraph $H_{7}$ of $H_{8}$.
Next, we apply Lemma~\ref{lm:BAedges} with the $n_{7}$-partitioned hypergraph $H_{7}$ and
the vertices $\gamma_{ik}^0=\gamma_{ik}$, for all $i<k$, $i,k\in I_{7}$.
Since the hypergraph $H$ does not satisfy the third conclusion of the theorem,
which matches the first outcome of Lemma~\ref{lm:BAedges},
we conclude that its second outcome must hold and
hence obtain an $n_{6}$-partitioned subhypergraph $H_{6}$ of $H_{7}$.
We now apply Lemmas~\ref{lm:AAedges},~\ref{lm:BBedges} and~\ref{lm:BAedges} in the same way
but with $\gamma_{ik}^0=\gamma_{ik}^1$ instead of $\gamma_{ik}^0=\gamma_{ik}$,
then once again with $\gamma_{ik}^0=\gamma_{ik}^2$.
After these six additional steps,
we obtain an $n_5$-partitioned subhypergraph $H_{5}$ of $H_{6}$,
which satisfies the second conclusions of each of Lemmas~\ref{lm:AAedges},~\ref{lm:BBedges} and~\ref{lm:BAedges}
for each of the choices $\gamma_{ik}^0=\gamma_{ik}$, $\gamma_{ik}^0=\gamma_{ik}^1$ and $\gamma_{ik}^0=\gamma_{ik}^2$
for all $i<k, i,k \in I_5$.
Note that in each of the nine applications of the three lemmas,
the density can drop by at most $\varepsilon$.
Hence, the density of $H_5$ is at least $d-11\varepsilon$ and the following holds for all $i<j<k$, $i,j,k\in I_5$:
\begin{itemize}
\item the degree of each top vertex of an $(i,j,k)$-triad
      is larger by at most $10\varepsilon$ than that of $\gamma_{ik}$,
\item for every $\ell\in I_5$, $i<\ell<j$,
      every left neighbor of $\gamma_{ik}$ in the $(i,j,k)$-triad has degree in the $(i,\ell,j)$-triad
      by at most $9\varepsilon$ larger than that of $\gamma^1_{ij}$,
\item for every $\ell\in I_5$, $j<\ell<k$,
      every right neighbor of $\gamma_{ik}$ in the $(i,j,k)$-triad has degree in the $(j,\ell,k)$-triad
      by at most $9\varepsilon$ larger than that of $\gamma^2_{ij}$, and
\item for every $\ell\in I_5$, $\ell>k$, and
      each of the three choices $(\gamma_{ik}^0,\gamma_{j\ell}^0)$ as $(\gamma_{ik},\gamma_{j\ell})$,
      $(\gamma_{ik}^1,\gamma_{j\ell}^1)$, and $(\gamma_{ik}^2,\gamma_{j\ell}^2)$, it holds that
      \begin{itemize}
      \item the $(j,k,\ell)$-triad has no edge
	    containing both a right neighbor of $\gamma^0_{ik}$ (in the $(i,j,k)$-triad) and
	    a right neighbor of $\gamma^0_{j\ell}$ (in the $(j,k,\ell)$-triad),
      \item the $(i,j,k)$-triad has no edge
            containing a left neighbor of $\gamma^0_{ik}$ (in the $(i,j,k)$-triad) and
	    a left neighbor of $\gamma^0_{j\ell}$ (in the $(j,k,\ell)$-triad), and
      \item no right neighbor of $\gamma_{ik}^0$ in the $(i,j,k)$-triad
            is a left neighbor of $\gamma_{j\ell}^0$ in the $(j,\ell,k)$-triad,
      \end{itemize}
\item for every $s,t\in I_5$, $k<s<t$, and
      each of the three choices $(\gamma_{ik}^0,\gamma_{kt}^0)$ as $(\gamma_{ik},\gamma_{kt})$,
      $(\gamma_{ik}^1,\gamma_{kt}^1)$, and $(\gamma_{ik}^2,\gamma_{kt}^2)$, it holds that
      \begin{itemize}
      \item the $(j,k,s)$-triad has no edge containing
            a right neighbor of $\gamma_{ik}^0$ (in the $(i,j,k)$-triad) and
            a left neighbor of $\gamma_{kt}^0$ (in the $(k,s,t)$-triad).
      \end{itemize}
\end{itemize}
The first three properties above hold as in each of the applications of Lemmas~\ref{lm:AAedges}, \ref{lm:BBedges} or \ref{lm:BAedges},
the degree of each of the vertices $\gamma_{ik},\gamma_{ik}^1$ and $\gamma^2_{ik}$ is decreased by at most $\varepsilon$.
The fourth property is ensured by Lemmas~\ref{lm:AAedges} and~\ref{lm:BBedges} (the third point of the fourth property follows from either of the first two), and 
the fifth property by Lemma~\ref{lm:BAedges}.

We now iterate the following for all $i<j<k, i,j,k\in I_5$ as long as possible:
if there is a vertex $w\in V_{ij} \cup V_{jk}$ such that
its codegree with $\gamma_{ik}$ is less than $\varepsilon$, delete all edges containing $w$ and $\gamma_{ik}$,
if there is a vertex $w\in V_{ij} \cup V_{jk}$ such that
its codegree with $\gamma^1_{ik}$ is less than $\varepsilon$, delete all edges containing $w$ and $\gamma^1_{ik}$, and
if there is a vertex $w\in V_{ij} \cup V_{jk}$ such that
its codegree with $\gamma^2_{ik}$ is less than $\varepsilon$, delete all edges containing $w$ and $\gamma^2_{ik}$.
Note that the degree of each of the vertices $\gamma_{ik}$, $\gamma_{ik}^1$ and $\gamma_{ik}^2$
has descreased by at most $2\varepsilon$ as the above procedure applies for each vertex $w\in V_{ij} \cup V_{jk}$ and
each of the vertices $\gamma_{ik}$, $\gamma_{ik}^1$ and $\gamma_{ik}^2$ at most once.
The resulting $n_5$-partitioned hypergraph is denoted by $H'_5$;
note that the density of $H'_5$ is at least $d-13\varepsilon$.

We now apply Lemma~\ref{lm:gammaABX} to the $n_5$-partitioned hypergraph $H_5'$ and
the vertices $\gamma_{ik}$, $\gamma^1_{ik}$ and $\gamma^2_{ik}$, $i,k\in I_5$,
to get an $n_4$-partitioned subhypergraph $H_4$ of $H_5'$ \emph{induced} by $I_4\subseteq I_5$ and
reals $d_{\kappa\kappa^1\kappa^2}$ for $\kappa,\kappa^1,\kappa^2\in\{A,B,X\}$ as in the statement of Lemma~\ref{lm:gammaABX}.
Since the sets $A_{ijk\ell}$ and $B_{ijk\ell}$ from the statement of Lemma~\ref{lm:gammaABX} are disjoint (this is implied by the third part of the fourth property of $H_5$ stated above),
we obtain that the sum of the $27$ reals $d_{\kappa\kappa^1\kappa^2}$ is at most $1+27\varepsilon\le 251/250$.

We now construct an auxiliary $3$-edge-colored $3$-uniform complete hypergraph $G$ with vertex set $I_4$;
an edge formed by $i<j<k$, $i,j,k\in I_4$,
is colored (choose the first option that applies) 
\begin{itemize} 
\item red if the degree of $\gamma^1_{ik}$ in the $(i,j,k)$-triad of $H_4$ is less than $\varepsilon^2$,
\item orange if the degree of $\gamma^2_{ik}$ in the $(i,j,k)$-triad of $H_4$ is less than $\varepsilon^2$ and
\item blue, otherwise.
\end{itemize}
By Ramsey's Theorem, there exist $I_3\subseteq I_4$ of size $n_3$ such that
all edges of $G$ formed by the elements of $I_3$ have the same color.

We first consider the case that all the edges formed by elements of $I_3$ are blue,
i.e., the degree of both $\gamma^1_{ik}$ and $\gamma^2_{ik}$ in each $(i,j,k)$-triad of
the $n_3$-partitioned subhypergraph of $H_4$ induced by $I_3$ is at least $\varepsilon^2$.
We now apply Lemma~\ref{lm:extend0ab} to the $n_3$-partitioned subhypergraph of $H_4$ induced by $I_3$ and
the vertices $\gamma^1_{ik}$, $i<k$, $i,k\in I_3$, to get an $n_2$-element subset $I_2\subseteq I_3$ and
vertices $\alpha^1_{ij}$ and $\beta^1_{ij}$, $i<j$, $i,j\in I_2$, such that
$\{\alpha^1_{ij},\beta^1_{jk},\gamma^1_{ik}\}$ is an edge for all $i<j<k$, $i,j,k\in I_2$.
We next apply Lemma~\ref{lm:extend0ab} to the $n_2$-partitioned subhypergraph of $H_4$ induced by $I_2$ and
the vertices $\gamma^2_{ik}$, $i<k$, $i,k\in I_2$, to get an $n_1$-element subset $I_1\subseteq I_2$ and
vertices $\alpha^2_{ij}$ and $\beta^2_{ij}$, $i<j$, $i,j\in I_1$, such that
$\{\alpha^2_{ij},\beta^2_{jk},\gamma^2_{ik}\}$ is an edge for all $i<j<k$, $i,j,k\in I_1$.
Let $j_1,\ldots,j_{n_1}$ be the indices contained in $I_1$ listed in the increasing order.

Let $J$ be the $N$-element set containing the indices $j_{kn}$ for $k\in [N]$ (recall that $n_1=(N+1)n$);
recall that $J\subseteq I_1\subseteq I_4$.
We apply Lemma~\ref{lm:gamma12ramsey} with $\varepsilon^2>0$ to the $N$-partitioned subhypergraph of $H_4$ induced by $J$
together with $\gamma^1_{ij}$ and $\gamma^2_{ij}$, $i<j$, $i,j\in J$,
to get $m\in\NN$, $k\in [m-2]$ and 
an $m$-partitioned subhypergraph $H'$ of $H_4$ with index set $\{j'_1,\ldots,j'_m\}\subseteq J$
with the properties given in the statement of the lemma;
in particular, the density of $H'$ is at least $d-14\varepsilon$. Also note that $2\le k\le m-3$. 

Suppose that the $(j'_k,j'_{k+1},j'_{k+2})$-triad of $H'$ has
an edge containing
a left neighbor $\alpha$ of $\gamma^2_{j'_k,j'_{k+2}}$ and
a right neighbor $\beta$ of $\gamma^1_{j'_k,j'_{k+2}}$.
By Lemma~\ref{lm:gamma12ramsey},
there exists $k'<k$ such that $\beta$ is a right neighbor of $\gamma^1_{j'_{k'},j'_{k+2}}$ and
there exists $k''>k+2$ such that $\alpha$ is a left neighbor of $\gamma^2_{j'_k,j'_{k''}}$.
We now set 
\begin{itemize}
\item $i_n=j'_{k'}$, $i_{2n}=j'_k$, $i_{2n+1}=j'_{k+1}$, $i_{2n+2}=j'_{k+2}$, $i_{3n+1}=j'_{k''}$,
\item $i_1,\ldots,i_{n-1}$ to be any elements of $I_1$ smaller than $j'_{k'}$,
\item $i_{n+1},\ldots,i_{2n-1}$ to be any elements of $I_1$ between $j'_{k'}$ and $j'_k$,
\item $i_{2n+3},\ldots,i_{3n}$ to be any elements of $I_1$ between $j'_{k+2}$ and $j'_{k''}$, and
\item $i_{3n+2},\ldots,i_{4n}$ to be any elements of $I_1$ larger than $j'_{k''}$.
\end{itemize}
Observe that the set $I=\{i_1,\ldots,i_{4n}\}$,
the vertices $\alpha$ and $\beta$, and
the vertices $\alpha^*_{ij},\beta^*_{ij},\gamma_{ij},\alpha^1_{ij},\beta^1_{ij},\gamma^1_{ij},\alpha^2_{ij},\beta^2_{ij},$ $\gamma^2_{ij}\in V_{ij}$, $i<j$, $i,j\in I$,
satisfy the last conclusion of the theorem,
which is impossible by the choice of $H$ (note that while some of the required edges may not be present in $H_4$,
they do exist in the original hypergraph $H$).  
Since $H$ cannot satisfy the last conclusion of the theorem,
the $(j'_k,j'_{k+1},j'_{k+2})$-triad of $H'$ has
no edge containing both a left neighbor of $\gamma^2_{j'_k,j'_{k+2}}$ and
a right neighbor of $\gamma^1_{j'_k,j'_{k+2}}$.

The rest of the proof concerns the $(j'_k,j'_{k+1},j'_{k+2})$-triad of $H'$.
For convenience, we set $i=j'_{k-1}$, $j=j'_k$, $k=j'_{k+1}$, $\ell=j'_{k+2}$ and $s=j'_{k+3}$ as
we only need to deal with these five indices, and
we write $T^1$ for the set of the top vertices in the $(j,k,\ell)$-triad that are left neighbors of $\gamma_{js}$ in the $(j,\ell,s)$-triad and
$T^2$ for the set of the top vertices that are right neighbors of $\gamma_{i\ell}$ in the $(i,j,\ell)$-triad.
We now summarize the properties of the $(j,k,\ell)$-triad of $H'$ that we use further.
\begin{itemize}
\item The density of the $(j,k,\ell)$-triad is at least $d-14\varepsilon$.
\item The degree of any top vertex of the $(j,k,\ell)$ triad is larger by at most $13\varepsilon$ than that of $\gamma_{j\ell}$.
\item The degree of any top vertex contained in $T^1$ is larger by at most $12\varepsilon$ than that of $\gamma^1_{j\ell}$.
\item The degree of any top vertex contained in $T^2$ is larger by at most $12\varepsilon$ than that of $\gamma^2_{j\ell}$.
\item The sets $T^1$ and $T^2$ are disjoint.
\item There is no edge containing both a left neighbor of $\gamma^2_{j\ell}$ and
      a right neighbor of $\gamma^1_{j\ell}$.
\end{itemize}
The fifth property of the $(j,k,\ell)$-triad follows
from the third part of the fourth property of the partitioned hypergraph $H_5$ stated earlier in the proof, and
the last property follows from Lemma~\ref{lm:gamma12ramsey} and the arguments given in the previous paragraph.
In addition, Lemma~\ref{lm:gammaABX} used to construct the partitioned hypergraph $H_4$ implies that
\begin{equation}
  \left|\frac{|T^1|}{|V_{j\ell}|}-\sum_{\kappa^1\kappa^2}d_{A\kappa^1\kappa^2}\right|\le 9\varepsilon
  \qquad\mbox{and}\qquad
  \left|\frac{|T^2|}{|V_{j\ell}|}-\sum_{\kappa^1\kappa^2}d_{B\kappa^1\kappa^2}\right|\le 9\varepsilon,\label{eq:T12}
\end{equation}
where the sums range over all choices of $\kappa^1$ and $\kappa^2$ in $\{A,B,X\}$.
Note that the conclusion of Lemma~\ref{lm:gammaABX} is used for the indices $i<j<\ell<s$ and
Lemma~\ref{lm:gamma12ramsey} does not cause removal of any edges from either the $(i,j,\ell)$-triad or $(j,\ell,s)$-triad.

Before proceeding further,
we define the $H'$ and the indices $i<j<k<\ell<s$ to be used in the analysis in the case
when all the edges of $G$ formed by the elements of $I_3$ are red or orange.
In such case,
we choose the indices $i<j<k<\ell<s$ to be any elements contained in $I_3$.
The $5$-partitioned hypergraph $H'$ is obtained from the $5$-partitioned subhypergraph of $H_4$ induced by $\{i,j,k,\ell,s\}$
by removing all edges of the $(j,k,\ell)$-triad containing $\gamma^1_{j\ell}$
in case the common color of all the edges formed by elements of $I_3$ is red, and
by removing all edges of the $(j,k,\ell)$-triad containing $\gamma^2_{j\ell}$
in case the common color is orange.
We again use $T^1$ for the set of the top vertices in the $(j,k,\ell)$-triad that
are left neighbors of $\gamma_{js}$ in the $(j,k,s)$-triad, and
$T^2$ for the set of the top vertices that
are right neighbors of $\gamma_{i\ell}$ in the $(i,j,\ell)$-triad.
Note that $T^1$ and $T^2$ satisfy \eqref{eq:T12} as only the edges of the $(j,k,\ell)$-triad were deleted, and
the $(j,k,\ell)$-triad satisfies the same six properties given in the case of the common color being blue.
In particular, the final property holds since one of $\gamma_{j\ell}^1$ and $\gamma_{j\ell}^2$ is contained in no edges at all.

We now continue the analysis jointly for all three cases of the common color of the edges in $G$ formed by $I_3$;
we will use only the six properties of the $(j,k,s)$-triad of $H'$ given above and the properties of $H_5$.
Let $H''$ be the $5$-partitioned subhypergraph of $H_4$ induced by $\{i,j,k,\ell,s\}$;
we use the $5$-partitioned hypergraph $H''$ for the definition of quantities in this paragraph.
Let $L_A$ be the set of left neighbors of $\gamma_{j\ell}$ in the $(j,k,\ell)$-triad,
$L_B$ the set of right neighbors of $\gamma_{ik}$ in the $(i,j,k)$-triad and
$L_X=V_{jk}\setminus (L_A\cup L_B)$ (note that the sets $L_A$ and $L_B$
are be disjoint by the third part of the fourth property of $H_5$).
Similarly, $L^1_A$ is the set of left neighbors of $\gamma^1_{j\ell}$ in the $(j,k,\ell)$-triad,
$L^1_B$ the set of right neighbors of $\gamma^1_{ik}$ in the $(i,j,k)$-triad and
$L^1_X=V_{jk}\setminus (L^1_A\cup L^1_B)$, and
$L^2_A$ is the set of left neighbors of $\gamma^2_{j\ell}$ in the $(j,k,\ell)$-triad,
$L^2_B$ the set of right neighbors of $\gamma^2_{ik}$ in the $(i,j,k)$-triad and
$L^2_X=V_{jk}\setminus (L^2_A\cup L^2_B)$ (again, the sets $L^1_A$ and $L^1_B$ are disjoint, and
the sets $L^2_A$ and $L^2_B$ are disjoint by the third part of the fourth property of $H_5$).
For $\kappa,\kappa^1,\kappa^2\in\{A,B,X\}$, we define
\[L_{\kappa\kappa^1\kappa^2}=L_\kappa\cap L^1_{\kappa^1}\cap L^2_{\kappa^2}.\]
Symmetrically, $R_A$ is the set of left neighbors of $\gamma_{ks}$ in the $(k,\ell,s)$-triad,
$R_B$ is the set of right neighbors of $\gamma_{j\ell}$ in the $(j,k,\ell)$-triad, and
$R_X=V_{k\ell}\setminus (R_A\cup R_B)$,
$R^1_A$ is the set of left neighbors of $\gamma^1_{ks}$ in the $(k,\ell,s)$-triad,
$R^1_B$ is the set of right neighbors of $\gamma^1_{j\ell}$ in the $(j,k,\ell)$-triad, and
$R^1_X=V_{k\ell}\setminus (R^1_A\cup R^1_B)$, and
$R^2_A$ is the set of left neighbors of $\gamma^2_{ks}$ in the $(k,\ell,s)$-triad,
$R^2_B$ is the set of right neighbors of $\gamma^2_{j\ell}$ in the $(j,k,\ell)$-triad, and
$R^2_X=V_{k\ell}\setminus (R^2_A\cup R^2_B)$, we set
\[R_{\kappa\kappa^1\kappa^2}=R_\kappa\cap R^1_{\kappa^1}\cap R^2_{\kappa^2}.\]
Since $H''$ is an induced subhypergraph of $H_4$,
Lemma~\ref{lm:gammaABX} yields that
\[\left|\frac{|L_{\kappa\kappa^1\kappa^2}|}{|V_{jk}|}-d_{\kappa\kappa^1\kappa^2}\right|\le \varepsilon
  \qquad\mbox{and}\qquad
  \left|\frac{|R_{\kappa\kappa^1\kappa^2}|}{|V_{k\ell}|}-d_{\kappa\kappa^1\kappa^2}\right|\le \varepsilon\]
for all $\kappa,\kappa^1,\kappa^2\in\{A,B,X\}$.

Next, we claim that the number of edges of $H''$ containing $\gamma_{j\ell}$, $\gamma^1_{j\ell}$ and $\gamma^2_{j\ell}$
in the $(j,k,\ell)$-triad is by at most $3\varepsilon |V_{jk}||V_{k\ell}|$ larger than
\[\sum_{A\kappa\kappa'\to B\lambda\lambda'}|L_{A\kappa\kappa'}|\cdot|R_{B\lambda\lambda'}|\:,\quad\quad
  \sum_{\kappa A\kappa'\to \lambda B\lambda'}|L_{\kappa A\kappa'}|\cdot|R_{\lambda B\lambda'}|\quad\quad \mbox{and} \quad\quad 
  \sum_{\kappa\kappa'A\to \lambda\lambda'B}|L_{\kappa\kappa'B}|\cdot|R_{\lambda\lambda'B}|,\]
respectively, where $\to$ denotes the relation used to define the problem \optlink{} and
the sums range over all choices of $\kappa,\kappa',\lambda,\lambda'\in\{A,B,X\}$ satisfying the relation $\to$.
Indeed, there are no edges containing a vertex in both $L_{\kappa\kappa'\kappa''}$ and $R_{\lambda\lambda'\lambda''}$
whenever $(\kappa,\lambda),(\kappa',\lambda')$ or $(\kappa'',\lambda'')$ is one of the pairs $(A,A)$, $(B,B)$ or $(B,A)$;
this follows from the first and second parts of the fourth property and the fifth property of $H_5$.
We next show that the edges that contain both a vertex in $L_{\kappa\kappa'A}$ and a vertex in $R_{\lambda B\lambda''}$
contribute at most $3\varepsilon |V_{jk}||V_{k\ell}|$ to each of the degrees.
Let us consider an edge of the $(j,k,\ell)$-triad of $H''$ that
contains a left neighbor $\alpha$ of $\gamma^2_{j\ell}$ and a right neighbor $\beta$ of $\gamma^1_{j\ell}$.
Since an edge with these properties is not present in $H'$,
one of the following had to have happened:
the edge does not exist in $H'$,
$\alpha$ is not a left of $\gamma^2_{j\ell}$ in $H'$, or
$\beta$ is not a right neighbor of $\gamma^1_{j\ell}$ in $H'$.
At most $\varepsilon^2|V_{jk}||V_{k\ell}|$ edges are present in $H''$ but not in $H'$ by Lemma~\ref{lm:gamma12ramsey}.
At most $\varepsilon|V_{jk}|$ vertices of $V_{jk}$ are left neighbors of $\gamma_{j\ell}^2$ in $H''$ but not in $H'$
since the codegree of $\gamma_{j\ell}^2$ and each left neighbor of $\gamma_{j\ell}^2$ in $H''$ is at least $\varepsilon$ and
the degree of $\gamma_{j\ell}^2$ in the $(j,k,\ell)$-triad in $H'$
is smaller by at most $\varepsilon^2$ compared to its degree in $H''$.
An analogous argument yields that
at most $\varepsilon|V_{k\ell}|$ vertices of $V_{k\ell}$ are right neighbors of $\gamma^1_{j\ell}$ in $H''$ but not in $H'$.
In total,
there are at most $(2\varepsilon+\varepsilon^2)|V_{jk}||V_{k\ell}| \leq 3\varepsilon |V_{jk}||V_{k\ell}|$ edges that
contain a left neighbor of $\gamma^2_{j\ell}$ and a right neighbor of $\gamma^1_{j\ell}$.

It follows that each of the degrees of $\gamma_{j\ell}$, $\gamma^1_{j\ell}$ and $\gamma^2_{j\ell}$ in $H''$
is larger by at most $3\varepsilon+32\cdot (2\varepsilon)=67\varepsilon$ (note that each sum has at most $32$ summands) than
\[\sum_{A\kappa\kappa'\to B\lambda\lambda'}d_{A\kappa\kappa'}d_{B\lambda\lambda'}\:,\quad\quad
  \sum_{\kappa A\kappa'\to \lambda B\lambda'}d_{\kappa A\kappa'}d_{\lambda B\lambda'}\quad\quad\mbox{and}\quad\quad
  \sum_{\kappa\kappa'A\to \lambda\lambda'B}d_{\kappa\kappa'A}d_{\lambda\lambda'B},\]
where the sums range over all choices of $\kappa,\kappa',\lambda,\lambda'\in\{A,B,X\}$ satisfying the relation $\to$.
Hence, the degree of every top vertex of the $(j,k,\ell)$-triad is at most
\[\left(\sum_{A\kappa\kappa'\to B\lambda\lambda'}d_{A\kappa\kappa'}d_{B\lambda\lambda'}\right)+80\varepsilon,\]
the degree of every top vertex contained in $T^1$ is at most
\[\left(\sum_{\kappa A\kappa'\to \lambda B\lambda'}d_{\kappa A\kappa'}d_{\lambda B\lambda'}\right)+80\varepsilon\]
and the degree of every top vertex contained in $T^2$ is at most
\[\left(\sum_{\kappa\kappa'A\to \lambda\lambda'B}d_{\kappa\kappa'B}d_{\lambda\lambda'B}\right)+80\varepsilon,\]
where all three sums range over all choices of $\kappa,\kappa',\lambda,\lambda'\in\{A,B,X\}$ satisfying the relation $\to$.
We obtain that
\begin{align*}
& \left(1-\frac{|T^1|+|T^2|}{|V_{j\ell}|}\right)
  \sum_{A\kappa_A\kappa_B\to B\lambda_A\lambda_B}d_{A\kappa_A\kappa_B}d_{B\lambda_A\lambda_B}+\\
& \frac{|T^1|}{|V_{j\ell}|}
  \min\left\{\sum_{A\kappa_A\kappa_B\to B\lambda_A\lambda_B}d_{A\kappa_A\kappa_B}d_{B\lambda_A\lambda_B},
             \sum_{\kappa A\kappa_B\to \lambda B\lambda_B}d_{\kappa A\kappa_B}d_{\lambda B\lambda_B}\right\}+\\
& \frac{|T^2|}{|V_{j\ell}|}
  \min\left\{\sum_{A\kappa_A\kappa_B\to B\lambda_A\lambda_B}d_{A\kappa_A\kappa_B}d_{B\lambda_A\lambda_B},
             \sum_{\kappa\kappa_AA\to \lambda\lambda_AB}d_{\kappa\kappa_AA}d_{\lambda\lambda_AB}\right\}
\end{align*}
is the density of the $(j,k,\ell)$-triad decreased by at most $80\varepsilon$,
i.e., the sum displayed above is at least $d-94\varepsilon$.
Using \eqref{eq:T12}, we obtain that the value of
\begin{align*}
& \left(\sum_{\kappa_A\kappa_B}d_{X\kappa_A\kappa_B}\right)
  \sum_{A\kappa_A\kappa_B\to B\lambda_A\lambda_B}d_{A\kappa_A\kappa_B}d_{B\lambda_A\lambda_B}+\\
& \left(\sum_{\kappa_A\kappa_B}d_{A\kappa_A\kappa_B}\right)
  \min\left\{\sum_{A\kappa_A\kappa_B\to B\lambda_A\lambda_B}d_{A\kappa_A\kappa_B}d_{B\lambda_A\lambda_B},
             \sum_{\kappa A\kappa_B\to \lambda B\lambda_B}d_{\kappa A\kappa_B}d_{\lambda B\lambda_B}\right\}+\\
& \left(\sum_{\kappa_A\kappa_B}d_{B\kappa_A\kappa_B}\right)
  \min\left\{\sum_{A\kappa_A\kappa_B\to B\lambda_A\lambda_B}d_{A\kappa_A\kappa_B}d_{B\lambda_A\lambda_B},
             \sum_{\kappa\kappa_AA\to \lambda\lambda_AB}d_{\kappa\kappa_AA}d_{\lambda\lambda_AB}\right\}
\end{align*}
is at least $d-130\varepsilon$ (in each of the three products forming the whole sum,
the left sum changes by at most $18\varepsilon$ and
the right sum is at most $(1+27\varepsilon)^2\le 2$).
Next note that the (non-negative) reals $d_{\kappa\kappa_A\kappa_B}$ do not need to sum to one,
i.e., they do not need to form a solution of the problem \optlink.
Let $S$ be the sum of the $27$ variables $d_{\kappa\kappa_A\kappa_B}$, and
observe that the sum $S$ is at least $1-27\varepsilon$ and at most $1+27\varepsilon$.
Hence, $x_{\kappa\kappa_A\kappa_B}=d_{\kappa\kappa_A\kappa_B}/S$ is a feasible solution for \optlink{} and
its value is at least (note that $d=4/27+\delta\le 5/27$ and $\varepsilon\le\delta/250$),
\[(d-130\varepsilon)/(1+27\varepsilon)^3\ge(d-130\varepsilon)(1-27\varepsilon)^3\ge d-(130+81+1)\varepsilon>4/27,\]
which is impossible by Theorem~\ref{thm:optim}.
The proof of the theorem is now complete.
\end{proof}

The next theorem differs from Theorem~\ref{thm:base} only by swapping the roles of $\gamma^1_{ij}$ and $\gamma^2_{ij}$
with respect to the vertices $\alpha$ and $\beta$ in the last case of its statement.
The proof is completely the same as that of Theorem~\ref{thm:base} except the following:
\begin{itemize}
\item We use Lemma~\ref{lm:gamma12ramsey} with roles of $\gamma^1_{ij}$ and $\gamma^2_{ij}$ swapped,
      which leads to the conclusion that the $(j,k,\ell)$-triad
      has no edge containing both a left neighbor of $\gamma^1_{j\ell}$ and a right neighbor of $\gamma^2_{j\ell}$ (rather
      than no edge containing both a left neighbor of $\gamma^2_{j\ell}$ and a right neighbor of $\gamma^1_{j\ell}$ as
      in the proof of Theorem~\ref{thm:base}).
\item When setting the problem \optlink, we swap the second and third coordinates, in particular,
      we eventually set $x_{\kappa\kappa_A\kappa_B}=d_{\kappa\kappa_B\kappa_A}/S$.
      This corresponds to the absence of edges containing a vertex of $L^1_A$ and a vertex of $R^2_B$ (while we have established
      the absence of edges containing a vertex of $L^2_A$ and a vertex of $R^1_B$ in the course of the proof of Theorem~\ref{thm:base}).
\end{itemize}

We omit the proof of the next theorem as the rest of the proof is literally the same as that of Theorem~\ref{thm:base}.

\begin{theorem}
\label{thm:base-swap}
Theorem \ref{thm:base} also holds with the last case replaced by the following.
\begin{itemize}
\item There exist vertices $\alpha^*_{ij}, \beta^*_{ij}, \gamma_{ij},\alpha^1_{ij},\beta^1_{ij},\gamma^1_{ij},\alpha^2_{ij},\beta^2_{ij},\gamma^2_{ij}\in V_{ij}$ for all $i<j$, $i,j\in I$, such that
      the following holds for all $i<j<k$, $i,j,k\in I$:
      \begin{itemize}
      \item $\{\gamma^1_{ij},\beta^*_{jk},\gamma_{ik}\}$ is an edge,
      \item $\{\alpha^*_{ij},\gamma^2_{jk},\gamma_{ik}\}$ is an edge,
      \item $\{\alpha^1_{ij},\beta^1_{jk},\gamma^1_{ik}\}$ is an edge, and
      \item $\{\alpha^2_{ij},\beta^2_{jk},\gamma^2_{ik}\}$ is an edge.
      \end{itemize}
      In addition,
      there exist $\alpha\in V_{i_{2n}i_{2n+1}}$ and $\beta\in V_{i_{2n+1}i_{2n+2}}$ such that
      the vertex $\alpha$ is a left neighbor of $\gamma^1_{i_{2n}i_{3n+1}}$,
      the vertex $\beta$ is a right neighbor of $\gamma^2_{i_ni_{2n+2}}$, and
      the $(i_{2n},i_{2n+1},i_{2n+2})$-triad contains an edge containing $\alpha$ and $\beta$.
\end{itemize}
\end{theorem}

We are now ready to prove the two main results of the paper.

\begin{theorem}
\label{thm:embed-2mod}
Let $m\in\NN$.
For every $\delta>0$, there exists $N\in\NN$ such that
every $N$-partitioned hypergraph $H$ with density at least $4/27+\delta$
embeds the tight $3$-uniform cycle $C_{3m+2}^{(3)}$.
\end{theorem}

\begin{figure}[t]
\begin{center}
\epsfbox{turan427-7.mps}
\vskip 8mm
\epsfbox{turan427-1.mps}
\vskip 8mm
\epsfbox{turan427-2.mps}
\vskip 8mm
\epsfbox{turan427-3.mps}
\end{center}
\caption{Embeddings of $C_5^{(3)}$ and $C_8^{(3)}$ in the four cases from Theorem~\ref{thm:embed-2mod}.
         The points correspond to the indices and the arcs to the parts of a partitioned hypergraph;
	 the types of vertices are indicated by Greek letters as in the notation of Theorem~\ref{thm:base}.}
\label{fig:embed-2mod}
\end{figure}

\begin{proof}
We apply Theorem~\ref{thm:base} with $\delta$ and $n=2m$ to get $N$.
Let $H$ be an $N$-partitioned hypergraph with density at least $4/27+\delta$.
Apply Theorem~\ref{thm:base} to get a $4n$-partitioned hypergraph $H_0$ with index set $I=\{i_1,\ldots,i_{4n}\}$ that
has the properties described in the statement of Theorem~\ref{thm:base}.
To establish the theorem,
it is enough to show that $H_0$ embeds the tight $3$-uniform cycle $C_\ell^{(3)}$,
i.e., that
there exist distinct indices $a_1,\ldots,a_{3m+2}\in I$ and vertices $v_k\in V_{a_ka_{k+1}}$ such that
for every $k\in [3m+2]$,
the vertices $v_k$ and $v_{k+1}$ are contained in an edge of the triad formed
by the parts $V_{i_{a_k}i_{a_{k+1}}}$, $V_{i_{a_{k+1}}i_{a_{k+2}}}$ and $V_{i_{a_k}i_{a_{k+2}}}$ (the subscripts are modulo $3m+2$).
To avoid using a large number of double indices in what follows,
we will deviate from our standard notation and
write $V_{ij}$, $i<j$, $i,j\in [4n]$, for the parts of the $4n$-partitioned hypergraph $H_0$ rather than the parts of $H$.

We distinguish four cases based on which of the four cases described in the statement of Theorem~\ref{thm:base} applies;
we also refer to Figure~\ref{fig:embed-2mod} for illustration of constructions used in the cases.

We start with assuming that the first case applies and let
$\alpha_{ij},\beta_{ij},\gamma_{ij}^0,\omega_{ij},\alpha^{\omega}_{ij},\gamma^{\omega}_{ij}\in V_{ij}$, $i<j$, $i,j\in [4n]$
be the vertices from the description of the case.
We choose the indices $a_1,a_2,\ldots,a_{3m+2}$ as
\[ 1,\: m+2,\: 2m+3,\: 2,\: m+3,\: 2m+4,\: \ldots,\: m,\: 2m+1,\: 3m+2,\: m+1,\: 2m+2 \]
and the vertices $v_1,\ldots,v_{3m+2}$ are chosen as
\begin{align*}
& \alpha_{1,m+2},\, \beta_{m+2,2m+3},\, \gamma^0_{2,2m+3},\,
  \alpha_{2,m+3},\, \beta_{m+3,2m+4},\, \gamma^0_{3,2m+4},\, \ldots,
  \alpha_{m,2m+1},\, \beta_{2m+1,3m+2}, \gamma^{\omega}_{m+1,3m+2}, \\
&  \omega_{m+1,2m+2}, \gamma^0_{1,2m+2}.
\end{align*}
to get a witness that $H_0$ embeds $C_{3m+2}^{(3)}$.
Indeed, it holds that
\begin{itemize}
    \item $\alpha_{i,m+1+i}$ and $\beta_{m+1+i,2(m+1)+i}$ belong (with $\gamma^0_{i,2(m+1)+i}$) to an edge of the $(i,m+1+i, 2(m+1)+i)$-triad for all $i=1,\ldots,m$,
    \item $\beta_{m+1+i,2(m+1)+i}$ and $\gamma^0_{i+1,2(m+1)+i}$ belong (with $\alpha_{i+1,m+1+i}$) to an edge of the $(i+1,m+1+i, 2(m+1)+i)$-triad for all $i=1,\ldots,m-1$,
    \item  $\gamma^0_{i+1,2(m+1)+i}$ and $\alpha_{i+1,m+1+i+1}$ belong (with $\beta_{m+1+i+1,2(m+1)+i}$) to an edge of the $(i+1,m+1+i+1, 2(m+1)+i)$-triad for for all $i=1,\ldots,m-1$,
    \item $\beta_{2m+1,3m+2}$ and $\gamma^{\omega}_{m+1,3m+2}$ belong (with $\omega_{m+1,2m+1}$) to an edge of the $(m+1,2m+1, 3m+2)$-triad,
    \item $\gamma^{\omega}_{m+1,3m+2}$and $\omega_{m+1,2m+2}$ belong (with $\beta_{2m+2,3m+2}$) to an edge of the $(m+1,2m+2, 3m+2)$-triad,
    \item $\omega_{m+1,2m+2}$ and $\gamma^0_{1,2m+2}$ belong (with $\alpha^{\omega}_{1,m+1}$) to an edge of the $(1,m+1,2m+2)$-triad, and
    \item $\gamma^0_{1,2m+2}$ and $\alpha_{1,m+2}$ belong (with $\beta_{m+2,2m+2}$) to an edge of the $(1,m+2,2m+2)$-triad.
\end{itemize} 

The second case is symmetric to the first case.
let $\alpha_{ij},\beta_{ij},\gamma_{ij}^0,\omega_{ij},\beta^{\omega}_{ij},\gamma^{\omega}_{ij}\in V_{ij}$, $i<j$, $i,j\in [4n]$,
be the vertices from the description of the case.
We choose the indices $a_1,a_2,\ldots,a_{3m+2}$ as
\[ 1,\: m+2,\: 2m+3,\: 2,\: m+3,\: 2m+4,\: \ldots,\: m,\: 2m+1,\: 3m+2,\: m+1,\: 2m+2 \]
and the vertices $v_1,\ldots,v_{3m+2}$ as
\begin{align*}
& \alpha_{1,m+2},\, \beta_{m+2,2m+3},\, \gamma_{2,2m+3}^0,\,
  \alpha_{2,m+3},\, \beta_{m+3,2m+4},\, \gamma_{3,2m+4}^0,\, \ldots,
  \alpha_{m,2m+1},\, \beta_{2m+1,3m+2}, \gamma_{m+1,3m+2}^0,\\ 
& \omega_{m+1,2m+2}, \gamma^{\omega}_{1,2m+2}
\end{align*}
to get a witness that $H_0$ embeds $C_{3m+2}^{(3)}$.

We next assume that the third case applies and
let $\alpha_{ij},\beta_{ij},\alpha'_{ij},\beta'_{ij},\gamma_{ij}^0,\gamma^{\bullet}_{ij}\in V_{ij}$, $i<j, i,j\in [4n]$, and
$\beta^{\bullet}_{ijkst}\in V_{jk}$ and $\alpha^{\bullet}_{ijkst}\in V_{ks}$, $i<j<k<s<t, i,j,k,s,t,\in [4n]$,
be the vertices from the description of the case.
We choose the indices $a_1,a_2,\ldots,a_{3m+2}$ as
\[ 1,\: m+1,\: 2m+1,\: 2,\: m+2,\: 2m+2,\: \ldots,\: m,\: 2m,\: 3m,\: 3m+1,\: 3m+2 \]
and the vertices $v_1,\ldots,v_{3m+2}$ as
\begin{align*}
& \alpha_{1,m+1},\, \beta_{m+1,2m+1},\, \gamma^0_{2,2m+1},\, 
  \alpha_{2,m+2},\, \beta_{m+2,2m+2},\, \gamma^0_{3,2m+2},\, \ldots,
                                        \gamma^0_{m,3m-1},\,
  \alpha'_{m,2m},\\
& \beta^{\bullet}_{m,2m,3m,3m+1,3m+2},\,
  \alpha^{\bullet}_{m,2m,3m,3m+1,3m+2},\, \beta'_{3m+1,3m+2},\, \gamma^0_{1,3m+2}
\end{align*}
to get a witness that $H_0$ embeds $C_{3m+2}^{(3)}$.

Finally, we deal with the fourth case.
Let $\alpha^*_{ij}$, $\beta^*_{ij}$, $\gamma_{ij}$, $\alpha^1_{ij}$, $\beta^1_{ij}$, $\gamma^1_{ij}$, $\alpha^2_{ij}$, $\beta^2_{ij}$, $\gamma^2_{ij}\in V_{ij}$, $i<j, i,j\in [4n]$, and
$\alpha\in V_{i_{2n},i_{2n+1}}$ and $\beta\in V_{i_{2n+1},i_{2n+2}}$ be the vertices as
in the description of the fourth case of Theorem~\ref{thm:base}.
In this case, we choose $a_1,a_2,\ldots,a_{3m+2}$ as
\begin{align*}
& n,\: 2n+2,\: 2n+1,\: 2n,\: 3n+1,\: 3n,\: 3n-1,\: 3n+4,\: 3n+3,\: 3n+2,\: 3n+7,\: \ldots,\: 3(n+m-1)+1
\end{align*}
and the vertices $v_1,\ldots,v_{3m+2}$ as
\begin{align*}
& \gamma^1_{n,2n+2},\,\beta_{2n+1,2n+2},\,\alpha_{2n,2n+1},\,\gamma^2_{2n,3n+1},\\
& \beta^2_{3n,3n+1},\, \alpha^2_{3n-1,3n},\, \gamma^2_{3n-1,3n+4},\, \ldots, \\
& \beta^2_{3(n+m-1)-3,3(n+m-1)-2},\, \alpha^2_{3(n+m-1)-4,3(n+m-1)-3},\, \gamma^2_{3(n+m-1)-4,3(n+m-1)+1},\\
& \gamma_{n,3(n+m-1)+1}.
\end{align*}
In particular, if $m=1$, the vertices $v_1,\ldots,v_5$ are
\[\gamma^1_{n,2n+2},\,\beta_{2n+1,2n+2},\,\alpha_{2n,2n+1},\,\gamma^2_{2n,3n+1},\,\gamma_{n,3n+1}.\]
The proof of the theorem is now complete.
\end{proof}

\begin{theorem}
\label{thm:embed-1mod}
Let $m\in\NN$.
For every $\delta>0$, there exists $N\in\NN$ such that
every $N$-partitioned hypergraph $H$ with density at least $4/27+\delta$
embeds the tight $3$-uniform cycle $C_{3m+4}^{(3)}$.
\end{theorem}

\begin{figure}
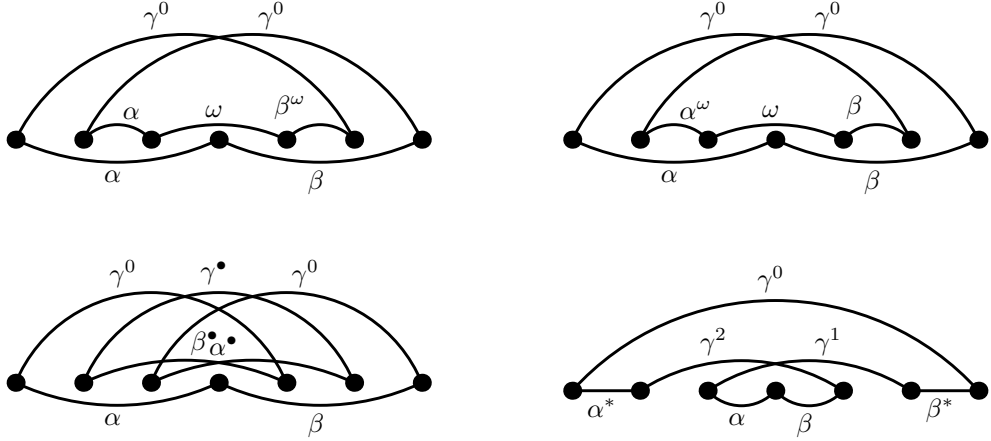

\begin{center}
\epsfbox{turan427-4.mps}
\hspace{1.5cm}
\epsfbox{turan427-8.mps}
\vskip 8mm
\epsfbox{turan427-5.mps}
\hspace{1.5cm}
\epsfbox{turan427-6.mps}
\end{center}
\caption{Embeddings of $C_7^{(3)}$ in the four cases from Theorem~\ref{thm:embed-1mod}.
         The points correspond to the indices and the arcs to the parts of a partitioned hypergraph;
	 the types of vertices are indicated by Greek letters as in the notation of Theorem~\ref{thm:base-swap}.}
\label{fig:embed-1mod}
\end{figure}

\begin{proof}
The proof follows along the same lines as the proof of Theorem~\ref{thm:embed-2mod}
but we need to choose the indices $a_1,\ldots,a_{3m+4}$ and vertices $v_1,\ldots,v_{3m+4}$ differently.
Let $N$ be the integer obtained by applying Theorem~\ref{thm:base-swap} with $\delta$ and $n=2m$.
Consider an $N$-partitioned hypergraph $H$ with with density at least $4/27+\delta$ and
let $H_0$ be the $4n$-partitioned hypergraph with index set $I=\{i_1,\ldots,i_{4n}\}$
as in the statement of Theorem~\ref{thm:base-swap}.
We will next find distinct indices $a_1,\ldots,a_{3m+4}\in I$ and vertices $v_k\in V_{i_{a_k}i_{a_{k+1}}}$, $k\in [3m+4]$, such that
the vertices $v_k$ and $v_{k+1}$ are contained in an edge of the triad formed
by the parts $V_{i_{a_k}i_{a_{k+1}}}$, $V_{i_{a_{k+1}}i_{a_{k+2}}}$ and $V_{i_{a_k}i_{a_{k+2}}}$ (the subscripts are modulo $3m+4$).
As in the proof of Theorem~\ref{thm:embed-2mod},
we deviate from our standard notation and
write $V_{ij}$, $i<j$, $i,j\in [4n]$, for the parts of the $4n$-partitioned hypergraph $H_0$.

We next distinguish four cases based on which of the cases described in the statement of Theorem~\ref{thm:base-swap} applies;
the illustration of the used constructions is given in Figure~\ref{fig:embed-1mod}.
We start with assuming that the first case applies and
let $\alpha_{ij},\beta_{ij},\gamma_{ij}^0,\omega_{ij},\beta^{\omega}_{ij},\gamma^{\omega}_{ij}\in V_{ij}$, $i<j$, $i,j\in [4n]$, be the vertices from the description of the case.
We choose the indices $a_1,a_2,\ldots,a_{3m+4}$ as
\begin{align*}
& 2m+2, 3m+4, 2, 3, 3m+2, 3m+3, 1,\\
& m+3, 2m+3, 4, m+4, 2m+4, 5, \ldots, m+1, 2m+1, 3m+1, m+2
\end{align*}
where the last $m-1$ indices are formed by triples $m+2+i,2m+2+i, 3+i$ for $i\in [m-1]$.
We next choose the vertices $v_1,\ldots,v_{3m+4}$ as
\begin{align*}
& \beta_{2m+2,3m+4},\, \gamma^0_{2,3m+4},\, \alpha_{2,3},\, \omega_{3,3m+2},\, \beta^{\omega}_{3m+2,3m+3},\, \gamma^0_{1,3m+3},\, \alpha_{1,m+3},\, \\
& \beta_{m+3,2m+3},\, \gamma^0_{4,2m+3},\, \alpha_{4,m+4},\, \beta_{m+4,2m+4},\, \gamma^0_{5,2m+4},\, \ldots, \alpha_{m+2,2m+2}
\end{align*}
to get a witness that $H_0$ embeds $C_{3m+4}^{(3)}$.
In particular, if $m=1$, the vertices $v_1,\ldots,v_7$ are
$\beta_{4,7}$, $\gamma^0_{2,7}$, $\alpha_{2,3}$, $\omega_{3,5}$, $\beta^{\omega}_{5,6}$, $\gamma^0_{1,6}$ and $\alpha_{1,4}$.

If the second case applies,
let $\alpha_{ij},\beta_{ij},\gamma^0_{ij},\omega_{ij},\alpha^{\omega}_{ij},\gamma^{\omega}_{ij}\in V_{ij}$, $i<j$, $i,j\in [4n]$
be the vertices from the description of the case.
The indices $a_1,a_2,\ldots,a_{3m+4}$ are chosen the same as before so
\begin{align*}
& 2m+2, 3m+4, 2, 3, 3m+2, 3m+3, 1, \\
& m+3, 2m+3, 4, m+4, 2m+4, 5, \ldots, m+1, 2m+1, 3m+1, m+2
\end{align*}
but the vertices $v_1,\ldots,v_{3m+4}$ are chosen as
\begin{align*}
& \beta_{2m+2,3m+4},\, \gamma^0_{2,3m+4},\, \alpha^{\omega}_{2,3},\, \omega_{3,3m+2},\, \beta_{3m+2,3m+3},\, \gamma_{1,3m+3},\, \alpha_{1,m+3},\, \\
& \beta_{m+3,2m+3},\, \gamma^0_{4,2m+3},\, \alpha_{4,m+4},\, \beta_{m+4,2m+4},\, \gamma^0_{5,2m+4},\, \ldots, \alpha_{m+2,2m+2}
\end{align*}
to get a witness that $H_0$ embeds $C_{3m+4}^{(3)}$.

If the third case applies,
let $\alpha_{ij},\beta_{ij},\alpha'_{ij},\beta'_{ij},\gamma_{ij}^0,\gamma^{\bullet}_{ij}\in V_{ij}$, $i,j\in [4n]$, and
$\beta^{\bullet}_{ijkst}\in V_{jk}$ and $\alpha^{\bullet}_{ijkst}\in V_{ks}$, $i,j,k,s,t,\in [4n]$,
be the vertices from the description of the case, and
choose $a_1,a_2,\ldots,a_{3m+4}$ as
\begin{align*}
& 2m+2, 3m+4, 3, 3m+3, 2, 3m+2, 1, \\
& m+3, 2m+3, 4, m+4, 2m+4, \ldots, m+1, 2m+1, 3m+1, m+2
\end{align*}
where compared to the two cases earlier, only the first seven indices differ.
The vertices $v_1,\ldots,v_{3m+4}$ are next chosen as
\begin{align*}
& \beta_{2m+2,3m+4},\, \gamma^0_{3,3m+4},\, \alpha^{\bullet}_{1,2,3,3m+3,3m+4},\, \gamma^{\bullet}_{2,3m+3},\, \beta^{\bullet}_{1,2,3m+2,3m+3,3m+4},\, \gamma^0_{1,3m+2},\, \alpha_{1,m+3},\, \\
& \beta_{m+3,2m+3},\, \gamma^0_{4,2m+3},\, \alpha_{4,m+4},\, \beta_{m+4,2m+4},\, \gamma^0_{5,2m+4},\, \ldots, \alpha_{m+2,2m+2}
\end{align*}
to get a witness that $H_0$ embeds $C_{3m+4}^{(3)}$.
In particular, if $m=1$, the cycle is formed by the vertices
$\alpha_{1,4}$, $\beta_{4,7}$, $\gamma^0_{3,7}$, $\alpha^{\bullet}_{1,2,3,6,7}$, $\gamma^{\bullet}_{2,6}$, $\beta^{\bullet}_{1,2,5,6,7}$, $\gamma^0_{1,5}$.

Finally, if the last case of Theorem~\ref{thm:base-swap} applies, we proceed as follows.
Let $\alpha^*_{ij}$, $\beta^*_{ij}$, $\gamma_{ij}$, $\alpha^1_{ij}$, $\beta^1_{ij}$, $\gamma^1_{ij}$, $\alpha^2_{ij}$, $\beta^2_{ij}$, $\gamma^2_{ij}\in V_{ij}$, $i,j\in [4n]$, and
$\alpha\in V_{i_{2n}i_{2n+1}}$ and $\beta\in V_{i_{2n+1}i_{2n+2}}$ be the vertices as in the description of the case.
We choose the indices $a_1,a_2,\ldots,a_{3m+4}$ as
\begin{align*}
& n-1, n, 2n+2, 2n+1, 2n, 3n+1,\\
& 3n, 3n-1, 3n+4, 3n+3, 3n+2, 3n+7, \ldots, 3(n+m-1)-4, 3(n+m-1)+1,\\
& 3(n+m-1)+2
\end{align*}
and the vertices $v_1,\ldots,v_{3m+4}$ as
\begin{align*}
& \alpha^*_{n-1,n},\,\gamma^2_{n,2n+2},\,\beta_{2n+1,2n+2},\,\alpha_{2n,2n+1},\,\gamma^1_{2n,3n+1},\\
& \beta^1_{3n,3n+1},\, \alpha^1_{3n-1,3n},\, \gamma^1_{3n-1,3n+4},\, \ldots, \\
& \beta^1_{3(n+m-1)-3,3(n+m-1)-2},\, \alpha^1_{3(n+m-1)-4,3(n+m-1)-3},\, \gamma^1_{3(n+m-1)-4,3(n+m-1)+1},\\
& \beta^*_{3(n+m-1)+1,3(n+m-1)+2},\gamma_{n-1,3(n+m-1)+2}.
\end{align*}
to get a witness that $H_0$ embeds $C_{3m+4}^{(3)}$.
In particular, if $m=1$, the vertices $v_1,\ldots,v_7$ are
\[\alpha^*_{n-1,n},\,\gamma^2_{n,2n+2},\,\beta_{2n+1,2n+2},\,\alpha_{2n,2n+1},\,\gamma^1_{2n,3n+1},\,\beta^*_{3n+1,3n+2},\,\gamma_{n-1,3n+2}.\]
The proof of the theorem is now complete.
\end{proof}

Theorems~\ref{thm:embed-2mod} and~\ref{thm:embed-1mod} yield the following.

\begin{theorem}
\label{thm:main}
Let $\ell\ge 5$ be a positive integer that is not divisible by three.
The uniform Tur\'an density of the tight $3$-uniform cycle $C_\ell^{(3)}$ is equal to $4/27$.
\end{theorem}

\begin{proof}
Fix $\ell\ge 5$ that is not divisible by three.
The uniform Tur\'an density of $C_\ell^{(3)}$
is at least $4/27$ by Proposition~\ref{prop:lower}.
If $\ell$ is equal to $2$ modulo $3$,
then the uniform Tur\'an density of $C_\ell^{(3)}$ is at most $4/27$
by Theorem~\ref{thm:reiher} and Theorem~\ref{thm:embed-2mod};
otherwise,
the uniform Tur\'an density of $C_\ell^{(3)}$ is at most $4/27$
by Theorem~\ref{thm:reiher} and Theorem~\ref{thm:embed-1mod}.
\end{proof}

Theorem~\ref{thm:cycles-main} follows from Theorem~\ref{thm:main} and Proposition~\ref{prop:cycle3}.

\section*{Acknowledgments}

The authors would like to thank Ander Lamaison Vidarte for carefully reading and commenting on an early draft of the paper.

\bibliographystyle{bibstyle}
\bibliography{turan427}

\begin{thebibliography}{10}
\providecommand{\url}[1]{\texttt{#1}}
\providecommand{\urlprefix}{URL }
\providecommand{\eprint}[2][]{\url{#2}}

\bibitem{BabT11}
R.~Baber and J.~Talbot: \emph{Hypergraphs do jump}, Combinatorics, Probability
  and Computing \textbf{20} (2011), 161--171.

\bibitem{BalCL21}
J.~Balogh, F.~C. Clemen and B.~Lidick\'y: \emph{Hypergraph tur\'an problems in
  $\ell_2$-norm} (2021), preprint arXiv:2108.10406.

\bibitem{ChuL99}
F.~Chung and L.~Lu: \emph{An upper bound for the {T}ur\'an number $t_3(n,4)$},
  Journal of Combinatorial Theory Series A \textbf{87} (1999), 381--389.

\bibitem{Con10}
D.~Conlon: \emph{An extremal theorem in the hypercube}, Electron. J. Combin.
  \textbf{17} (2010), R111, 7pp.

\bibitem{Erd64}
P.~Erd{\H o}s: \emph{On extremal problems of graphs and generalized graphs},
  Israel Journal of Mathematics \textbf{2} (1964), 183--190.

\bibitem{Erd81}
P.~Erd{\H o}s: \emph{On the combinatorial problems which i would most like to
  see solved}, Combinatorica \textbf{1} (1981), 25--42.

\bibitem{Erd90}
P.~Erd{\H o}s: \emph{Problems and results on graphs and hypergraphs:
  similarities and differences}, in: J.~Ne\v{s}et\v{r}il and V.~R\"odl (eds.),
  Mathematics of {R}amsey theory (1990), 223--233.

\bibitem{ErdS66}
P.~Erd{\H{o}}s and M.~Simonovits: \emph{A limit theorem in graph theory},
  Studia Sci. Math. Hungar. \textbf{1} (1966), 51--57.

\bibitem{ErdS82}
P.~Erd{\H{o}}s and V.~T. S{\'o}s: \emph{On {R}amsey-{T}ur{\'a}n type theorems
  for hypergraphs}, Combinatorica \textbf{2} (1982), 289--295.

\bibitem{ErdS46}
P.~Erd{\H o}s and A.~H. Stone: \emph{On the structure of linear graphs},
  Bulletin of the American Mathematical Society \textbf{52} (1946), 1087--1091.

\bibitem{FraF84}
P.~Frankl and Z.~F{\"u}redi: \emph{An exact result for 3-graphs}, Discrete
  Mathematics \textbf{50} (1984), 323--328.

\bibitem{GarKL21}
F.~Garbe, D.~Kr{\'a}l' and A.~Lamaison: \emph{Hypergraphs with minimum positive
  uniform {T}ur\'an density} (2021), preprint arXiv:2105.09883.

\bibitem{GleKV16}
R.~Glebov, J.~Volec and D.~Kr{\'a}l': \emph{A problem of {E}rd{\H{o}}s and
  {S}{\'o}s on 3-graphs}, Israel Journal of Mathematics \textbf{211} (2016),
  349--366.

\bibitem{HaxLPRRS09}
P.~E. Haxell, T.~\L{}uczak, Y.~Peng, V.~R\"odl, A.~Ruci\'nski and J.~Skokan:
  \emph{The {R}amsey number for 3-uniform tight hypergraph cycles},
  Combinatorics, Probability and Computing \textbf{18} (2009), 165--203.

\bibitem{HuaM19}
H.~Huang and J.~Ma: \emph{On tight cycles in hypergraphs}, SIAM J. Discrete
  Math. \textbf{33} (2019), 230--237.

\bibitem{Jan21}
B.~Janzer: \emph{Large hypergraphs without tight cycles} (2020), preprint
  arXiv:2012.07726.

\bibitem{KatNS64}
G.~Katona, T.~Nemetz and M.~Simonovits: \emph{On a problem of {T}ur{\'a}n in
  the theory of graphs}, Mat. Lapok \textbf{15} (1964), 228--238.

\bibitem{Kee11}
P.~Keevash: \emph{Hypergraph {T}ur{\'a}n problems}, in: R.~Chapman (ed.),
  Surveys in Combinatorics 2011, London Mathematical Society Lecture Note
  Series (2011), 83--140.

\bibitem{Let21}
S.~Letzter: \emph{Hypergraphs with no tight cycles} (2021), preprint
  arXiv:2106.12082.

\bibitem{Man07}
W.~Mantel: \emph{Problem 28}, Wiskundige Opgaven \textbf{10} (1907), 60--61.

\bibitem{MubPS}
D.~Mubayi, O.~Pikhurko and B.~Sudakov: \emph{Hypergraph {T}ur\'an {P}roblem:
  Some open questions} (2011),
  \url{https://homepages.warwick.ac.uk/\~{}maskat/Papers/TuranQuestions.pdf}.

\bibitem{MubR02}
D.~Mubayi and V.~R\"{o}dl: \emph{On the {T}ur\'{a}n number of triple systems},
  J. Combin. Theory Ser. A \textbf{100} (2002), 136--152.

\bibitem{Raz07}
A.~A. Razborov: \emph{Flag algebras}, J. Symbolic Logic \textbf{72} (2007),
  1239--1282.

\bibitem{Raz10}
A.~A. Razborov: \emph{On 3-hypergraphs with forbidden 4-vertex configurations},
  SIAM Journal on Discrete Mathematics \textbf{24} (2010), 946--963.

\bibitem{Rei20}
C.~Reiher: \emph{Extremal problems in uniformly dense hypergraphs}, European
  Journal of Combinatorics \textbf{88} (2020), 103117.

\bibitem{ReiRS16}
C.~Reiher, V.~R{\"o}dl and M.~Schacht: \emph{Embedding tetrahedra into
  quasirandom hypergraphs}, Journal of Combinatorial Theory Series B
  \textbf{121} (2016), 229--247.

\bibitem{ReiRS18}
C.~Reiher, V.~R{\"o}dl and M.~Schacht: \emph{Hypergraphs with vanishing
  {T}ur{\'a}n density in uniformly dense hypergraphs}, Journal of the London
  Mathematical Society \textbf{97} (2018), 77--97.

\bibitem{ReiRS18c}
C.~Reiher, V.~R{\"o}dl and M.~Schacht: \emph{On a generalisation of {M}antel's
  theorem to uniformly dense hypergraphs}, International Mathematics Research
  Notices \textbf{16} (2018), 4899--4941.

\bibitem{ReiRS18a}
C.~Reiher, V.~R{\"o}dl and M.~Schacht: \emph{On a {T}ur\'an problem in weakly
  quasirandom 3-uniform hypergraphs}, Journal of the European Mathematical
  Society \textbf{20} (2018), 1139--1159.

\bibitem{ReiRS18b}
C.~Reiher, V.~R{\"o}dl and M.~Schacht: \emph{Some remarks on $\pi$}, in:
  S.~Butler, J.~Cooper and G.~Hurlbert (eds.), Connections in Discrete
  Mathematics: A Celebration of the Work of Ron Graham (2018), 214--239.

\bibitem{Rod86}
V.~R{\"o}dl: \emph{On universality of graphs with uniformly distributed edges},
  Discrete Mathematics \textbf{59} (1986), 125--134.

\bibitem{RRS11}
V.~R\"odl, A.~Ruci\'nski and E.~Szemer\'edi: \emph{Dirac-type conditions for
  hamiltonian paths and cycles in 3-uniform hypergraphs}, Adv. Math.
  \textbf{227} (2011), 1225--1299.

\bibitem{Sid95}
A.~Sidorenko: \emph{What we know and what we do not know about {T}ur\'an
  numbers}, Graphs and Combinatorics \textbf{11} (1995), 179--199.

\bibitem{TomS21}
I.~Tomon and B.~Sudakov: \emph{The extremal number of tight cycles}, to appear
  in International Mathematics Research Notices.

\bibitem{Tur41}
P.~Tur{\'a}n: \emph{Eine {E}xtremalaufgabe aus der {G}raphentheorie}, Mat. Fiz.
  Lapok \textbf{48} (1941), 61.

\bibitem{Ver16}
J.~Verstra\"{e}te: \emph{Extremal problems for cycles in graphs}, in: Recent
  trends in combinatorics, \emph{IMA Vol. Math. Appl.}, volume 159 (2016),
  83--116.

\end{thebibliography}

\end{document}